\documentclass{amsart}

\usepackage[english]{babel}     
\usepackage{amsmath}      
\usepackage{amssymb}     
\usepackage{mathrsfs} 
\usepackage{stmaryrd}
\usepackage{hyperref}
\usepackage[utf8]{inputenc}
\usepackage{xcolor}
\usepackage{todonotes}
\usepackage{tikz-cd}
\usetikzlibrary{patterns}
\usepackage{sseq}
\usepackage{colonequals}

\makeatletter
\def\oversortoftilde#1{\mathop{\vbox{\m@th\ialign{##\crcr\noalign{\kern3\p@}%
      \sortoftildefill\crcr\noalign{\kern3\p@\nointerlineskip}%
      $\hfil\displaystyle{#1}\hfil$\crcr}}}\limits}

\def\sortoftildefill{$\m@th \setbox\z@\hbox{$\braceld$}%
  \braceld\leaders\vrule \@height\ht\z@ \@depth\z@\hfill\braceru$}

\makeatother

\newtheorem{theorem}{Theorem}[section]
\newtheorem{lemma}[theorem]{Lemma}
\newtheorem{proposition}[theorem]{Proposition}
\newtheorem{corollary}[theorem]{Corollary}

\theoremstyle{definition}
\newtheorem{definition}[theorem]{Definition}
\newtheorem{example}[theorem]{Example}

\theoremstyle{remark}
\newtheorem{remark}[theorem]{Remark}

\numberwithin{equation}{section}

\newcommand{\att}{\big \vert}

\newcommand{\pr}			{{\operatorname{\mathsf{pr}}}}

\newcommand{\pol}		{{\operatorname{\mathrm{pol}}}}

\newcommand{\frakg}			{{\operatorname{\mathfrak{g}}}}
\newcommand{\frakh}			{{\operatorname{\mathfrak{h}}}}
\newcommand{\gl}			{{\operatorname{\mathfrak{gl}}}}

\newcommand{\dif}{\mathop{}\!\mathrm{d}}    

\DeclareMathOperator{\supp}{supp}
\DeclareMathOperator{\spn}{span}

\DeclareMathOperator{\sign}{sign}
\DeclareMathOperator{\id}{id}
\DeclareMathOperator{\Hom}{Hom}

\DeclareMathOperator{\Diff}{Diff}
\newcommand{\hotimes}	{\mathbin{\widehat{\otimes}}}

\DeclareFontFamily{U}{MnSymbolC}{}
\DeclareSymbolFont{MnSyC}{U}{MnSymbolC}{m}{n}
\DeclareFontShape{U}{MnSymbolC}{m}{n}{
    <-6>  MnSymbolC5
   <6-7>  MnSymbolC6
   <7-8>  MnSymbolC7
   <8-9>  MnSymbolC8
   <9-10> MnSymbolC9
  <10-12> MnSymbolC10
  <12->   MnSymbolC12}{}
\DeclareMathSymbol{\intprod}{\mathbin}{MnSyC}{'270}

\begin{document}

\title{A local-to-global analysis of Gelfand-Fuks cohomology}
\author{Lukas Miaskiwskyi}

\begin{abstract}
We present a novel proof technique to construct the Gelfand-Fuks spectral sequence for diagonal Chevalley-Eilenberg cohomology of vector fields on a smooth manifold, performing a local-to-global analysis through a notion of generalized good covers from the theory of factorization algebras and cosheaves.
This approach yields a unified way to deal with the problem of comparing ``sheaf-like'' data over different Cartesian powers of the manifold, and is easily generalized to the study of other cohomology theories associated to geometric objects.
Independently, we lay out a detailed and easily accessible exposition on the continuous Chevalley-Eilenberg cohomology of formal vector fields and of vector fields on Euclidean space, modernizing and elaborating on well-established literature on the subject by Bott, Fuks, and Gelfand.
\end{abstract}

\maketitle

\tableofcontents

\newpage

\section{Introduction}
Beginning around fifty years ago, a plethora of literature has been created to understand the continuous Chevalley-Eilenberg cohomology $\mathfrak{X}(M)$ of the vector fields on a smooth manifold $M$. This cohomology carries the name \emph{Gelfand--Fuks cohomology}, in reference to the authors who opened the investigation of this subject with a series of highly novel papers \cite{gel1969cohomologyI, gel1970cohomologyII, gel1970formalcohomology}. Initially, it was hoped that this cohomology might contain invariants for the smooth structure of $M$, hence be a potential tool for a classification of smooth structures on a given topological manifold. 
\\
Unfortunately, these hopes were denied by a paper by Bott and Segal, which showed that the Gelfand--Fuks cohomology was isomorphic to the singular cohomology of a mapping space that can be functorially constructed from $M$, and from which no new invariants arise \cite{bott1977cohomology}.
Regardless, these explorations brought with them a lot of applications, for example in the theory of foliations \cite{fuks1973characteristic} or for the construction of the Virasoro algebra \cite{virasoro1970subsidiary}. Further, many related open problems are still being pursued, like the continuous cohomology of the Lie algebra of symplectic, Hamiltonian or divergence-free vector fields on symplectic/Riemannian manifolds \cite{janssens2016universal, janssens2018integrability}. 
\\
The goal of this document is two-fold: For one, we want to present a novel proof technique, using so-called \emph{$k$-good covers} from the theory of functor calculus and factorization algebras (cf.~\cite{boavida2013manifold}), to construct well-known spectral sequences which calculate the Gelfand--Fuks cohomology of certain smooth manifolds, cf. \cite[Theorem 2.4.1a, 2.4.1.b]{fuks1984cohomology}. This approach is inspired by the treatment of Gelfand--Fuks cohomology in the framework of factorization algebras in a preprint by Kapranov and Hennion \cite{hennion2018gelfand}. The idea is to use a local-to-global analysis to reconstruct the spectral sequence for the Gelfand--Fuks cohomology of the manifold from the Gelfand--Fuks cohomology of a local part. The use of $k$-good covers solves the problem that this reconstruction necessarily compares data between different Cartesian powers $M, M^2, M^3, \dots$ of a manifold $M$. This is a natural problem to encounter here, since cochains for Gelfand--Fuks cohomology are maps on multiple copies of $\mathfrak{X}(M)$.
 An advantage of this approach over previous proofs is that it is easily generalizable to other situations, as we plan to show in future work \cite{miaskiwskyi2021loday}.
\\
The other goal of this document is to lay out a streamlined, detailed and relatively elementary path to the fundamental results of Gelfand--Fuks cohomology, accessible to any researcher with a solid understanding of homological algebra, sheaf theory and differential geometry. To this end, we largely follow the general strategies in \cite{fuks1984cohomology, bott1973lectures}, filling in nontrivial details that have been left to the reader in the original literature, modernizing some of the language used, and replacing some of the arguments with ones which the author perceives as clearer.
We make no a claim to exhaustiveness: we restrict ourselves to Gelfand--Fuks cohomology with trivial coefficients, and direct the reader to \cite{tsujishita1981continuous} for an overview of the study of other coefficient modules.

We begin in Section~\ref{SectionLieAlgebraOfFormals} with a study of continuous cohomology of the Lie algebra of formal vector fields, i.e. vector fields whose coefficient functions are formal power series. They represent the infinitesimal counterpart of $\mathfrak{X}(M)$ and their cohomology is calculated using a spectral sequence over which one can get full control. This section is largely a review of \cite{gel1970formalcohomology}.
In Section~\ref{SectionLocalCohomology}, we tie the cohomology of formal vector fields to the Gelfand--Fuks cohomology of Euclidean space, which may itself be understood as the local counterpart to Gelfand--Fuks cohomology. This section is a review of \cite{bott1973lectures}.
The Section~\ref{sec:CosheafAspectsOfLocalGelfandFuks}, we examine the transformation behavior of Gelfand--Fuks cohomology on Euclidean space under diffeomorphisms. The proofs and formulations are original, though the results are implicitly used in the literature.
This prepares the local-to-global analysis of the Gelfand--Fuks cohomology on an arbitrary smooth manifold, which we carry out in Section~\ref{SectionGelfandFuksOnSmooths}. We give a variation of the well-known spectral sequences that calculate Gelfand--Fuks cohomology for a class of orientable, smooth manifolds. For the sake of completion, we explain how it allows a full calculation of the Gelfand--Fuks cohomology of the circle $S^1$ and may be used to make certain general statements about finite-dimensionality of the Gelfand--Fuks cohomology.
\noindent {\it Acknowledgement.}  The author would like to thank Bas Janssens, Tobias Diez, and Milan Niestijl for interesting discussions and many useful suggestions. He was supported by the NWO Vidi Grant 639.032.734 ``Cohomology and representation theory of infinite-dimensional Lie groups''.   

\section{The Lie algebra of formal vector fields}
\label{SectionLieAlgebraOfFormals}
In this section, we mainly elaborate on the methods given in \cite[Chapter 2.2]{fuks1984cohomology} and \cite{gel1970formalcohomology} to analyze the Lie algebra of formal vector fields, an infinitesimal version of the Lie algebra of vector fields on a smooth manifold. We also use methods from \cite[Corollary 1]{guillemin1973some} to analyze stable Chevalley-Eilenberg cohomology of this space.
\subsection{Definition and first properties}
\begin{definition}[Formal vector fields]
Let~$n \in \mathbb{N}$. We define the \emph{Lie algebra of formal vector fields}~$W_n$ to be equal to the topological Lie algebra 
\begin{align}
W_n := \mathbb{R} \llbracket x_1,\dots,x_n \rrbracket \otimes \mathbb{R}^n.
\end{align}
Its topology is induced by the projective limit topology of~$\mathbb{R}\llbracket x_1,\dots,x_n \rrbracket$ and its Lie bracket is given by
\begin{align}
[f \partial_i, g \partial_j] := 
f \frac{\partial g}{\partial x_i} \cdot \partial_j
-
g \frac{\partial f}{\partial x_j} \cdot \partial_i, \quad f,g 
\in 
\mathbb{R} \llbracket x_1,\dots,x_n \rrbracket.
\end{align}
\end{definition}
\begin{remark}
\label{RemarkTransformationOfWn} 
There is a more geometric definition of~$W_n$ which we will use in Section~\ref{SectionLocalCohomology}, as the space of infinity-jets~$J_p^\infty \mathfrak{X}(\mathbb{R}^n)$ of vector fields at an arbitrary point~$p \in \mathbb{R}^n$. Any choice of local frame around~$p$ induces a continuous Lie algebra isomorphism 
\begin{align}
J_p^\infty \mathfrak{X}(\mathbb{R}^n) \cong W_n.
\end{align} 
\end{remark}
We first examine the structure of~$W_n$.
\begin{definition}
The element~$E := \sum_{i=1}^n x_i \partial_i \in \mathfrak{g}_0$ is the \emph{Euler vector field} in~$W_n$. The eigenspaces
\begin{align}
\mathfrak{g}_k = \{ X \in W_n : [E,X] = k \cdot X \}
\end{align}
give~$W_n = \widehat{\bigoplus}_{k \geq \mathbb{Z}} \mathfrak{g}_k$ as the completion of a graded Lie algebra.
Elements of~$\mathfrak{g}_k$ are called \emph{homogeneous (of degree~$k$)}. 
\end{definition}
More explicitly, we find for all~$k \in \mathbb{Z}$
\begin{align}
\mathfrak{g}_k = 
\left 
\{ \sum_{i=1}^n p_i \partial_i \in W_n 
:
 p_i \text{ homogeneous polynomials of degree } k+1 
 \right 
 \}.
\end{align}
In particular,~$\frakg_k = 0$ if~$k < -1$, and in low orders, we have Lie algebra isomorphisms:
\begin{align}
\mathfrak{g}_{-1} =  
\spn 
\left\{ \partial_i : i = 1,\dots, n \right \} \cong \mathbb{R}^n, 
\quad
\mathfrak{g}_{0} = 
\spn \left\{ x_i \partial_j : i, j = 1,\dots, n \right\} \cong \mathfrak{gl}_n(\mathbb{R}).
\end{align}
\begin{definition}
Let~$\frakg$ be a topological Lie algebra. Its \emph{continuous Chevalley-Eilenberg cohomology} is the cohomology of the cochain complex
\begin{align}
C^\bullet(\frakg) := \bigoplus_{k \geq 0} C^k(\frakg).
\end{align}
By~$C^k(\frakg)$ we denote the space of multilinear, skew-symmetric, continuous maps
\begin{equation}
c : \frakg^{k} \to \mathbb{R},
\end{equation}
and the differential~$d : C^\bullet(\frakg) \to C^{\bullet+1}(\frakg)$ of the complex is
\begin{equation}
\begin{gathered}
dc(X_1,\dots,X_{k+1}) = 
\sum_{1 \leq i < j \leq k+1} (-1)^{i+j-1} c([X_i,X_j],X_1,\dots,X_{k+1}) \quad \forall X_i \in \frakg.
\end{gathered}
\end{equation}
The space~$C^\bullet(\frakg)$ assumes the structure of a differential graded algebra with the wedge product
\begin{equation}
\begin{aligned}
(c_1 \wedge c_2) &(X_1,\dots,X_{k+l}) := 
\\
&\frac{1}{k! l!} \sum_{\sigma \in \Sigma_{k + l}} \sign(\sigma) c_1(X_{\sigma(1)} ,\dots, X_{\sigma(k)}) c_2(X_{\sigma(k + 1)}, \dots, X_{\sigma(k + l)})
\end{aligned}
\end{equation}
for~$c_1 \in C^k(\frakg), c_2 \in C^l(\frakg), X_1,\dots,X_{k+l} \in \frakg$.
\end{definition}
\begin{remark}
If~$\frakg$ is finite-dimensional, the continuity assumption for cochains in~$C^\bullet(\frakg)$ is redundant. If~$\frakg = W_n$ with its projective topology, then the continuity assumption on~$c \in C^\bullet(W_n)$ just means that there is a~$k \in \mathbb{Z}$ so that~$c(X,\cdot,\dots,\cdot) = 0$ for all~$X$ with~$\deg X > k$. In particular,~$c$ is only nonzero on a finite-dimensional subspace of~$\Lambda^\bullet W_n$.
\end{remark}
Recall the following:
\begin{definition}
\label{DefinitionLieAlgebraActionOnCochains}
Let~$\mathfrak{g}$ be a Lie algebra,~$Y \in \mathfrak{g}$ and~$c \in C^k(\mathfrak{g})$ for some~$k \geq 0$. 
\begin{itemize}
\item[i)]
Denote the natural Lie algebra action of an element~$Y$ on~$C^k(\mathfrak{g})$ by~$Y \cdot c$; the formula is given for~$Y,X_1,\dots,X_k \in \mathfrak{g}$ by
\begin{align} 
(Y \cdot c)(X_1,\dots,X_k) &:= -\sum_{i = 1}^k c(X_1,\dots,[Y,X_i],\dots,X_k),
\end{align}
and~$Y \cdot c = 0$ if~$c \in C^0(\frakg)$.
\item[ii)]
Denote by~$Y \intprod c \in C^{k-1}(\mathfrak{g})$ the \emph{interior product of~$c$ with~$Y$,} which is defined via
\begin{align}
(Y \intprod c)(X_1,\dots,X_{k-1}) = c(Y,X_1,\dots,X_{k-1})
\end{align}
and~$Y \intprod c = 0$ if~$c \in C^0(\frakg)$.
\end{itemize}
\end{definition}
A straightforward calculation yields the following homotopy relation:
\begin{lemma}
\label{LemmaInteriorProductHomotopy}
Let~$\mathfrak{g}$ be a Lie algebra,~$c \in C^\bullet(\frakg)$ and~$Y \in \mathfrak{g}$. Then we have the following chain homotopy formula:
\begin{align}
d ( Y \intprod c) + Y \intprod dc
= 
-Y \cdot c
\end{align}
\end{lemma}
%
%
A well-known corollary of the previous statement is:
\begin{corollary}
\label{CorollaryLieAlgebraActsTriviallyOnCohomology}
The action of a Lie algebra~$\mathfrak{g}$ on its cochains~$C^\bullet(\frakg)$ commutes with the Chevalley-Eilenberg differential, and the induced action on~$H^\bullet(\mathfrak{g})$ is trivial.
\end{corollary}
Using the grading of~$W_n$ induced by the Euler vector field~$E$, we can also define a grading of the cochains:
\begin{definition}
Let~$r \in \mathbb{Z}$ and~$k \geq 0$ an integer. Define
\begin{align}
C^k_{(r)}(W_n) := 
\{ c \in C^k(W_n) : E \cdot c = -r \cdot c \}.
\end{align}
\end{definition}
\begin{remark}
More explicitly,~$c \in C^k_{(r)}(W_n)$ if and only if we have for all homogeneous formal vector fields~$X_1,\dots,X_k \in W_n$
\begin{align}
\sum_{i=1}^k \deg X_i \neq r \implies c(X_1,\dots,X_k) = 0.
\end{align}
\end{remark}
\begin{proposition}[\cite{fuks1984cohomology}, Section 1.5 and 2.2]
The spaces~$C^\bullet_{(r)} (W_n)~$ fulfil the following properties:
\begin{itemize}
\item[i)] We have~$C^\bullet(W_n) = \bigoplus_{r \in \mathbb{Z}} C^\bullet_{(r)}(W_n).
$
\item[ii)] For all~$r \in \mathbb{Z}$, the spaces~$C^\bullet_{(r)}(W_n)$ are subcomplexes of~$C^\bullet(W_n)$.
\item[iii)] For all~$r,s \in \mathbb{Z}$ we have
\begin{align}
C^\bullet_{(r)}(W_n) \wedge C^\bullet_{(s)}(W_n) \subset C^\bullet_{(r + s)}(W_n).
\end{align}
\item[iv)] If~$r < -k$, then~$C^k_{(r)}(W_n) = 0$.
\item[v)]
\label{PropositionWnCohmologyInZeroethDegree} 
The inclusion~$C^\bullet_{(0)}(W_n) \subset C^\bullet(W_n)$ induces an algebra isomorphism
\begin{align}
H^\bullet(C^\bullet_{(0)}(W_n)) \cong H^\bullet(W_n).
\end{align}
\end{itemize}
\end{proposition}
\begin{proof} 
\textit{i)} The direct sum decomposition follows since every~$c \in C^k(W_n)$ is zero on homogeneous vector fields of sufficiently high degree, and its evaluation on any collection~$X_1,\dots,X_k \in W_n$ can be uniquely decomposed into summands of homogeneous vector fields.
\\
\textit{ii)}
This follows since the action of~$E$ commutes with the Lie algebra differential by Lem\-ma~\ref{LemmaInteriorProductHomotopy}.
\\
\textit{iii)} This is due to 
\begin{align}
E \cdot (c_1 \wedge c_2) = (E \cdot c_1) \wedge c_2 + c_1 \wedge (E \cdot c_2) \quad \forall c_1,c_2 \in C^\bullet(W_n).
\end{align}
\textit{iv)} Due to the pidgeonhole principle, any collection of~$k$ elements in~$W_n$ whose degrees sum up to a value smaller than~$-k$ must have an element with degree smaller~$ -1$. Such an element is necessarily zero, which shows the statement. 
\\
\textit{v)} Lemma~\ref{LemmaInteriorProductHomotopy} shows that for all~$c \in C^\bullet_{(r)}(W_n)$ we have
\begin{align}
d(E \intprod c) + E \intprod (dc) = - r \cdot c. 
\end{align}
Thus, for~$r \neq 0$, the map~$- \frac{1}{r}(E \intprod \cdot)$ defines a chain homotopy between the identity and zero for the cochain complex~$C^\bullet_{(r)}(W_n)$, and hence~$H^\bullet(C^\bullet_{(r)}(W_n)) = 0$. 
We conclude that all cohomology classes of~$C^\bullet(W_n)$ admit a representative fully contained in~$C^\bullet_{(0)}(W_n)$. This shows that the inclusion induces an isomorphism of vector spaces. By statement iii)~$C_{(0)}^\bullet(W_n)$ is a subalgebra of~$C^\bullet(W_n)$ with respect to the wedge product, hence the inclusion induces an algebra isomorphism on cohomology.
\end{proof}
In the following we will write~$H^\bullet_{(r)}(W_n) := H^\bullet(C^k_{(r)}(W_n))$.
\subsection{Stable cohomology of formal vector fields}
We first focus on certain low-dimensional cohomology, the so-called \emph{stable cohomology} of~$W_n$, due to Guillemin and Shnider. They prove in \cite[Corollary 1]{guillemin1973some} that~$H^k(W_n)$ is trivial in dimension~$k = 1,\dots, n$. 
Note that their paper makes much more general statements, in particular about stable cohomology of formal Lie algebras corresponding to other classical vector field Lie algebras, e.g. formal Hamiltonian and divergence-free vector fields.
\begin{definition}
Define for all~$r \in \mathbb{Z}$,
\begin{align}
\partial C^\bullet_{(r)}(W_n) :=
\{ \partial_i \cdot c \in C^\bullet_{(r+1)}(W_n) : c \in C^\bullet_{(r)}(W_n) \},
\end{align}
and
\begin{align}
\partial C^\bullet(W_n) := \bigoplus_{r \in \mathbb{Z}} \partial C^\bullet_{(r)}(W_n).
\end{align}
Recall that~$\partial_i \cdot c$ denotes the action of~$\partial_i \in \mathfrak{g}_{-1}$ on the cochain~$c$ (see Definition~\ref{DefinitionLieAlgebraActionOnCochains}).
\end{definition}
By Corollary~\ref{CorollaryLieAlgebraActsTriviallyOnCohomology}, the Lie algebra action of~$\frakg$ on~$C^\bullet(W_n)$ commutes with the Che\-val\-ley\--Ei\-len\-berg differential, and thus:
\begin{lemma}
\label{LemmaDifferentialComplexesAreComplexes}
For all~$r \in \mathbb{Z}$, the space~$\partial C^\bullet_{(r)}(W_n)$ is a subcomplex of~$C^\bullet_{(r+1)}(W_n)$.
\end{lemma}
We need one more preparing definition, since the degree zero component of cochain complexes is often troublesome.
\begin{definition}[Reduced Complex]
\label{def:ReducedComplex}
If~$C^\bullet$ is a cochain complex, define the \emph{reduced complex}~$\tilde{C}^\bullet$ as 
\begin{align}
\tilde{C}^0 = 0, \quad \tilde{C}^k := C^k\quad \forall k \geq 1,
\end{align}
equipped with the inherited differential from~$C^\bullet$.
\\
We denote by~$\tilde{H}^\bullet$ the cohomology of the reduced complex.
\end{definition}
The aim of this section is the construction of a Koszul complex relating the complexes~$C_{(r)}^\bullet(W_n)$ for different values of~$r$. 
\begin{proposition}
\label{PropKoszulComplexForWn}
There exists an exact sequence of cochain complexes
\begin{equation}
\label{eq:ExactSequenceOfFormalCochains}
\begin{aligned}
0 
\to 
\tilde{C}^\bullet_{(0)}(W_n) \otimes \Lambda^{n} \mathfrak{g}_{-1}
&\to 
\tilde{C}^\bullet_{(1)}(W_n) \otimes \Lambda^{n - 1} \mathfrak{g}_{-1}
\to
\dots
\\
&\to 
\tilde{C}^\bullet_{(n)}(W_n) \otimes \Lambda^{0} \mathfrak{g}_{-1}
\to
\tilde{C}^\bullet_{(n)}(W_n) / \partial \tilde{C}^\bullet_{(n-1)}(W_n) 
\to 
0,
\end{aligned}
\end{equation}
where the differentials in every term are induced by the Chevalley-Eilenberg differential of~$C^\bullet(W_n)$. 
\end{proposition}
\begin{proof}
Write~$V := \frakg_{-1}$. Denote by~$\vee$ the product of symmetric tensors, and define for all~$r \leq n - 1$ the map
\begin{equation}
\begin{aligned}
\sigma_r :
S^\bullet V \otimes  \Lambda^r V
&\to
S^\bullet V \otimes \Lambda^{r-1} V,
\\
u \otimes (\partial_{i_1} \wedge \dots \wedge \partial_{i_r}) 
&\mapsto 
\sum_{j=1}^r 
(-1)^j
(\partial_{i_j} \vee u) 
\otimes 
(\partial_{i_1} \wedge \dotsb \widehat{\partial_{i_j}} \dotsb \wedge \partial_{i_r}) 
\quad 
\forall 
u \in S^\bullet V.
\end{aligned}
\end{equation}
These maps give rise to the well-known, acyclic Koszul complex
\begin{equation}
\begin{aligned}
0 \to S^\bullet V \otimes \Lambda^n V
&
\to
S^\bullet V \otimes  \Lambda^{n-1} V
\to
\dots
\\
&
\to
S^\bullet V \otimes  \Lambda^1 V
\to
S^\bullet V \otimes  \Lambda^0 V
\to
\mathbb{R}
\to 0.
\end{aligned}
\end{equation}
Taking the tensor product of the above exact sequence with~$S^\bullet V,$ and using the canonical isomorphism~$S^\bullet(V^2) \cong S^\bullet(V) \otimes S^\bullet(V)$, we get the exact sequence
\begin{equation}
\begin{aligned}
0 
\to 
S^\bullet V^2 \otimes\Lambda^n V
&\to
S^\bullet V^2 \otimes  \Lambda^{n-1} V
\to
\dots
\\
&
\to
S^\bullet V^2 \otimes  \Lambda^0 V
\to
S^\bullet V
\to 0,
\end{aligned}
\end{equation}
and inductively, if we denote by~$V_\Delta \subset V^k$ the diagonal subspace, then we get an exact sequence
\begin{equation}
\label{eq:ExactComplexAfterFirstTensorProduct}
\begin{aligned}
0 
\to 
S^\bullet V^k 
\otimes
\Lambda^n V
&\to
S^\bullet V^k
\otimes 
\Lambda^{n-1} V
\to
\dots
\\
&
\to
S^\bullet V^k 
\otimes  
\Lambda^0 V
\to
S^\bullet (V^k/V_\Delta)
\to 0.
\end{aligned}
\end{equation}
Let~$k \geq 1$ and denote by~$\Sigma_k$ the permutation group in~$k$ elements. The tensor product~$S^\bullet (V^k) \otimes (V^*)^{\otimes k}$ admits a~$\Sigma_k$ action by signed, simultaneous permutation of the tensor factors in~$(V^*)^{\otimes k}$ and the direct summands in~$V^k$. Taking invariants with respect to this action, we find
\begin{align}
\left(S^\bullet (V^k) \otimes (V^*)^{\otimes k} \right)^{\Sigma_k} \cong C^k(W_n).
\end{align}
Thus, taking the tensor product of the complex \eqref{eq:ExactComplexAfterFirstTensorProduct} with~$(V^*)^{\otimes k}~$ and taking invariants under a finite group are exact functors of~$\mathbb{R}$-vector spaces, and as such we get for every~$k \geq 1$ an exact sequence
\begin{equation}
\label{ComplexExactWn}
\begin{aligned}
0 
\to 
C^k(W_n) \otimes \Lambda^{n} V
&\to 
C^k(W_n) \otimes \Lambda^{n - 1} V
\to
\dots
\\
&\to 
C^k(W_n) \otimes \Lambda^{0} V
\to
C^k(W_n) / \partial \tilde{C}^k(W_n) 
\to 
0.
\end{aligned}
\end{equation}
The last nontrivial map in \eqref{ComplexExactWn} is the canonical quotient projection, the others are:
\begin{equation}
\label{eq:InducedDifferentialOnFinalComplex}
\begin{gathered}
C^k(W_n) \otimes \Lambda^r V
\to
C^k(W_n) \otimes \Lambda^{r-1} V,
\\
c \otimes (\partial_{i_1} \wedge \dots \wedge \partial_{i_r})
\mapsto 
\sum_{j=1}^r 
(-1)^j
(\partial_{i_j} \cdot c) 
\otimes 
(\partial_{i_1} \wedge \dotsb \widehat{\partial_{i_j}} \dotsb \wedge \partial_{i_r}) 
\quad 
\forall c \in \tilde{C}^\bullet(W_n). 
\end{gathered}
\end{equation}
By Corollary~\ref{CorollaryLieAlgebraActsTriviallyOnCohomology}, the differential of~$C^\bullet(W_n)$ commutes with the action of~$S^\bullet V$, and by Lemma~\ref{LemmaDifferentialComplexesAreComplexes}, the differential commutes with the projection 
\begin{align}
C^\bullet(W_n)
\to
C^\bullet(W_n)/\partial C^\bullet(W_n).
\end{align} 
Hence, by taking the direct sum of the complexes \eqref{ComplexExactWn} for all~$k \geq 1$, we receive the following exact sequence of chain complexes
\begin{equation}
\label{eq:CombinedComplex}
\begin{aligned}
0 
\to 
\tilde{C}^\bullet(W_n) \otimes \Lambda^{n} V
&\to 
\tilde{C}^\bullet(W_n) \otimes \Lambda^{n - 1} V
\to
\dots
\\
&\to 
\tilde{C}^\bullet(W_n) \otimes \Lambda^{0} V
\to
\tilde{C}^\bullet(W_n) / \partial \tilde{C}^\bullet(W_n) 
\to 
0.
\end{aligned}
\end{equation} 
With respect to the grading~$\tilde{C}^\bullet(W_n) = \bigoplus_r \tilde{C}^\bullet_{(r)}(W_n)$ the maps \eqref{eq:InducedDifferentialOnFinalComplex} restrict to maps on the graded components 
\begin{align}
C^\bullet_{(k)}(W_n) \otimes \Lambda^r(V) 
\to
C^\bullet_{(k + 1)}(W_n) \otimes \Lambda^{r-1}(V) \quad \forall k,r \in \mathbb{Z},
\end{align}
and the canonical quotient projection~$C^\bullet(W_n) \to
C^\bullet(W_n) / \partial C^\bullet(W_n)~$ restricts to
\begin{align}
C_{(k)}^\bullet(W_n)
\to
C_{(k)}^\bullet(W_n) / \partial C_{(k - 1)}^\bullet(W_n) \quad \forall k \in \mathbb{Z}.
\end{align}
Considering the graded component of \eqref{eq:CombinedComplex} whose leftmost term is~$C^\bullet_{(0)}(W_n) \otimes \Lambda^n V$ yields the desired statement.
\end{proof}
\begin{remark}
The construction of this sequence in \cite[Thm 1]{guillemin1973some} is carried out differently. A detailed proof of their construction would require a study of Hopf algebra theory, which we do not want to carry out here: they implicitly use that if~$M$ is a free module over a Hopf algebra~$H$, then so is~$\Lambda^k_{\mathbb{K}} M$ with its induced, diagonal action for all~$k > 0$. This can, for example, be shown with the methods of \cite[Chapter 7.2]{duascualescu2001hopf}.
\end{remark}
\begin{proposition}[\cite{guillemin1973some}, Corollary 1]
\label{PropWnCohoZeroInLowDegree}
We have~$H^k(W_n) = 0$ if~$k = 1,\dots,n$.
\end{proposition}
\begin{proof}
Consider the exact sequence \eqref{eq:ExactSequenceOfFormalCochains}. We have established that~$\tilde{C}^\bullet_{(r)}(W_n)$ is an acyclic subcomplex whenever~$r \neq 0$, and as such, all terms in the exact sequence are acyclic except for the leftmost and rightmost nontrivial ones. 
Now we can combine this with the exactness of \eqref{eq:ExactSequenceOfFormalCochains} and either apply a straightforward diagram chase, or by view this sequence of cochain complexes as a double complex and comparing the associated spectral sequences. Both methods allow one to deduce the following isomorphisms:
\begin{equation}
\label{eq:IsomoprhismsForDegreeZeroFormalCohomology}
H^k \left( \tilde{C}^\bullet_{(0)}(W_n) \right)
\cong
H^k \left( \tilde{C}^\bullet_{(0)}(W_n) \otimes \Lambda^n \mathfrak{g}_{-1} \right)
\cong
H^{k-n}\left( \tilde{C}^\bullet_{(n)}(W_n) / \partial \tilde{C}^\bullet_{(n-1)}(W_n) \right).
\end{equation}
But, as a relative complex, the complex on the right-hand side is zero in all degrees smaller than~$1$, hence so is its cohomology. Hence, for all~$k = 1,\dots, n$ we have
\begin{align}
H^k(W_n) \cong H^k_{(0)}(W_n) = \tilde{H}^k_{(0)}(W_n) = 0.
\end{align}
\end{proof}
\subsection{A spectral sequence for the cohomology of formal vector fields}
One can do even better than Proposition~\ref{PropWnCohoZeroInLowDegree}: We will formulate a spectral sequence due to Gelfand and Fuks \cite{gel1970formalcohomology} which calculates the cohomology of~$W_n$, and fully specify its differentials. In other words, we will be able to calculate the dimension of~$H^\bullet(W_n)$ in every degree and for every~$n \in \mathbb{N}$.
The information from the previous section about low degree cohomology will aid us for the analysis of this spectral sequence.
To this end, we begin with a short recollection of some representation and cohomology theory of~$\gl_n(\mathbb{R})$.
We largely follow the proof in \cite{fuks1984cohomology}, with certain adaptations that will be indicated.
\begin{theorem}[\cite{kolar2013natural}, Theorem 24.4]
\label{TheoremGLnInvariantsContainedWhere}
Write~$V = \mathbb{R}^n$, and consider~$V$ and~$V^*$ as~$\gl_n(\mathbb{R})$-modules with the defining representation and the dual thereof.\\
For every~$\sigma \in \Sigma_r$, define
\begin{gather*}
\Psi_{\sigma} : V^{\otimes r} \otimes (V^*)^{\otimes r} \to \mathbb{R}, \\
v_1 \otimes \dots \otimes v_r \otimes \alpha_1 \otimes \dots \otimes \alpha_r
\mapsto 
\alpha_1(v_{\sigma(1)}) \dots \alpha_r(v_{\sigma(r)}) \quad \forall \alpha_i \in V^*, v_i \in V.
\end{gather*}
Then~$\{\Psi_{\sigma}\}_{\sigma \in \Sigma_r}$ is a spanning set for
\begin{align}
\Hom_{\mathbb{R}} \left(V^{\otimes r} \otimes (V^*)^{\otimes r}, \mathbb{R}\right)^{\gl_n(\mathbb{R})}
\cong
\left((V^*)^{\otimes r}  \otimes V^{\otimes r} \right)^{\gl_n(\mathbb{R})}
\end{align}
For~$r \leq n$, the set~$\{\Psi_\sigma\}$ is linearly independent.
Further, if~$r \neq s$, then 
\begin{align}
\Hom_{\mathbb{R}} \left(V^{\otimes r} \otimes (V^*)^{\otimes s}, \mathbb{R}\right)^{\gl_n(\mathbb{R})} = 0.
\end{align}
\end{theorem}
%
%
%
%
%
\begin{theorem}[\cite{fuks1984cohomology}, Theorem 2.1.1]
\label{TheoremCohomologyOfGLn}
The cohomology ring~$H^\bullet(\mathfrak{gl}_n(\mathbb{R}))$ is isomorphic to the exterior algebra
\begin{align}
\Lambda^\bullet [\phi_1,\dots,\phi_{2n-1}],
\end{align}
where the~$\phi_i$ are generators in degree~$i$. The inclusion~$\mathfrak{gl}(n-1, \mathbb{R}) \to \mathfrak{gl}(n, \mathbb{R})$ induces a morphism
\begin{align}
H^q(\mathfrak{gl}(n,\mathbb{R})) \to H^q(\mathfrak{gl}(n - 1,\mathbb{R}))
\end{align}
which is an isomorphism for~$q \leq  2n - 3$.
\end{theorem}
\begin{remark}
Note that our reference states the above theorem in an erroneous way: They state the map induced by the inclusion has a one-dimensional kernel for~$q = n$, which, for example, cannot be true when~$n = 2$, since the second cohomology vanishes for all~$\mathfrak{gl}(n,\mathbb{R})$. They also write that the inclusion only induces an isomorphism in degree~$< n$, but their spectral sequence argument actually shows the above, stronger property (see also \cite[Cor 4D.3]{hatcher2002algebraic}.
\end{remark}
\begin{lemma}
\label{lem:InitialShapeOfSecondPage}
Let~$n \in \mathbb{N}$ be arbitrary, and consider the Hochschild-Serre spectral sequence~$\{E^{p,q}_r, d_r\}$ of the Lie-algebra-subalgebra pair~$\mathfrak{g}_0 \subset W_n$ in continuous cohomology (cf. Appendix~\ref{AppendixHochschildSerre}).
\\
We have for all~$p, q \geq 0$
\begin{align}
E^{p,q}_2
=
\begin{cases}
H^q(\frakg_0) \otimes
\left( 
\Lambda^r \mathfrak{g}_{-1} \otimes \Lambda^r \mathfrak{g}_1 
\right)^{\frakg_0} 
&
\text{ if }
p \text{ even  and } p = 2r, \\
0 & \text{ if } p \text{ odd.}
\end{cases}
\end{align}
\end{lemma}
\begin{proof}
By definition of the Hochschild-Serre spectral sequence, the first page takes the following form:
\begin{equation}
\begin{aligned}
\label{eq:FirstPageOfHochschildSerre}
E^{p,q}_1 &= 
H^q 
\left(
\mathfrak{g}_0 ; \Lambda^p (W_n /\mathfrak{g}_0)^*	 
\right) 
=
H^q 
\left(
\mathfrak{g}_0 ; 
\Lambda^p \left( \bigoplus_{j \neq 0} \mathfrak{g}_j \right)^* 
\right) 
\\
&=
\bigoplus_{p_{-1} + p_1 + p_2 + \dots = p} 
H^q 
\left(
\mathfrak{g}_0 ; 
\bigotimes_{j \neq 0}
\Lambda^{p_j} \mathfrak{g}_j ^* 
\right)
\end{aligned}
\end{equation}
Note that, as Lie algebras,~$\mathfrak{g}_0 \subset W_n$ is isomorphic to~$\mathfrak{gl}_n(\mathbb{R}) = \mathfrak{gl}(V)$ via
\begin{align}
\sum_{i,j = 1}^n a_{ij} x_i \partial_j \mapsto (a_{ij})_{1 \leq i ,j \leq n}.
\end{align}
By Weyl's complete reducibility theorem \cite[Theorem 10.9]{hall2015lie}, all coefficient modules in \eqref{eq:FirstPageOfHochschildSerre} are completely reducible. Together with reducibility of the Lie algebra~$\gl_n(\mathbb{R})$ and \cite[Theorem 10]{hochschild1953cohomology}, we may reduce the coefficient space in the above cohomologies to the~$\mathfrak{gl}_n(\mathbb{R})$-invariants. 
Hence,
\begin{equation}
\begin{aligned}
E^{p,q}_1 
&= 
\bigoplus_{p_{-1} + p_1 + p_2 + \dots = p} 
H^q 
\left(
\mathfrak{g}_0 ; 
\left(
\bigotimes_{j \neq 0}
\Lambda^{p_j} \mathfrak{g}_j^*
\right)^{\mathfrak{gl}(V)}
 \right)
\\
&=
H^q(\mathfrak{g}_0) \otimes \left( 
\bigoplus_{p_{-1} + p_1 + p_2 + \dots = p}  
 \left(
\bigotimes_{j \neq 0}
\Lambda^{p_j} \mathfrak{g}_j^*
\right)^{\mathfrak{gl}(V)}
\right) .
\end{aligned}
\end{equation}
By definition,~$\frakg_j^* = (S^{j + 1} V) \otimes V^*$ contains~$j + 1$ tensor factors that transform \emph{covariantly} (i.e. copies of~$V$) under the~$\gl(V)$ action, and one tensor factor that transforms \emph{contravariantly} (i.e. a copy of~$V^*$), hence~$\Lambda^{p_j} \frakg_j^*$ contains~$j \cdot p_j$ covariant and~$p_j$ contravariant tensor factors. Hence
\begin{align}
\bigotimes_{j \neq 0}
\Lambda^{p_j} \mathfrak{g}_j^*
\subset 
(V^*)^
{
\sum_{j \neq 0} p_j
}
\otimes
V^
{
\sum_{j \neq 0} (j + 1)p_j
}
\end{align}
The last part of Theorem~\ref{TheoremGLnInvariantsContainedWhere} then implies that the space of~$\gl(V)$-invariants of the space~$\bigotimes_{j \neq 0}
\Lambda^{p_j} \mathfrak{g}_j^*$ is only nonzero if
\begin{align}
\label{EquationFirstRequirementOnIndices}
\sum_{j \neq 0} p_j
=
\sum_{j \neq 0} (j + 1)p_j,
\text{ or equivalently }
p_{-1} 
= 
\sum_{j = 1}^\infty j \cdot p_j
\end{align}
Simultaneously, again from Theorem~\ref{TheoremGLnInvariantsContainedWhere}, we know that~$\mathfrak{gl}_n(\mathbb{R})$-invariants in a tensor module 
\begin{align}
V^{\otimes r} \otimes (V^*)^{\otimes r} \cong \Hom((V^*)^{\otimes r} \otimes V^{\otimes r}, \mathbb{R})
\end{align} 
can be described as the linear combinations of the functionals which contract all covariant indices with permutations of the contravariant indices. 
Correspondingly, the~$\mathfrak{gl}(V)$ invariants in the subspaces~$\bigotimes_{j \neq 0}
\Lambda^{p_j} \mathfrak{g}_j^*$ are given by subjecting these functionals to the required (skew-)symmetrizations.
%
%
Hence if 
\begin{align}
p_{-1}  > \sum_{j =1}^\infty p_j,
\end{align} 
then by the pidgeonhole principle, any invariant tensor contracts at least two contravariant factors belonging to~$\Lambda^{p_{-1}} \mathfrak{g}_{-1}$ with two covariant factors, both belonging to a single copy of~$\mathfrak{g}_j$ within~$\Lambda^{p_j} \mathfrak{g}_j = \mathfrak{g}_j \wedge \dots \wedge \mathfrak{g}_j$ for some~$j \geq 1$.
However, in such a contraction the contravariant factors would behave skew-symmetrically and the covariant ones symmetrically under permutation, hence their contraction is zero.
Thus, we get the additional requirement
\begin{align}
\label{EquationSecondRequirementOnIndices}
p_{-1} \leq \sum_{j =1}^\infty p_j
\end{align}
Combining \eqref{EquationFirstRequirementOnIndices} and \eqref{EquationSecondRequirementOnIndices} we end up with
$p_k = 0$ for~$k \geq 2$ and~$p_{-1} = p_1 =: r$. This implies that~$p = 2r$ is even whenever there are nontrivial invariants, and thus every other column in the page~$E^{p,q}_1$ vanishes. Hence all differentials on the first page are trivial, and~$E^{p,q}_1 = E^{p,q}_2$. This concludes the proof.
\end{proof}
Since~$H^q(\frakg_0)$,~$\Lambda^r \frakg_{-1}$, and~$\Lambda^r \frakg_1$ are nonzero only for finitely many~$q,r \geq 0$, and all involved spaces are finite-dimensional in every degree, we have:
\begin{corollary}
For all~$q \geq 0$, the continuous cohomology~$H^q(W_n)$ is fi\-nite\--di\-men\-sio\-nal, and~$H^q(W_n) \neq 0$ only in finitely many degrees.
\end{corollary}
Let us further analyze the invariant space~$\left( 
\Lambda^r \mathfrak{g}_{-1} \otimes \Lambda^r \mathfrak{g}_1 
\right)^{\frakg_0}$.
\begin{lemma}
\label{lem:FinalShapeOfSecondPage}
Let~$\{E^{\bullet, \bullet}_2, d_r\}$ be the spectral sequence from Lemma~\ref{lem:InitialShapeOfSecondPage} of the Lie\--al\-ge\-bra\--sub\-al\-ge\-bra pair~$\frakg_0 \subset W_n$.
\\
For all~$r = 1,\dots,n$, there exist multiplicative generators~$\Psi_{2r} \in E^{2r,0}_2$ so that
\begin{align}
E^{\bullet,0}_2 &= \mathbb{R}[\Psi_2, \Psi_4\dots,\Psi_{2n}] / \langle \Psi_{i_1} \dots \Psi_{i_k} : i_1 + \dots + i_k > 2n \rangle.
\end{align}
\end{lemma}
\begin{proof}
By Lemma~\ref{lem:InitialShapeOfSecondPage} we have
\begin{align}
E^{\bullet, 0}_2 = \bigoplus_{r \geq 0} \left( \Lambda^r \frakg_{-1} \otimes \Lambda^r \frakg_1 \right)^{\frakg_0}.
\end{align}
By Theorem~\ref{TheoremGLnInvariantsContainedWhere} the elements in the invariant space~$\left( 
\Lambda^r \mathfrak{g}_{-1} \otimes \Lambda^r \mathfrak{g}_1 
\right)^{\frakg_0}$ arise by taking, for every permutation~$\sigma \in \Sigma_r$, the functional 
\begin{align}
\bigotimes^r V \times \bigotimes^r \left(V^* \otimes V^*  \otimes V \right) \to \mathbb{R},
\end{align}
with
\begin{equation}
\begin{aligned}
(
	\alpha_1^1 \otimes \dots &\otimes \alpha_r^1,
	( \beta_1^1 \otimes \beta_1^2 \otimes \alpha_{r+1}^2)
	\otimes \dots \otimes 
	(\beta_r^1 \otimes \beta_{r}^2 \otimes \alpha_{2r}^2)
)
\\
&\mapsto
\beta_1^1(\alpha_1^1) \dots \beta_r^1 (\alpha_r^1) \cdot \beta_1^2(\alpha_{\sigma(1)}^2) \dots \beta_r^2(\alpha_{ \sigma(r)}^2), \quad 
\forall \alpha_i \in V, \,
\beta_i^1, \beta_i^2 \in V^*,
\end{aligned}
\end{equation}
skew-symmetrizing over the first~$r$ and the last~$r$ arguments, and symmetrizing over the exchange~$\beta_i^1 \leftrightarrow \beta_i^2$. Denote the arising functional by~$\Psi_{\sigma} \in \left( \Lambda^r \mathfrak{g}_{-1} \otimes \Lambda^r \mathfrak{g}_1 \right)^{\mathfrak{gl}(V)}$.
From the skew-symmetry of the functionals in the first~$r$ and last~$r$ arguments, one deduces that 
\begin{equation}
\label{eq:PsisInvariantUnderConjugation}
\Psi_{\sigma} = \Psi_{\tau \sigma \tau^{-1}} \quad \forall \sigma, \tau \in \Sigma_k.
\end{equation}
Another straightforward calculation shows that for~$\sigma \in \Sigma_r$ and~$\tau \in \Sigma_l$ we have
\begin{equation}
\label{eq:PsisProductFormula}
\Psi_\sigma \wedge \Psi_\tau = \Psi_{\sigma \tau} 
\in 
\left( \Lambda^{r + l} \mathfrak{g}_{-1} \otimes \Lambda^{r + l} \mathfrak{g}_1 \right)^{\mathfrak{gl}(V)}.
\end{equation}
Because~$\mathfrak{g}_{-1}$ is~$n$-dimensional,~$\left( \Lambda^r \mathfrak{g}_{-1} \otimes \Lambda^r \mathfrak{g}_1 \right)^{\mathfrak{gl}(V)} = 0$ if~$r > n$, so in particular
\begin{align}
\Psi_\sigma \wedge \Psi_\tau = 0 \text{ if } \sigma \in \Sigma_r, \tau \in \Sigma_l, r + l > n.
\end{align}
Denote by~$\Psi_r \in \left( \Lambda^r \mathfrak{g}_{-1} \otimes \Lambda^r \mathfrak{g}_1 \right)^{\mathfrak{gl}(V)}$ the~$\Psi_\sigma$ corresponding to an~$r$-cycle~$\sigma \in \Sigma_r$. This is well-defined, since by \eqref{eq:PsisInvariantUnderConjugation} the functional~$\Psi_{\sigma}$ only depends on the conjugacy class of~$\sigma$, and all~$r$-cycles are conjugate to one another. 
\\
Since every permutation can be uniquely (up to ordering) decomposed into a composition of cycles, \eqref{eq:PsisProductFormula} shows that~$ \left( \Lambda^r \mathfrak{g}_{-1} \otimes \Lambda^r \mathfrak{g}_1 \right)^{\mathfrak{gl}(V)}$ is multiplicatively generated by the~$\Psi_{2r} \in \left( \Lambda^r \mathfrak{g}_{-1} \otimes \Lambda^r \mathfrak{g}_1 \right)^{\mathfrak{gl}(V)}$ for every~$r = 1,\dots, n$. 
Hence every element in~$E^{\bullet,0}_2$ is given as a unique (up to ordering) product of elements in~$\{\Psi_2,\dots,\Psi_{2n}\}$. Since all~$\Psi_{\sigma}$ for~$\sigma \in \Sigma_r$ and~$r \leq n$ are nonzero, the only relation between these generators is that products~$\Psi_{i_1} \dots \Psi_{i_k}$ are zero if~$i_1 + \dots + i_k > 2n$, and the lemma is proven.
\end{proof}
To understand this spectral sequence further, we will need the Borel transgression theorem. To formulate it, let us first define some terminology.
\begin{definition}
Let~$\{E^{p,q}_r, d_r\}_{r \geq 0}$ be a cohomological first-quadrant spectral sequence. Denote by~$\kappa_r^{r+1} : \ker d_r \to E^{\bullet, \bullet}_{r+1}$ the natural quotient map from cocycles of the~$r$-th page differential~$d_r$ to the~$r+1$-th page, and 
\begin{align}
\kappa_r^s = \kappa_{s-1}^s \circ \dots \circ \kappa_{r}^{r+1} \quad \forall s > r,
\end{align}
where the domain of~$\kappa_r^s$ is defined inductively as all the~$c \in E^{\bullet, \bullet}_{r+1}$ in the domain of~$\kappa_r^{s-1}$ so that~$\kappa_r^{s-1} c \in \ker d_{s-1}$.
We call an element~$c \in E^{p,0}_2$ \emph{transgressive} if, for all~$r$ with~$2 \leq r \leq p$, we have that~$c$ is in the domain of~$\kappa_2^r$.
\end{definition}
Intuitively, the transgressive elements in~$E^{p,0}_{2}$ are the ones which ``survive'' until the very last moment: Only the differential~$d_{p+1} : E^{p,0}_{p+1} \to E^{0,p +1}_{p+1}$, also called the \emph{transgression}, can map it to something nontrivial.
By abuse of notation, we often denote an element in the domain of~$\kappa_r^s$ by the same symbol as its image under~$\kappa_r^{s}$ in the higher page~$E_{s}^{\bullet, \bullet}$
 The following theorem was originally proven in \cite{borel1953cohomologie}, but we cite a slightly stronger version from \cite[Thm 2.9]{mimura1991topology}.
\begin{theorem}[Borel transgression theorem]
\label{thm:BorelTransgression}
Consider two finite-dimensional, graded vector spaces~$B^\bullet := \bigoplus_{p \in \mathbb{N}_0} B^p$ and~$F^\bullet := \bigoplus_{q \in \mathbb{N}_0} F^q$. Assume there are elements~$x_i \in F^\bullet$ of odd degree such that
\begin{align}
\Lambda^\bullet[x_1,\dots,x_l] \to F^\bullet
\end{align}
is bijective in degrees~$\leq N$ and injective in degree~$N + 1$.
Let further~$\{E_r^{p,q}, d_r\}_{r \geq 0}$ be a cohomological spectral sequence whose second page has the form
\begin{align}
E_2^{p,q} = B^p \otimes F^q,
\end{align}
and which converges towards a graded vector space~$H^\bullet$ with~$H^k = 0$ if~$0 < k \leq N + 2$.
Then we can choose the generators~$x_i$ to be transgressive, and if~$y_1,\dots,y_l \in B^\bullet$ denotes a collection of elements with
\begin{align}
d_{\deg x_i + 1}  x_i =  y_i \quad i =1,\dots, l,
\end{align}
then the map
\begin{align}
\mathbb{R}[y_1,\dots,y_l] \to B^\bullet 
\end{align}
is bijective for degrees~$\leq N$ and injective for degree~$N + 1$.
\end{theorem}
Using this, we can fully describe the desired spectral sequence:
\begin{theorem}[\cite{fuks1984cohomology}, Theorem 2.2.4]
\label{TheoremCohomologyOfWn}
Let~$n \in \mathbb{N}$ be arbitrary, and consider the Hoch\-schild\--Serre spectral sequence~$\{E^{p,q}_r, d_r\}$ of the pair~$\mathfrak{g}_0 \subset W_n$ in continuous cohomology.
Its second page takes the form
\begin{align}
E^{0,\bullet}_2 &= \Lambda^\bullet [ \phi_1, \phi_3,\dots, \phi_{2n -1}] , \\
E^{\bullet,0}_2 &= \mathbb{R}[\Psi_2, \Psi_4\dots,\Psi_{2n}] / \langle \Psi_{i_1} \dots \Psi_{i_k} : i_1 + \dots + i_k > 2n \rangle, 
\\
E^{p,q}_2 &= E^{p,0}_2 \otimes E^{0,q}_2,
\end{align}
and all differentials of the spectral sequence are fully specified on the generators by
\begin{align}
d_{i + 1}\phi_i &= \Psi_{i+1} \quad i \in \{1,3,\dots,2n-1\}.
\end{align}
\end{theorem}
\begin{proof}
The form of the second page follows from Theorem~\ref{TheoremCohomologyOfGLn}, Lemma~\ref{lem:InitialShapeOfSecondPage}, and Lemma~\ref{lem:FinalShapeOfSecondPage}. It remains to show the statement about the differentials.
Consider the Hochschild-Serre spectral sequence for the pair of Lie algebras~$(W_{3n}, \gl_{3n}(\mathbb{R}))$. By Proposition~\ref{PropWnCohoZeroInLowDegree}, we know that~$H^k(W_{3n}) = 0$ in degrees~$k = 1,\dots, 3n$, and up to degree~$2n$, the zeroeth column of the spectral sequence is equal to the exterior algebra~$\Lambda^\bullet [\phi_1,\phi_3,\dots,\phi_{2n - 1}]$, the~$\phi_i$ being the generators of~$H^\bullet(\gl_{3n}(\mathbb{R}))$. Hence we can apply the Borel transgression theorem~\ref{thm:BorelTransgression} with~$N = 2n - 1$, implying that generators~$\tilde \phi_1,\dots,\tilde \phi_{2n-1}$ of the zeroeth column can be chosen so that
\begin{align}
d_{2i} \tilde \phi_i =  \Psi_{i+1}, \quad i \in \{1, 3,\dots,2n - 1\}.
\end{align}
Now the inclusion~$W_n \to W_{3n}$ induces a morphism from the Hochschild-Serre spectral sequence for the Lie-algebra-subalgebra pair~$(W_{3n}, \gl_{3n}(\mathbb{R}))$ to the spectral sequence for~$(W_{n}, \gl_{n}(\mathbb{R}))$. Under this morphism, the generators~$\Psi_i$ for~$W_{3n}$ restrict to equivalent ones for~$W_n$ by the explicit formula for them given in Lemma~\ref{lem:FinalShapeOfSecondPage}. 
\\
By functoriality of the Hochschild-Serre spectral sequence, the generators~$\tilde{\phi}_i$ must then restrict to something nonzero in the space~$E^{0,2i - 1}_{2i}$ of the spectral sequence for~$W_n$. By an inductive argument, this vector space is one-dimensional and generated by the generator~$\phi_i$, so the image of~$\tilde{\phi}_i$ under this morphism must be a nonzero multiple of~$\phi_i$. Hence, up to a nonzero constant, we have in the spectral sequence for~$W_n$ that
\begin{align}
d_{2i} \phi_i = \Psi_{i+1}, \quad i \in \{1,3,\dots,2n - 1\}.
\end{align}
This proves that all generators~$\phi_1,\dots,\phi_{2n-1}$ in the spectral sequence for~$W_n$ map as desired. Since the differential of the Hochschild-Serre spectral sequence is multiplicative and all pages are generated by the~$\phi_i$ and the~$\Psi_i$, this fully specifies the differential on every page.
\end{proof}
\begin{remark}
In \cite{fuks1984cohomology}, the above argument is carried out with the Lie\--al\-ge\-bra\--sub\-al\-ge\-bra pair~$(W_{2n},\gl_{2n}(\mathbb{R}))$ rather than~$(W_{3n},\gl_{3n}(\mathbb{R}))$, which would not fulfil the requirements of the version of the Borel transgression theorem we use here. 
\end{remark}
\def\sseqpacking{\sspackhorizontal}
\begin{figure}
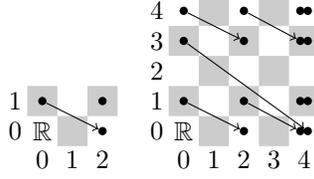

\centering
\begin{sseq}[grid=chess,labelstep=1]{0...2}{0...1}
\ssdrop{\mathbb{R}}
\ssmove 0 1
\ssdropbull
\ssarrow{2}{-1}
\ssdropbull
\ssmove 0 1
\ssdropbull
\end{sseq}
\begin{sseq}[grid=chess,labelstep=1]{0...4}{0...4}
\ssdrop{\mathbb{R}}
\ssmove 0 1
\ssdropbull
\ssarrow{2}{-1}
\ssdropbull
\ssmove 0 1
\ssdropbull
\ssarrow{2}{-1}
\ssdropbull 
\ssdropbull \ssname{b}
\ssmove 0 1
\ssdropbull 
\ssdropbull 
\ssmoveto 0 3
\ssdropbull \ssname{a}
\ssmove 0 1
\ssdropbull
\ssarrow{2}{-1}
\ssdropbull
\ssmove 0 1
\ssdropbull
\ssarrow{2}{-1}
\ssdropbull
\ssdropbull 
\ssmove 0 1
\ssdropbull
\ssdropbull
\ssgoto{a}
\ssgoto{b}
\ssstroke[arrowto]
\end{sseq}
\caption{The spectral sequences for~$W_1$ and~$W_2$, with nonvanishing differentials indicated. Every dot represents one basis element of the term in the given position.
The cohomology of~$W_1$ is only nontrivial in degree~$0$ and~$3$, whereas the cohomology of~$W_2$ is nontrivial in degree~$0, 5,7$ and~$8$, degree~$5$ and~$8$ having multiplicity 2.}
\end{figure}
This allows one to fully calculate the dimensions of~$H^\bullet(W_n)$ in all degrees and even offers some insight into the behavior of representatives of the cohomology classes.
We are going to sumarize the most important properties of~$H^\bullet(W_n)$ in the following corollary:
\begin{corollary}
\label{CorollaryWedgeProductOfFormalsIsZero}
The space~$H^k(W_n)$ is trivial when~$1 \leq k \leq 2n$ or~$k > n^2 + 2n$. The wedge product of two cohomology classes of positive degree in~$H^\bullet(W_n)$ is zero.
\end{corollary}
\begin{proof}
Any element in~$E^{p,q}_2$ with~$(p,q) \neq (0,0)$ is a linear combination of terms of the form
\begin{align}
\label{eq:SpecialTermInSpectralSequence}
\phi_{i_1} \dots \phi_{i_s} \Psi^{m_1}_{j_1} \dots \Psi^{m_t}_{j_t},
\end{align}
where we have ordered the groups of indices so that~$i_1 < \dots < i_s$ and~$j_1 < \dots < j_t$. We further have
\begin{align}
i_1 + \dots + i_s = q, \quad m_1 j_1 + \dots + m_t j_t = p,
\end{align} 
and~$s$ and~$t$ are possibly zero, but not both at the same time. 
Theorem~\ref{TheoremCohomologyOfWn} shows:
\begin{itemize}
\item[i)] If~$s = 0$ or~$i_1 > j_1$, then~$\phi_{j_1 - 1} \phi_{i_1} \dots \phi_{i_s} \Psi^{m_1 - 1}_{j_1} \dots \Psi^{m_t}_{j_t}$ maps to the term \eqref{eq:SpecialTermInSpectralSequence} under the differential~$d_{2(j_1 - 1)}$.
\item[ii)] If~$t = 0$ or~$i_1 < j_1$, then the term \eqref{eq:SpecialTermInSpectralSequence} maps to~$\phi_{i_2} \dots \phi_{i_s} \Psi_{i_1 + 1} \Psi^{m_1}_{j_1} \dots \Psi^{m_t}_{j_t}$ under the differential~$d_{2 i_1}$. 
\end{itemize}
Since necessarily~$i_1$ is odd and~$j_1$ is even, one of i) or ii) must hold.
Let us show that under the assumptions of ii), and if~$p \leq n$ or~$p + q \leq 2n$, then the product~$\Psi_{i_1 + 1} \Psi^{m_1}_{j_1} \dots \Psi^{m_t}_{j_t}$ is nonzero. This is equivalent to showing
\begin{align}
i_1 + 1 + m_1 j_1 + \dots + m_t j_t \leq 2n.
\end{align}
If~$t = 0$, then this is trivial, so assume from here on~$t > 0$ and~$i_1 < j_1$.
If~$p \leq n$, then we know that
\begin{align}
i_1 < j_1 \leq m_1 j_1 + \dots + m_t j_t = p,
\end{align}
thus
\begin{align}
i_1 + 1 + m_1 j_1 + \dots + m_t j_t \leq 2p \leq 2n.
\end{align}
On the other hand, assume~$p + q \leq 2n$.
Note that~$s = 1$ is never the case, since~$i_1$ is always odd and~$q$ is always even. Hence assume~$s > 1$, so that~$i_1 + 1 \leq i_1 + \dots + i_s$. But then
\begin{align}
i_1 + 1 + m_1 j_1 + \dots + m_t j_t \leq
i_1 + \dots + i_s + m_1 j_1 + \dots + m_t j_t = q + p \leq 2n.
\end{align}
With this, we have shown that~$E^{p,q}_\infty$ is zero when~$0 < p + q \leq 2n$, implying that~$H^k(W_n)$ vanishes in degree~$0 < k \leq 2n$. Further, we have shown that~$E^{p,q}_\infty$ is zero when~$p \leq n$, so two cohomology classes of positive degree in~$H^\bullet(W_n)$ correspond to two equivalence classes in some spaces~$E^{p,q}_\infty, E^{p',q'}_\infty$ with~$p,p' > n$. By multiplicativity of the Hochschild-Serre spectral sequence, their product must then correspond to an equivalence class in~$E^{p + p', q + q'}_\infty$, which must be zero because~$p + p' > 2n$. This proves the corollary.
\end{proof}

\section{Gelfand--Fuks cohomology on Euclidean space}
\label{SectionLocalCohomology}
In this section, we calculate the Gelfand--Fuks cohomology~$H^\bullet(\mathfrak{X}(M))$ for~$M = \mathbb{R}^n$. We follow an elaborate outline by Bott, see \cite{bott1973lectures}. This approach will allow us to easily extend our proof to the Gelfand--Fuks cohomology of finite disjoint unions~$\mathbb{R}^n \sqcup \dots \sqcup \mathbb{R}^n$ and also to certain \emph{diagonal} cohomologies thereof, a concept which we introduce in Section~\ref{SectionGelfandFuksOnSmooths}.
\subsection{Definitions and calculation}
The Lie algebra of smooth vector fields~$\mathfrak{X}(\mathbb{R}^n)$ on Euclidean space is a locally convex Lie algebra with respect to the standard Fréchet topology. We are interested in its continuous Chevalley-Eilenberg cohomology with respect to this topology. We can express vector fields in the canonical coordinates
\begin{align}
\mathfrak{X}(\mathbb{R}^n) =
\left \{ \sum_{i=1}^n f_i \partial_i : 
f_i \in C^\infty(\mathbb{R}^n) \right \}
\end{align}
Let us again identify some structures:
\begin{definition}
Let~$t > 0$. We define the family of scaling operators ~$\{T_t\}_{t > 0}$ as
\begin{align}
T_t : \mathbb{R}^n \to \mathbb{R}^n, 
\quad
x \mapsto tx.
\end{align}
\end{definition}
We adopt the notation that if~$\phi : \mathbb{R}^n \to \mathbb{R}^n$ is a local diffeomorphism, then we denote by~$\phi^*$ its pullback on vector fields~$\mathfrak{X}(\mathbb{R}^n)$
\begin{align}
\phi^* (X) := (d \phi)^{-1} (X \circ \phi). 
\end{align}
We overload our notation and also write~$\phi^*$ for the pullback on cochains~$C^k(\mathfrak{X}(\mathbb{R}^n))$ for all~$k \geq 0$:
\begin{gather}
\label{eq:PullbackOfCochainsAlongLocalDiffeos}
(\phi^* c)(X_1,\dots,X_k) := c(\phi^* X_1, \dots, \phi^* X_k), \quad \forall k \geq 0.
\end{gather}
Note that for~$c \in C^0(\mathfrak{X}(\mathbb{R}^n))$ this amounts to~$\phi^*c = c$.
\\
For the scaling operators~$\{T_t\}_{t > 0}$, this translates to
\begin{gather*}
T_t^* X = \frac 1 t (X \circ T_{t}) \quad \forall X \in \mathfrak{X}(\mathbb{R}^n) , \\
(T_t^* c)(X_1,\dots,X_k) := \frac{1}{t^k} c(X_1 \circ T_{t}, \dots, X_k \circ T_{t}) \quad \forall c \in C^k(\mathfrak{X}(\mathbb{R}^n).
\end{gather*}
\begin{definition}
We define the subspace~$\mathfrak{X}_{\pol}(\mathbb{R}^n) \subset \mathfrak{X}(\mathbb{R}^n)$ of \emph{polynomial vector fields}
\begin{align}
\mathfrak{X}_{\pol}(\mathbb{R}^n) := 
\left \{ 
\sum_{i=1}^n f_i  \partial_i : 
f_i \in \mathbb{R}[x_1,\dots,x_n] 
\right \}.
\end{align} 
It admits the structure of a graded Lie algebra
\begin{align}
\mathfrak{X}_{\pol}(\mathbb{R}^n) = \bigoplus_{k \in \mathbb{Z}} P_k , \quad
P_k := \{
X \in  \mathfrak{X}_\mathrm{pol}(\mathbb{R}^n) : T_t^* X = t^k X \quad \forall t > 0 \}.
\end{align}
Elements of~$P_k$ are called \emph{homogeneous} vector fields (of degree~$k$).
\end{definition}
\begin{remark}
Compare this to the grading~$W_n = \bigoplus_{k \in \mathbb{Z}} \frakg_k$ on formal vector fields. The sets~$P_k$ are the images of the natural embeddings~$\mathfrak{g}_k \to \mathfrak{X}(\mathbb{R}^n)$ that we get from considering finite formal vector fields in~$W_n$ as polynomial vector fields on~$\mathbb{R}^n$.
\end{remark}
\begin{definition}
For all~$k \in \mathbb{Z}, q \geq 0$, define
\begin{align}
F^k C^q (\mathfrak{X}(\mathbb{R}^n)) 
:= 
\left\lbrace
c \in C^q(\mathfrak{X}(\mathbb{R}^n)) 
: 
\lim_{t \to 0} t^{-k} T_t^* c(X_1,\dots,X_q) 
\text{ exists } \forall X_i \in \mathfrak{X}(\mathbb{R}^n) 
\right\rbrace. 
\end{align}
\end{definition}
\begin{lemma}
\label{LemmaPropertiesOfRnFiltration}
The spaces~$F^k C^\bullet (\mathfrak{X}(\mathbb{R}^n))$ for~$k \in \mathbb{Z}$ fulfil the following properties:
\begin{itemize}
\item[i)] For every~$k \in \mathbb{Z}$, the space~$F^k C^\bullet (\mathfrak{X}(\mathbb{R}^n))$ is a subcomplex of~$C^\bullet (\mathfrak{X}(\mathbb{R}^n))$.
\item[ii)] We have a descending chain
\begin{align}
\dots 
\subset F^k C^\bullet (\mathfrak{X}(\mathbb{R}^n)) 
\subset F^{k-1} C^\bullet (\mathfrak{X}(\mathbb{R}^n))
\subset F^{k-2} C^\bullet (\mathfrak{X}(\mathbb{R}^n))
\subset
\dots
\end{align}
\item[iii)] For all~$k,l \in \mathbb{Z}$ we have
 \begin{align}
F^k C^\bullet (\mathfrak{X}(\mathbb{R}^n)) \wedge
F^l C^\bullet (\mathfrak{X}(\mathbb{R}^n))
\subset
F^{k+l} C^\bullet (\mathfrak{X}(\mathbb{R}^n)).
 \end{align}
 \item[iv)] For~$k \leq -n$ we have~$F^k C^\bullet (\mathfrak{X}(\mathbb{R}^n)) = C^\bullet (\mathfrak{X}(\mathbb{R}^n))$.
\end{itemize}
Summarizing,~$\{F^k C^\bullet (\mathfrak{X}(\mathbb{R}^n)) \}_{k \in \in \mathbb{Z}}$ constitutes a descending filtration of the Che\-val\-ley\--Ei\-len\-berg complex~$C^\bullet(\mathfrak{X}(\mathbb{R}^n))$ which is bounded from above.
\end{lemma}
\begin{proof}
\textit{i)}
Pulling back vector fields along local diffeomorphisms is a Lie algebra homomorphism, so
\begin{align}
T_t^* [X,Y] = [T_t^* X, T_t^* Y] \quad \forall X, Y \in \mathfrak{X}(\mathbb{R}^n).
\end{align}
Hence, the pullback on cochains~$T_t^*$ commutes with the Lie algebra differential, so if the appropriate limits exist for a cochain~$c$, they also do for~$dc$. Thus the~$F^k C^\bullet(\mathfrak{X}(\mathbb{R}^n))$ are indeed subcomplexes.
\\
\textit{ii)} 
This follows since if the limit~$\lim_{t \to 0} t^{-k} f(t)$ exists for some function~$t \mapsto f(t)$, so does~$\lim_{t \to 0} t^{-(k-1)}f(t) = 0$.
\\
\textit{iii)}
The compatibility with the wedge product follows from 
\begin{align}
T_t^* (c_1 \wedge c_2) = T_t^* c_1 \wedge T_t^* c_2 \quad \forall c_1, c_2 \in C^\bullet(\mathfrak{X}(\mathbb{R}^n)).
\end{align}
\textit{iv)}
In degree zero the statement is clear, since
\begin{align}
F^k C^0(\mathfrak{X}(\mathbb{R}^n)) = C^0(\mathfrak{X}(\mathbb{R}^n))
\quad
\forall k \leq 0.
\end{align}
Assume now~$q > 0$. Fix~$c \in C^q(\mathfrak{X}(\mathbb{R}^n))$ and~$X_1,\dots, X_q \in \mathfrak{X}(\mathbb{R}^n)$.
Applying the Hadamard lemma to every one of the~$X_i$ shows that there are vector fields~$X_k^{(i)}$ for all~$k = 1,\dots,q$ and~$i = 1,\dots,n$ so that
\begin{align}
X_{k}(x) = X_{k}(0) + \sum_{i=1}^n x_i  X_{k}^{(i)}(x), \quad \forall x \in \mathbb{R}^n.
\end{align}
Then
\begin{align}
T_t^* X_{k}(x) = \frac{1}{t} X_{k}(0) + \sum_{i=1}^n x_i X_{k}^{(i)}(tx) \quad \forall x \in \mathbb{R}^n.
\end{align}
Hence we can rewrite
\begin{align}
T_t^* c (X_1,\dots,X_q) =
c \left( 
\frac{1}{t} X_{1}(0) + \sum_{i=1}^n x_i  X_{1}^{(i)}(tx),
\dots,
\frac{1}{t} X_{q}(0) + \sum_{i=1}^n x_i  X_{q}^{(i)}(tx) \right).
\end{align}
Decomposing this expression using multilinearity of~$c$, we find that all the terms whose order in~$t$ is lower than~$-n$ have to vanish, since, for any collection of indices $\{i_1,\dots,i_{n+1} \} \subset \{1,\dots,q\}$, the set~$\{X_{i_1}(0), \dots, X_{i_{n+1}}(0)\}$is necessarily linearly dependent and~$c$ is skew-symmetric. Note also that on any compact set in~$\mathbb{R}^n$, the vector fields~$x \mapsto x_i X_{k}^{(i)}(tx)$ converge uniformly to the vector field~$x \mapsto x_i X_k^{(i)}(0)$ for~$t \to 0$, and the same holds for their derivatives.
Combining the two previous facts, the continuity of~$c$ lets us conclude that the limit~$\lim_{t \to 0} \tfrac{1}{t^{-n}} T_t^* c (X_1,\dots,X_q)$ exists. This proves the statement.
\end{proof}
The analysis at the end of the previous proof motivates a different characterization of the filtration:
\begin{lemma}
\label{LemmaVecFieldsRnFiltrationCharacterization}
Let~$q  > 0$. A cochain~$c \in C^q(\mathfrak{X}(\mathbb{R}^n))$ lies in~$F^k C^q(\mathfrak{X}(\mathbb{R}^n))$ if and only if for all homogeneous vector fields~$X_1,\dots,X_q \in \mathfrak{X}_{\pol}(\mathbb{R}^n)$ we have
\begin{align}
\sum_{i=1}^q \deg X_i < k \implies c(X_1,\dots,X_k) = 0. 
\end{align}
\end{lemma}
\begin{proof}
$\Rightarrow:$ Let~$c \in F^k C^q(\mathfrak{X}(\mathbb{R}^n))$, and let~$X_1,\dots,X_q \in \mathfrak{X}_{\pol}(\mathbb{R}^n)$ be any homogeneous polynomial vector fields. Then
\begin{align}
\label{eq:ScalingOnHomogeneousVectorFields}
t^{-k} T_t^* c
(X_1,\dots,X_q) = t^{\left( \sum_{i=1}^q \deg X_i \right) - k} c(X_1,\dots,X_q).
\end{align}
If~$\sum_{i=1}^q \deg X_i  < k$, this can only converge to a finite value in the limit~$t \to 0$ if~$c$ vanishes on $(X_1,\dots,X_q)$. This proves the first implication.
\\
$\Leftarrow:$
Let~$c \in C^q(\mathfrak{X}(\mathbb{R}^n))$ be so that it vanishes on all homogeneous vector fields whose degree adds to a value smaller~$k$.
Set~$r := \max \{q + k, 1\}$. Given any vector fields~$X_1,\dots,X_q \in \mathfrak{X}(\mathbb{R}^n)$, apply the Hadamard lemma to each of them~$r$ times to write, in multiindex notation,
\begin{align}
X_k (x) = Y_k(x) 
+ 
\sum_{ \substack{\vec\alpha \in \mathbb{N}^{n}_0 \\ |\vec{\alpha}| =  r + 1} } 
x_{1}^{\alpha_1} \dots x_{n}^{\alpha_n} Z_k^{\vec \alpha}(x)
=: Y_k(x) + Z_k(x),
\end{align}
where~$Y_k$ is a polynomial vector field whose homogeneous components are of degree~$\leq r - 1$, and~$Z_k^{\vec \alpha} \in \mathfrak{X}(\mathbb{R}^n)$. 
Using multilinearity of~$c$, decompose~$T_t^* c(X_1,\dots,X_q)$ into summands of the form~$\tfrac{1}{t^k} T_t^* c$ with all arguments being some~$Y_k$ or some~$Z_k$ for~$k = 1,\dots,q$.
The limits~$\lim_{t \to 0} t T_t^* Y_k$ and~$\lim_{t \to 0} t^{-r} T_t^* Z_k$ in~$\mathfrak{X}(\mathbb{R}^n)$ exist for all~$k = 1,\dots,q$. Thus any summand in the decomposition of~$\tfrac{1}{t^k}  \left( T_t^* c \right)(X_1,\dots,X_q)$ is of the following form for a certain~$s \geq 0$ and certain~$i_1,\dots,i_s,j_1,\dots,j_{q-s} \in \{1,\dots,q\}$,:
\begin{equation}
\begin{aligned}
\frac{1}{t^k} \left( T_t^* c \right)(&Z_{i_1},\dots,Z_{i_s},Y_{j_1},\dots,Y_{j_{q - s}})
\\
&=
t^{r s - (q - s) - k} 
(T_t^* c)( t^{-r} Z_{i_1},
\dots,
t^{-r}  Z_{i_s}, 
t Y_{j_1},
\dots,
t Y_{j_{q-s}}).
\end{aligned}
\end{equation}
If~$s \geq 1$, then we have due to~$r \geq q + k$,
\begin{align}
r s - (q - s) - k \geq s \geq 1,
\end{align}
and the limit~$t \to 0$ exists.\\
If~$s = 0$, then the summand is of the form~$\tfrac{1}{t^k} T_t^* c (Y_1,\dots,Y_q)$. The~$Y_k$ are polynomial vector fields, so we may use multilinearity to decompose this term so that we get terms of~$\tfrac{1}{t^k} T_t^* c~$ whose arguments are homogeneous polynomial vector fields. In every such summand,~$\tfrac{1}{t^k} T_t^* c$ can be replaced by~$t^{\Sigma-k} c$, where~$\Sigma$ is the sum of the degrees of inserted homogeneous vector fields. By assumption on~$c$, every summand where~$\Sigma < k$ must vanish. 
This implies that as~$t \to 0$, the term~$\tfrac{1}{t^k} T_t^* c (Y_1,\dots,Y_k)$ converges to a finite value. This concludes the proof.
\end{proof}
\subsection{Gelfand--Fuks cohomology of Euclidean space}
Let us now connect the Lie algebra~$\mathfrak{X}(\mathbb{R}^n)$ to the Lie algebra of formal vector fields~$W_n$. 
To this end, fix, for the rest of the section some local frame of vector fields around~$0 \in \mathbb{R}^n$ to induce an isomorphism~$J^\infty_0 \mathfrak{X}(\mathbb{R}^n) \stackrel{\sim}{\to} W_n$, cf. Remark~\ref{RemarkTransformationOfWn}, so that every elements of~$W_n$ can be written as the infinity-jet~$j^\infty_0 X$ at zero of some~$X \in \mathfrak{X}(\mathbb{R}^n)$.
\begin{definition}
\label{def:QuasiIsomorphismFromFiltrationToWn}
If~$X \in \mathfrak{X}(\mathbb{R}^n)$, denote by~$\tilde{X}^{(r)}$ the polynomial vector field corresponding to the~$r$-jet of~$X$ at zero. Define then for all~$k \in \mathbb{Z}$ the maps
\begin{gather}
\begin{gathered}
\gamma_k : 
F^k C^\bullet (\mathfrak{X}(\mathbb{R}^n)) 
\to 
C^\bullet_{(k)}(W_n), 
\\
(\gamma_k c) (j^\infty_0 X_1, \dots, j^\infty_0 X_q) 
:= 
\lim_{r \to \infty}
\lim_{t \to 0}
t^{-k} (T_t^* c')
(\tilde{X}^{(r)}_1 ,\dots, \tilde{X}^{(r)}_q ), 
\end{gathered}
\\
\beta_k : 
C^\bullet_{(k)}(W_n) 
\to F^{k} 
C^\bullet(\mathfrak{X}(\mathbb{R}^n)),
\quad
(\beta_k c')
(Y_1,\dots,Y_k) 
:= 
c' 
\left(j^\infty_0 Y_1,\dots, j^\infty_0 Y_k\right).
\end{gather}
for all~$X_1,\dots,X_q \in W_n,$~$Y_1,\dots,Y_q \in \mathfrak{X}(\mathbb{R}^n)$,~$c \in F^k C^q(\mathfrak{X}(\mathbb{R}^n))$ and~$c' \in C^q_{(k)}(W_n)$.
\end{definition}
\begin{lemma}
\label{lem:GammakAndBetakChainMaps}
The maps~$\beta_k$ and~$\gamma_k$ are well-defined chain maps with~$\gamma_k \circ \beta_k = \id$.
\end{lemma}
\begin{proof}
Note first that~$\gamma_k$ is well-defined: By definition of the filtration, for every~$c \in F^k C^\bullet(\mathfrak{X}(\mathbb{R}^n))$ the pointwise limit~$\lim_{t \to 0} t^{-k} T_t^* c$ exists. Further, the sequence 
\begin{align}
\left(
\lim_{t \to 0}
t^{-k} (T_t^* c)
(\tilde{X}^{(r)}_1 ,\dots, \tilde{X}^{(r)}_q )
\right)_{r \in \mathbb{N}}
\end{align}
is eventually constant in~$r$, since the cochain~$\lim_{t \to 0} T_t^* c$ vanishes on homogeneous vector fields whose sum of degrees is larger than~$k$, cf.~\eqref{eq:ScalingOnHomogeneousVectorFields}.
Further, if~$c \in F^k C^q(\mathfrak{X}(\mathbb{R}^n))$, then by Lemma~\ref{LemmaVecFieldsRnFiltrationCharacterization},~$c$ vanishes on polynomial vector fields the sum of whose degrees is smaller than~$k$. Hence~$\gamma_k c \in C^q_{(k)}(W_n)$.\\
Analogously, if~$c \in C^k_{(r)}(W_n)$, then~$\beta_k c$ vanishes on polynomial vector fields whose sum of degrees is smaller than~$k$, hence Lemma~\ref{LemmaVecFieldsRnFiltrationCharacterization} implies~$\beta_k c \in F^k C^q(\mathfrak{X}(\mathbb{R}^n))$.
\\
The identification of a finite formal vector field~$X_n$ with its Taylor polynomial in~$\mathfrak{X}_\mathrm{pol}(\mathbb{R}^n)$ is a Lie algebra morphism, and so is the pullback of a vector field by the diffeomorphism~$T_t$. Hence~$\gamma_k$ is a chain map.
\\
The map~$\beta_k$ is a chain map since taking the infinite jet of a vector field at zero is a Lie algebra morphism~$\mathfrak{X}(\mathbb{R}^n) \to W_n$.
\\
It remains to check the composition of the two maps. Let us prove
\begin{align}
(\gamma_k \beta_k c)(j^\infty_0 X_1,\dots,j^\infty_0 X_q)
=
c(j^\infty_0 X_1,\dots,j^\infty_0 X_q)
\end{align}
for all $c \in C^q_{(k)}(W_n)$ and homogeneous formal vector fields~$j^\infty_0 X_1,\dots,j^\infty_0 X_q \in W_n$ whose degree sums up to~$k$. This suffices, since we have~$c, \gamma_k \beta_k c \in C^q_{(k)}(W_n)$, so both cochains vanish on all other homogeneous vector fields. 
\\
For a homogeneous vector field $X$ and sufficiently large~$r$, we have
\begin{align}
j^\infty_0 \tilde{X}^{(r)}
=
j^\infty_0 X 
\text{ and }
\deg \tilde{X}^{(r)}
=
\deg j^\infty_0 X.
\end{align}
Hence, since if~$j^\infty_0 X_1,\dots,j^\infty_0 X_q \in W_n$ are homogeneous formal vector fields whose degree sums up to~$k$, then
\begin{equation}
\begin{aligned}
(\gamma_k \beta_k c)(j^\infty_0 X_1,\dots,j^\infty_0 X_q)
&=
\lim_{r \to \infty}
c 
(j^\infty_0 \tilde{X}^{(r)}_1,\dots,j^\infty_0 \tilde{X}^{(r)}_q ) 
\\
&= 
c(j^\infty_0 X_1,\dots,j^\infty_0 X_q).
\end{aligned}
\end{equation}
This concludes the proof.
\end{proof}
\begin{corollary}
\label{cor:FiltrationRnExactSequence}
For every~$k \in \mathbb{Z}$,
\begin{align}
0 \to 
F^{k+1} C^\bullet(\mathfrak{X}(\mathbb{R}^n))
\hookrightarrow
F^{k} C^\bullet(\mathfrak{X}(\mathbb{R}^n))
\stackrel{\gamma_k}{\to}
C^\bullet_{(k)}(W_n)
\to 
0
\end{align}
is a split exact sequence of cochain complexes.
\end{corollary}
\begin{proof}
Because $\{F^{k} C^\bullet(\mathfrak{X}(\mathbb{R}^n))\}_{k \in \mathbb{Z}}$ constitutes a filtration of~$C^\bullet(\mathfrak{X}(\mathbb{R}^n))$, the inclusion 
$F^{k+1} C^\bullet(\mathfrak{X}(\mathbb{R}^n))
\to
F^{k} C^\bullet(\mathfrak{X}(\mathbb{R}^n))$
is a chain map. The map~$\gamma_k$ is a chain map due to Lemma~\ref{lem:GammakAndBetakChainMaps}. Hence, the sequence is a sequence of chain complexes.
The injectivity of the first map is clear, and the second map is surjective, since it is split by~$\beta_k$. Exactness at the middle term follows from the characterization of the filtration in Lemma~\ref{LemmaVecFieldsRnFiltrationCharacterization}. This concludes the proof.
\end{proof}
For the following lemma, recall the action of a Lie algebra on its Chevalley-Eilenberg cochains and the interior product $\intprod$ defined in Definition~\ref{DefinitionLieAlgebraActionOnCochains}.
\begin{lemma}
\label{lem:DiffeomorphismsInducedHomotopies}
Let $M$ be a smooth manifold, let~$\{\phi_t \}_{t > 0}$ a one-parameter semigroup of diffeomorphisms on~$M$, and~$\{X_t \in \mathfrak{X}(M)\}_{t > 0}$ its time-dependent generator. Then, for all~$t_0,t_1 > 0$ and~$c \in C^\bullet(\mathfrak{X}(M))$ we have
\begin{align}
\phi_{t_1}^* c - \phi_{t_0}^* c
=
K_{t_1,t_0} d c + d K_{t_1,t_0} c,
\end{align}
where~$K_{t_1,t_0} := -\int_{t_0}^{t_1} \phi_t^* (X_t \intprod c) dt$. 
\end{lemma}
\begin{proof}
By definition of the generator~$\{X_t\}_{t \in \mathbb{R}}$, we have for every~$Y \in \mathfrak{X}(\mathbb{R}^n)$
\begin{align}
\frac{d}{dt} \phi_t^* Y  = \phi_t^* [X_t,Y] \quad \forall t \in \mathbb{R}.
\end{align}
By standard differentiation rules and the homotopy formula from Lemma~\ref{LemmaInteriorProductHomotopy} we have for~$c \in C^q(\mathfrak{X}(M))$ and~$Y_1,\dots,Y_q \in \mathfrak{X}(\mathbb{R}^n)$:
\begin{equation}
\begin{aligned}
\frac{d}{dt} \phi^*_t c (X_1,\dots,X_q)
&=
\sum_{k=1}^q
(\phi_t^* c) (X_1, \dots, [X_t,X_k], \dots,X_q)
\\
&=
- (X_t \cdot \phi_t^* c)(X_1,\dots,X_q)
=
- d (X_t \intprod \phi_t^* c) - X_t \intprod (\phi_t^* dc).
\end{aligned}
\end{equation}
An application of the fundamental theorem of calculus now gives the desired statement.
\end{proof}
\begin{corollary}
\label{CorollaryF1ForFiltrationIsAcyclic}
The complex~$F^1 C^\bullet(\mathfrak{X}(\mathbb{R}^n))$ is acyclic.
\end{corollary}
\begin{proof}
Consider for every~$t_0,t_1 > 0$ the operator~$K_{t_0,t_1} : C^\bullet(\mathfrak{X}(\mathbb{R}^n)) \to C^\bullet(\mathfrak{X}(\mathbb{R}^n))$ from Lemma~\ref{lem:DiffeomorphismsInducedHomotopies} associated to the one-parameter semigroup~$\{T_t\}_{t > 0}$. By definition, for~$c \in F^1 C^k(\mathfrak{X}(\mathbb{R}^n))$ and all all~$X_1,\dots,X_k \in \mathfrak{X}(\mathbb{R}^n)$, we have 
\begin{align}
\lim_{t \to 0} (T_t^* c)(X_1,\dots,X_k) = 0.
\end{align}
But then
\begin{equation}
\begin{aligned}
c(X_1,\dots,X_k) &=  \lim_{t \to 0} (T_1^* c - T_t^* c)(X_1,\dots, X_k)
\\
&=
\lim_{t \to 0} \left(
K_{1,t} d c + d K_{1,t} c \right) (X_1,\dots,X_k)
=:
(K dc + dKc)(X_1,\dots,X_k).
\end{aligned}
\end{equation}
In the above, we defined  the operator~$K : C^\bullet(\mathfrak{X}(\mathbb{R}^n) \to C^\bullet(\mathfrak{X}(\mathbb{R}^n))$ as the pointwise limit of the operators~$K_{1,t}$:
\begin{align}
Kc = \lim_{t \to 0} K_{1,t}c = \int_0^1 T_t^* (X_t \intprod c) \dif t,
\end{align}
where~$X_t$ is the generator of the semigroup of diffeomorphisms~$\{T_t\}_{t > 0}$. Hence~$K$ is a chain homotopy between the identity map and zero on~$F^1 C^\bullet(\mathfrak{X}(\mathbb{R}^n))$, which proves the statement.
\end{proof}
Finally we can state a variation of Lemma 1 in Section 2.4.B. of \cite{fuks1984cohomology}:
\begin{theorem}
\label{TheoremIsomorphismLocalToInfinitesimal}
The inclusion
\begin{align}
F^{0} C^\bullet(\mathfrak{X}(\mathbb{R}^n)) 
&\hookrightarrow 
C^\bullet(\mathfrak{X}(\mathbb{R}^n))
\end{align}
and the maps
\begin{align}
\gamma_0 : F^0 C^q (\mathfrak{X}(\mathbb{R}^n)) \to C^q_{(0)}(W_n), \quad
\beta_0 : C^q_{(0)}(W_n) \to F^{0} C^q(\mathfrak{X}(\mathbb{R}^n))
\end{align}
from Definition~\ref{def:QuasiIsomorphismFromFiltrationToWn} are quasi-isomorphisms and unital algebra morphisms.
\\
In particular, there is an isomorphism of algebras
\begin{align}
H^\bullet(\mathfrak{X}(\mathbb{R}^n)) 
\cong 
H^\bullet(W_n),
\end{align}
and the wedge product of two elements of positive degree in~$H^\bullet(\mathfrak{X}(\mathbb{R}^n))$ is zero.
\end{theorem}
\begin{proof}
The compatibility of the maps with the algebra structure is immediate from the multiplicativity of the filtration~$F^k C^q (\mathfrak{X}(\mathbb{R}^n))$ and the formulas for~$\gamma_0, \beta_0$.
By the split exact sequence of Corollary~\ref{cor:FiltrationRnExactSequence}, we have for every~$k \in \mathbb{Z}$ an isomorphism of cochain complexes 
\begin{align}
\label{eq:DirectSumSplittingOfFiltration}
F^k C^\bullet(\mathfrak{X}(\mathbb{R}^n))
\cong
F^{k+1}  C^\bullet(\mathfrak{X}(\mathbb{R}^n))
\oplus
C^\bullet_{(k)}(W_n).
\end{align}
Inserting~$k = 0$ into \eqref{eq:DirectSumSplittingOfFiltration}, the acyclicity of~$F^{1} C^\bullet(\mathfrak{X}(\mathbb{R}^n))$ by Corollary~\ref{CorollaryF1ForFiltrationIsAcyclic} shows that we have algebra isomorphisms
\begin{align}
H^\bullet(F^{0} \tilde{C}^\bullet(\mathfrak{X}(\mathbb{R}^n))) 
\cong 
H^\bullet_{(0)}(W_n) \stackrel{\text{Prop.~\ref{PropositionWnCohmologyInZeroethDegree}}}{\cong} 
H^\bullet(W_n),
\end{align}
hence~$\gamma_0$ and its splitting~$\beta_0$ must be quasi-isomorphisms.
\\
Let further~$k = -1, \dots, -n$ in \eqref{eq:DirectSumSplittingOfFiltration}. By Proposition~\ref{PropositionWnCohmologyInZeroethDegree}, the complexes~$\tilde{C}^\bullet_{(k)}(W_n)$ are acyclic and we have isomorphisms
\begin{equation}
\begin{aligned}
H^\bullet(F^{0} C^\bullet(\mathfrak{X}(\mathbb{R}^n)))
&\cong
H^\bullet(F^{-1} C^\bullet(\mathfrak{X}(\mathbb{R}^n)))
\\
&\cong
\dots
\\
&\cong
H^\bullet(F^{-n} C^\bullet(\mathfrak{X}(\mathbb{R}^n))) 
\\
&
\stackrel{\text{Lem.~\ref{LemmaPropertiesOfRnFiltration}}}{=}
H^\bullet(C^\bullet(\mathfrak{X}(\mathbb{R}^n))) 
=
H^\bullet(\mathfrak{X}(\mathbb{R}^n)).
\end{aligned}
\end{equation}
Hence the inclusion~$F^{0} C^\bullet(\mathfrak{X}(\mathbb{R}^n)) 
\hookrightarrow 
C^\bullet(\mathfrak{X}(\mathbb{R}^n))$ is a quasi-isomorphism.
\\
The product of positive degree elements in~$H^\bullet(\mathfrak{X}(\mathbb{R}^n))$ is zero since this is true for~$H^\bullet(W_n)$ by Corollary~\ref{CorollaryWedgeProductOfFormalsIsZero}.
\end{proof}

We end this section by an extension of our proof of~$H^\bullet(\mathfrak{X}(\mathbb{R}^n)) \cong H^\bullet(W_n)$ to the setting of Gelfand--Fuks cohomology of a disjoint union of finitely many copies of Euclidean space. 
\begin{remark}
 The isomorphism of topological Lie algebras
\begin{align}
\mathfrak{X}(\mathbb{R}^n \sqcup \mathbb{R}^n) \cong \mathfrak{X}(\mathbb{R}^n) \oplus \mathfrak{X}(\mathbb{R}^n),
\end{align}
insinuates that the Gelfand--Fuks cohomology of such a disjoint union may be calculated by the use of a \emph{Künneth formula}: For finite-dimensional Lie algebras~$\frakg, \frakh$ over~$\mathbb{R}$, the Künneth theorem implies
\begin{align}
H^\bullet(\frakg \oplus \frakh) \cong H^\bullet(\frakg) \otimes H^\bullet(\frakh).
\end{align}
And indeed, such Künneth theorems in Lie algebra cohomology are well-known in the purely algebraic setting, but extending them to \emph{continuous} Lie algebra cohomology relies on nontrivial topological assumptions in order to deal with the arising topological tensor products and their completion. For such formulas, the reader may, for example, consult \cite{gourdeau2005kuenneth}. We avoid this approach here.
\end{remark}
\begin{proposition}
\label{PropositionCohomologyOfDisjointUnionsOfRns}
Let~$M := \bigsqcup_{i=1}^r  \mathbb{R}^n$ be a disjoint collection of copies of~$\mathbb{R}^n$. Then every choice of order on the copies of~$\mathbb{R}^n$ induces an algebra isomorphism
\begin{align}
H^\bullet
\left(
\mathfrak{X}\left( M
\right)
\right) 
\cong 
H^\bullet(\mathfrak{X}(\mathbb{R}^n))^{\otimes^r}.
\end{align}
\end{proposition}
\begin{proof}
We mimic the proof for the~$r = 1$ situation, but we expand the scaling of~$\mathbb{R}^n$ to the same scaling in every copy of~$\mathbb{R}^n$:
\begin{align}
T_t:
\bigsqcup_{i=1}^r  \mathbb{R}^n 
\to 
\bigsqcup_{i=1}^r  \mathbb{R}^n, 
\quad
x \mapsto tx.
\end{align}
The definition of the corresponding spaces~$F^k C^q(\mathfrak{X}(M))$ is identical to in the~$r = 1$ case, and by the same proofs, they constitute a descending, filtration that is bounded from below with
\begin{align}
F^k C^\bullet(\mathfrak{X}(M)) = C^\bullet(\mathfrak{X}(M)) \quad \forall k \leq - r n.
\end{align}
In analogy to Definition~\ref{def:QuasiIsomorphismFromFiltrationToWn}, we can define a map~$\gamma_k^{(r)}$, which, together with some choice of ordering on the components of~$M$ gives rise to an exact sequence for every~$k \in \mathbb{Z}$:
\begin{equation}
\begin{aligned}
0 \to 
F^{k+1} C^\bullet(\mathfrak{X}(\mathbb{R}^n))
&\to
F^{k} C^\bullet(\mathfrak{X}(\mathbb{R}^n))
\\
&\stackrel{\gamma_k^{(r)}}{\to}
\bigoplus_{k_1 + \dots + k_r = k}
C^\bullet_{(k_1)}(W_n)
\otimes
\dots
\otimes
C^\bullet_{(k_r)}(W_n) 
\to 
0.
\end{aligned}
\end{equation}
For the tensor product complex on the right hand side, we can use the Künneth theorem to calculate its cohomology, and due to acyclicity of~$C^\bullet_{(k)}(W_n)$ for~$k \neq 0$, the only one of the complexes with nontrivial cohomology is the one with the condition~$k_1 + \dots + k_r = 0$. By the same steps as in Corollary~\ref{CorollaryF1ForFiltrationIsAcyclic} and Theorem~\ref{TheoremIsomorphismLocalToInfinitesimal} we arrive at the desired isomorphism of vector spaces.
This isomorphism respects the wedge product, as we see with the arising quasi-isomorphism
\begin{equation}
\beta_0^{(r)} :
C^\bullet_{(0)}(W_n)^{\otimes^r}
\to
C^\bullet(\mathfrak{X}(M)), 
\quad
c_1 \otimes \dots \otimes c_r \mapsto \beta_0^{1} c_1 \wedge \dots \wedge \beta_0^{r} c_r,
\end{equation} 
where~$\beta_0^{k}$ maps formal cochains exactly like the map~$\beta_0$ from the Definition~\ref{def:QuasiIsomorphismFromFiltrationToWn}, but all jets of vector fields are evaluated at the zero in the~$k$-th copy of~$\mathbb{R}^n$. Because the~$\beta_0$ in the~$r = 1$ case respect the wedge product, so does~$\beta_0^{(r)}$.
\end{proof}
The formula for the quasi-isomorphism~$\beta_0^{(r)}$ from the previous proof implies:
\begin{corollary}
Let~$B_1, \dots, B_r \subset \mathbb{R}^n$ be pairwise disjoint sets diffeomorphic to~$\mathbb{R}^n$. Assume their union is contained in another set~$C \subset \mathbb{R}^n$ diffeomorphic to~$\mathbb{R}^n$. The extension map
\begin{align}
[\iota_{B_1 \cup \dots \cup B_r}^C] : 
\bigotimes_{i=1}^r
H^\bullet(\mathfrak{X}(B_i))
\cong
H^\bullet
\left(
\mathfrak{X}\left( 
B_1 \cup \dots \cup B_r
\right)
\right)
\to 
H^\bullet(\mathfrak{X}(C))
\end{align}
is given by
\begin{align}
[c_1] \otimes \dots \otimes [c_r]
\mapsto
[\iota_{B_i}^{C} c_1 \wedge  \dots \wedge \iota_{B_r}^C c_r].
\end{align}
\end{corollary}

\section{Cosheaf-theoretic aspects of Gel\-fand\--Fuks co\-ho\-mo\-lo\-gy}
\label{sec:CosheafAspectsOfLocalGelfandFuks}
The previous section concludes the analysis of the cohomology of~$\mathfrak{X}(\mathbb{R}^n)$. This constitutes an important building block to understand the Gelfand--Fuks cohomology for arbitrary smooth manifolds~$M$. We will explore some properties related to \emph{extension} of the cochains from a smaller to a larger open set of~$M$. The following presentation of these results, especially the framing in terms of cosheaf-theoretic data, is a novel contribution to the literature, though the results themselves are implicitly used in \cite{gel1969cohomologyI, gel1970cohomologyII, fuks1984cohomology}. We also remark a close similarity of these methods to the standard constructions in the theory of factorization algebras \cite{costello2016factorization}.

Recall that we identify~$J^\infty_0 \mathfrak{X}(\mathbb{R}^n) \cong W_n$ by a choice of local frame of vector fields around 0. The group of local diffeomorphisms~$\phi : \mathbb{R}^n \to \mathbb{R}^n$ that fix zero admits a right action on infinity-jets~$j^\infty_0 X \in J^\infty_0 \mathfrak{X}(\mathbb{R}^n)$ via the pullback of vector fields:
\begin{align}
\phi^* j^\infty_0 X := j^\infty_0( \phi^* X).
\end{align}
This action factors through to an action of the group of infinity-jets of diffeomorphisms that fix zero, denoted~$J^\infty_0 \Diff(\mathbb{R}^n)$.
Hence,~$J^\infty_0 \Diff(\mathbb{R}^n)$ acts on~$W_n$ and by pullback on the complex~$C^\bullet(W_n)$, and we write for all~$\phi \in \Diff(\mathbb{R}^n)$, all~$c \in C^k(W_n)$ and~$X_1,\dots,X_k \in \mathfrak{X}(\mathbb{R}^n)$:
\begin{align}
((j^\infty_0 \phi)^* c)(j^\infty_0 X_1,\dots,j^\infty_0 X_k) :=
c(j^\infty_0 (\phi^* X_1), \dots, j^\infty_0 (\phi^ * X_k)).
\end{align}
%
%
\begin{lemma}
\label{lem:LocalDiffeomorphismsUniqueUpToOrientation}
Let~$\phi: \mathbb{R}^n \to \mathbb{R}^n$ be a local diffeomorphism. 
\begin{itemize}
\item[i)] The induced map~$[\phi^*] : H^\bullet(\mathfrak{X}(\mathbb{R}^n)) \to H^\bullet(\mathfrak{X}(\mathbb{R}^n))$ is a unital algebra isomorphism.
\item[ii)] Assume~$\phi$ is additionally orientation-preserving with respect to some fixed orientation on~$\mathbb{R}^n$. Then~$[\phi^*] = \id$.
\end{itemize}

\end{lemma}
\begin{proof}
Lemma~\ref{lem:DiffeomorphismsInducedHomotopies} shows that the maps induced on~$C^\bullet(\mathfrak{X}(\mathbb{R}^n))$ by translations
\begin{align}
\tau_a : \mathbb{R}^n \to \mathbb{R}^n, \quad x \mapsto x + a, \quad a \in \mathbb{R}^n
\end{align}
are homotopic to the identity~$\id = \tau_0^*$. Hence on cohomology~$\tau_a$ acts as the identity for all~$a \in \mathbb{R}^n$. Since all~$\tau_a$ are orientation-preserving, we may without loss of generality assume that, in all that follows,~$\phi$ fixes zero.\\
\textit{i)}
Recall the map~$\beta_0 : C^\bullet(W_n) \to C^\bullet(\mathfrak{X}(\mathbb{R}^n))$ from Definition~\ref{def:QuasiIsomorphismFromFiltrationToWn}. By Theorem~\ref{TheoremIsomorphismLocalToInfinitesimal},~$\beta_0$ is a quasi-isomorphism, so every cohomology class in~$H^\bullet(\mathfrak{X}(\mathbb{R}^n))$ has a representative of the form~$\beta_0 c$ for some~$c \in C^\bullet(W_n)$. By definition of~$\beta_0$, it intertwines the action of the group of local diffeomorphisms that fix zero on~$C^\bullet(\mathfrak{X}(\mathbb{R}^n)$ with the action of~$J^\infty_0 \Diff(\mathbb{R}^n)$ on~$C^\bullet(W_n)$, meaning
\begin{align}
\phi^* (\beta_0) c = \beta_0((j^\infty_0 \phi)^* c).
\end{align}
Since~$\phi$ is a local diffeomorphism, it admits a local inverse around 0, so the action of~$j^\infty_0 \phi$ on~$c$ is invertible. Thus~$\phi^* : H^\bullet(\mathfrak{X}(\mathbb{R}^n)) \to H^\bullet(\mathfrak{X}(\mathbb{R}^n))$ is an isomorphism of vector spaces. 
\\
The conditions~$\phi^*(c_1 \wedge c_2) = \phi^* c_1 \wedge \phi^* c_2$ for all~$c_1,c_2 \in C^\bullet(\mathfrak{X}(\mathbb{R}^n))$ and~$\phi^* (1) = 1$ follow directly from the definition of the pullback of local diffeomorphisms. Hence~$\phi$ induces a unital algebra isomorphism.\\
\textit{ii)}
Fix an arbitrary cocycle~$c \in C^\bullet(W_n)$.  
Since~$c$ is a continuous cochain, it is zero on formal vector fields of sufficiently high degree. Hence, for some~$N \in \mathbb{N}$, it factors through to a multilinear map 
\begin{align}
c_N : \left( \frac{W_n}{\bigoplus_{r=N + 1}^\infty \frakg_r} \right)^k \to \mathbb{R}.
\end{align} 
The action of~$J^\infty_0 \Diff(\mathbb{R}^n)$ on~$W_n$ descends to an action of the Lie group~$J^N_0 \Diff(\mathbb{R}^n)$ on the quotient~$\frac{W_n}{\bigoplus_{r=N + 1}^\infty \frakg_r}$, so that so that for all~$X_1,\dots,X_k \in \mathfrak{X}(\mathbb{R}^n)$
\begin{align}
(j^\infty_0 \phi^* c)
(j^\infty_0 X_1,\dots, j^\infty_0 X_k) 
= 
(j^N_0 \phi^* c_N)(j^N_0 X_1,\dots, j^N_0 X_k),
\end{align}
where the action on the right-hand side is defined analogously to the action of the in\-fi\-ni\-ty\--jets. 
The Lie algebra of~$J^N_0 \Diff(\mathbb{R}^n)$ is given by
\begin{align}
J^N_0 \mathfrak{X}(\mathbb{R}^n) 
\cong 
\frac{\bigoplus_{r=0}^\infty \frakg_r}{\bigoplus_{r=N + 1}^\infty \frakg_r}
\subset
\frac{W_n}{\bigoplus_{r=N + 1}^\infty \frakg_r}.
\end{align}
Since~$\phi$ is orientation-preserving, its~$N$-jet at zero lies in the identity component of the Lie group~$J_0^N \Diff(\mathbb{R}^n)$. This component is generated by the image of the exponential map, there exist vector fields $Y_1,\dots,Y_r \in \mathfrak{X}(\mathbb{R}^n)$ so that for all~$X_1,\dots,X_k \in \mathfrak{X}(\mathbb{R}^n)$
\begin{equation}
\begin{aligned}
(j^\infty_0 \phi^* c ) (j^\infty_0 X_1,\dots, j^\infty_0 X_k) 
&= 
(j^N_0 \phi^* c_N ) (j^N_0 X_1,\dots, j^N_0 X_k) 
\\
&= 
(\exp(j^N_0 Y_1)^* \dotsm \exp(j^N_0 Y_r)^* c_N) (j^N_0 X_1,\dots, j^N_0 X_k) 
\\
&=
(\exp(j^\infty_0 Y_1)^* \dotsm \exp(j^\infty_0 Y_r)^* c) (j^\infty_0 X_1,\dots, j^\infty_0 X_k).
\end{aligned}
\end{equation}
The action of~$J^\infty_0 \mathfrak{X}(\mathbb{R}^n) \cong W_n$ on~$H^\bullet(W_n)$ is trivial by Corollary~\ref{CorollaryLieAlgebraActsTriviallyOnCohomology}, and as a consequence 
\begin{align}
[\phi^*]([\beta_0 c]) 
=
[\beta_0( \exp(j^\infty_0 Y_1)^* \dotsm \exp(j^\infty_0 Y_r)^* c)] 
= [\beta_0 c].
\end{align}
Hence~$[\phi^*] = \id$ as a map on cohomology, and the statement is shown.
\end{proof}
We will now make use of the language of \emph{cosheaves} to describe the extension of Gel\-fand\--Fuks cochains between open sets of a smooth manifold. While sheaf theory is well known, the dual concept of cosheaves is less commonly considered. For self\--con\-tained\-ness, we direct the reader to Appendix~\ref{AppendixCosheaves}, or \cite{bredon2012sheaf} for a more detailed study of both sheaf and cosheaf theory.
\begin{definition}
\label{DefinitionPresheafAndPrecosheafStructure}
Let~$M$ be a smooth manifold, and~$U \subset V$ open subsets of~$M$.
\begin{itemize}
\item[i)]
Define the \emph{extension of cochains}
\begin{align}
\iota_U^V : C^\bullet(\mathfrak{X}(U)) &\to C^\bullet(\mathfrak{X}(V))
\end{align}
on cochains of degree~$k > 0$ as
\begin{align}
(\iota_U^V c)(X_1,\dots,X_k) &:= c(X_1 \att_U, \dots, X_k \att_U)
\end{align}
for all~$c \in C^k(\mathfrak{X}(U))$,~$X_1,\dots,X_k \in \mathfrak{X}(V)$.\\ On cochains of degree 0, we set~$\iota_U^V : C^0(\mathfrak{X}(U)) \to C^0(\mathfrak{X}(V))$ to be the identity.
\item[ii)]
The extension of cochains induces an \emph{extension of cohomology classes}
\begin{align}
\iota_U^V : H^\bullet(\mathfrak{X}(U)) \to H^\bullet(\mathfrak{X}(V))
\end{align}
which we will denote with the same symbol~$\iota_U^V$ by an abuse of notation.
\item[iii)]
We define a precosheaf~$\mathcal{H}^\bullet$ of algebras over~$M$, assigning to an open set~$U \subset M$ the algebra~$H^\bullet(\mathfrak{X}(U))$, and to an inclusion of open sets~$U \subset V$ the extension map~$\iota_U^V : H^\bullet(\mathfrak{X}(U)) \to H^\bullet(\mathfrak{X}(V))$.
\end{itemize}
\end{definition}
An alternative perspective on the extension maps~$\iota_U^V$ is that they are the pullback along the inclusion map~$U \to V$, as defined in \eqref{eq:PullbackOfCochainsAlongLocalDiffeos}.
This map is a local diffeomorphism, hence this pullback is well-defined.
\\ 
For the next corollary, recall that a~$\mathrm{C}$-valued precosheaf on~$M$ is simply a covariant functor from the category~$\mathrm{Open}(M)$ of open sets of~$M$ to the category~$\mathrm{C}$. 
\begin{definition}
Let~$A$ be a associative~$\mathbb{R}$- algebra and~$X$ a locally connected topological space. We define the \emph{constant precosheaf of algebras} associated to~$A$ as the assignment
\begin{align}
U \mapsto A^{\otimes \pi_0(U)}.
\end{align}
Here, if~$\pi_0(U) = \infty$, we set~$A^{\otimes \pi_0(U)}$ to be the infinite coproduct of algebras, i.e.
\begin{align}
A^{\otimes \pi_0(U)} := \varinjlim A_1 \otimes \dots \otimes A_n,
\end{align}
and the inclusion maps~$A_1 \otimes \dots \otimes A_n \to A_1 \otimes \dots \otimes A_m$ for~$n \leq m$ of this colimit are given by
\begin{align}
a_1 \otimes \dots \otimes a_n 
\mapsto 
a_1 \otimes \dots \otimes a_n \otimes 1 \otimes \dots \otimes 1.
\end{align}
The extension maps of the precosheaf are given by taking products along connected components.
\end{definition}
\begin{remark}
This is the categorical generalization of a constant cosheaf from the category of vector spaces, replacing the coproduct~$\oplus$ of vector spaces by the coproduct~$\otimes$ of algebras. However, the category of~$\mathbb{R}$-algebras is not even preadditive. Hence, we cannot define exact sequences, and hence cosheaves, in the way we did in Appendix~\ref{AppendixCosheaves}. This is why we refrain from calling this construction a cosheaf.
\end{remark}
\begin{corollary}
\label{cor:CosheafFunctorsEquivalentOnBalls}
Let~$M$ be a smooth, orientable manifold of dimension~$n$. Consider the full subcategory of~$\mathrm{Open}(M)$ defined by finite unions of pairwise disjoint sets diffeomorphic to~$\mathbb{R}^n$:
\begin{align}
\mathcal{B}_M :=
\{U_1 \cup \dots \cup U_k : k \in \mathbb{N}, U_i \cong \mathbb{R}^n, U_i \cap U_j = \emptyset \text{ for } i \neq j\}.
\end{align}
The functor associated to the restriction of the precosheaf~$\mathcal{H}^\bullet$ to this subcategory is naturally isomorphic to the analogous restriction of the functor associated to the constant cosheaf of algebras~$U \mapsto H^\bullet(\mathfrak{X}(\mathbb{R}^n))^{\otimes^{\pi_0(U)}}$.
\end{corollary}
\begin{proof}
Fix some orientation on~$\mathbb{R}^n$. Using orientability of~$M$, we can construct for every connected~$U \in \mathcal{B}_M$ a diffeomorphism~$\phi_{U} : U \to \mathbb{R}^n$ so that if~$V \in \mathcal{B}_M$ is any other connected set with~$U \cap V \neq \emptyset$, the transition map~$(\phi_{U} \circ \phi_{V}^{-1})|_{\phi_V(V \cap U)}$ is orientation-preserving.
Fix now some set~$U \in \mathcal{B}_M$ with connected components~$U_1,\dots,U_k$. In Proposition~\ref{PropositionCohomologyOfDisjointUnionsOfRns} we defined an isomorphism 
\begin{align}
\label{eq:TensorFactorizeCohomology}
H^\bullet(\mathfrak{X}(U_1)) 
\otimes
\dots 
\otimes
H^\bullet(\mathfrak{X}(U_k))
 \cong 
H^\bullet(\mathfrak{X}(U)).
\end{align}
Thus, under the identification \eqref{eq:TensorFactorizeCohomology} we can express every element in~$H^\bullet(\mathfrak{X}(U))$ as a span of elements of the form
\begin{align}
[\phi^*_{U_1} c_1] 
\otimes \dots \otimes 
[\phi^*_{U_k} c_k], 
\quad 
c_1,\dots,c_k \in C^\bullet(\mathfrak{X}(\mathbb{R}^n)).
\end{align}
Let us show that the collection of the maps 
\begin{align}
\label{eq:NaturalIsomorphism}
[\phi^*_{U_1}]
\otimes 
\dots 
\otimes 
[\phi^*_{U_k}]
: 
H^\bullet(\mathfrak{X}(\mathbb{R}^n))^{\pi_0(U)}
\to
H^\bullet(\mathfrak{X}(U)) 
\end{align}
for all~$U \in \mathcal{B}$ induce the desired natural isomorphism of functors. By Lemma~\ref{lem:LocalDiffeomorphismsUniqueUpToOrientation}, every tensor factor is an isomorphism of vector spaces, and by an explicit calculation with the wedge product, the maps are compatible with the algebra structures. It remains to show that the precosheaf extension maps are respected. 
\\
Let~$V \in \mathcal{B}_M$ contain~$U$. Assume first that two connected components of~$U$, say~$U_1, U_2$ lie in a single connected component of~$V$, say~$V_1$. 
By Corollary~\ref{CorollaryWedgeProductOfFormalsIsZero} the wedge product of~$H^\bullet(W_n)$ is zero in nonzero degree, so the extension map~
\begin{align}
H^{q_1}(\mathfrak{X}(U_1)) \otimes H^{q_2}(\mathfrak{X}(U_2)) \to H^{q_1 + q_2}(\mathfrak{X}(V_1))
\end{align}
is zero if~$q_1$ and~$q_2$ are simultaneously nonzero. 
\\
If one of the~$q_1,q_2$ is equal to zero, say~$q_1 = 0$, then the extension map~$H^{q_1}(\mathfrak{X}(U_1)) \otimes H^{q_2}(\mathfrak{X}(U_2)) \to H^{q_1 + q_2}(\mathfrak{X}(V_1))$ is simply the isomorphism
\begin{align}
\mathbb{R} \otimes H^{q_2}(\mathfrak{X}(U_2)) \xrightarrow{\sim} H^{q_2}(\mathfrak{X}(U_2)),
\quad
1 \otimes [c] \mapsto [c].
\end{align}
For the constant precosheaf of algebras~$U \mapsto H^\bullet( \mathfrak{X}(\mathbb{R}^n) )^{\otimes \pi_0(U)}$ analogous properties hold, and as a consequence the maps \eqref{eq:NaturalIsomorphism} commute with inclusions where two connected components of~$U$ are contained in a single connected component of~$V$.
\\
Hence it remains to consider the case in which every connected component~$V_i$ of~$V$ contains at most one connected component of~$U$, so by reordering we may assume~$U_i \subset V_i$ for all~$i = 1,\dots, k$. It suffices to show that~$(\phi_{V_i}^{-1} )^* \circ \iota_{U_i}^{V_i} \circ \phi_{U_i}^*$ induces the identity map on~$H^\bullet(\mathfrak{X}(\mathbb{R}^n))$ for all~$i = 1,\dots,k$. As maps on cochains~$C^\bullet(\mathfrak{X}(\mathbb{R}^n)$ we have the identity
 \begin{align}
(\phi_{V_i}^{-1})^* 
\circ 
\iota_{U_i}^{V_i}
\circ 
\phi_{U_i}^*
=
\left( \phi_{V_i}^{-1}|_{U_i} \circ \phi_{U_i} \right)^*.
\end{align} 
Now~$\phi_{V_i}^{-1}|_{U_i} \circ \phi_{U_i}$ is an orientation-preserving local diffeomorphism, so by Lemma~\ref{lem:LocalDiffeomorphismsUniqueUpToOrientation} the map 
$(\phi_{V_i}^{-1})^* 
\circ 
\iota_{U_i}^{V_i}
\circ 
\phi_{U_i}^*$ equals the identity on cohomology and~$[\iota_{U_i}^{V_i} \circ \phi_{U_i}] = [\phi_{V_i}]$. Hence, the statement is shown.
\end{proof}
\begin{remark}
In the non-orientable case, the previous proof shows that the restriction of~$\mathcal{H}^\bullet$ is naturally isomorphic to the restriction of a \emph{locally constant cosheaf}~$U \mapsto S(U)$, i.e. every point~$x \in M$ has an open neighborhood~$U_x$ so that the restriction of~$S$ to~$U_x$ is a constant cosheaf.
\end{remark}

\section{Gelfand--Fuks cohomology for smooth manifolds}
\label{SectionGelfandFuksOnSmooths} 
In this section, we construct a spectral sequence due to Gelfand and Fuks that calculates the continuous Lie algebra cohomology for smooth manifolds, following a local-to-global principle using sheaf theoretic ideas.
The spectral sequences were originally constructed in \cite{gel1969cohomologyI}, by an involved global analysis of the cochain spaces~$C^\bullet(\mathfrak{X}(M))$ in terms of explicit distributions.
\\
The proposed local-to-global principle has originally been outlined in \cite{bott1973lectures} and \cite{bott1975gelfand}, and, according to the last reference, was initially suggested by Segal.
However, in these latter two references, there are some unaddressed subtleties: It is (indirectly) claimed that the assignment of open sets~$U$ to~$C^\bullet(\mathfrak{X}(U))$ is a \emph{cosheaf} of graded vector spaces, i.e. its \v{C}ech homology (see Appendix~\ref{AppendixCosheaves}) vanishes with respect to every good cover~$\mathcal{U}$ of~$M$. 
We show in Proposition~\ref{prop:CochainsAreNotCosheaves} that this is false. We present a novel proof that works around this problem by using so-called~$k$-\emph{good} covers, an adaptation to the concept of a good cover originating from \cite{boavida2013manifold}. This was inspired by the recent preprint \cite{hennion2018gelfand}, treating Gelfand--Fuks cohomology in the setting of factorization algebras. 
\\
However, we want to emphasize that this subtlety does not influence the validity of the final results of Bott and Segal. The mistake is not repeated in \cite{bott1977cohomology} and \cite{fuks1984cohomology}, where similar, but more sophisticated \v{C}ech-theoretic methods are used.
\\
Regardless, our proof gives a more elementary way to calculate Gelfand--Fuks spectral sequences for \emph{$k$-diagonal cohomology}, an approximation of Gelfand--Fuks cohomology which we will introduce in the following section. The expression for the spectral sequences have been given in  \cite{fuks1984cohomology} without an explicit proof for~$k \neq 1$; the proof in \cite{bott1973lectures} is only a sketch, with previously mentioned issues, and the proof in \cite{gel1969cohomologyI} uses explicit, distribution-theoretic formulas which may be difficult for non-experts to retrace. 
\\
The construction we propose in the following section is inspired by a preprint by Kapranov and Hennion \cite{hennion2018gelfand}, in which they use the theory of factorization algebras to prove that Gelfand--Fuks cohomology can be identified with singular cohomology of a certain section space. We borrow from this theory a notion of generalized good covers which makes it possible to construct the desired spectral sequences from a local-to-global principle. This strategy is easily generalizable to construct local-to-global spectral sequences in other cohomology theories, as we plan to show in future work  \cite{miaskiwskyi2021loday}.

\subsection{Diagonal Filtration}
Fix a smooth manifold~$M$ of dimension~$n$.
The previously established precosheaf structure (see Definition~\ref{DefinitionPresheafAndPrecosheafStructure}) of the co\-chains~$C^\bullet(\mathfrak{X}(M))$ does \emph{not} extend to a cosheaf structure. 
%
\begin{proposition}
\label{prop:CochainsAreNotCosheaves}
Let~$M$ be a smooth manifold. If~$k > 1$ and~$\dim M > 0$, then the precosheaf~$U \mapsto C^k(\mathfrak{X}(U))$ from Definition~\ref{DefinitionPresheafAndPrecosheafStructure} is not a cosheaf.
\end{proposition}
\begin{proof}
Because~$\dim M > 0$, there are smooth, nonzero~$X_1, \dots, X_k \in \mathfrak{X}(M)$ whose supports are pairwise disjoint, and some~$c \in C^k(\mathfrak{X}(M))$ with~$c(X_1,\dots, X_k) \neq 0$.
Then the sets 
\begin{align}
U_i := M \setminus  \left( \bigcup_{j \neq i} \supp X_j \right), \quad i = 1,\dots, k
\end{align}
define an open cover~$\{U_i\}_{i = 1,\dots,k}$ of~$M$. If the assignment of an open set~$U \subset M$ to co\-chains~$C^k(\mathfrak{X}(U))$ was a cosheaf, then there would exist~$c_i \in C^\bullet(\mathfrak{X}(U_i))$ for~$i = 1, \dots, k$ with
\begin{align}
c = \sum_{i=1}^k \iota_{U_i}^M c_i.
\end{align} 
But then, because~$X_i \att_{U_j} = 0$ for~$i \neq j$, it follows that
\begin{align}
0 \neq c(X_1, \dots, X_k) = \sum_{i=1}^k c_i (X_1|_{U_i}, \dots, X_k|_{U_i}) = 0.
\end{align}
A clear contradiction, hence, the precosheaf~$U \mapsto C^k(\mathfrak{X}(U))$ is not a cosheaf for~$k > 1$.
\end{proof}
Hence, as we increase the number of arguments in our cochains, we may get \emph{locality} or \emph{diagonality problems} as in the above proof. It will be valuable to replace these spaces by certain diagonal replacements:
\begin{definition}
Let~$U$ be an open subset of~$M$.
\begin{itemize}
\item[i)] Define the graded vector space~$B^\bullet(\mathfrak{X}(U)) :=\bigoplus_{q \geq 0} B^q(\mathfrak{X}(U)),$ where
\begin{align}
B^q(\mathfrak{X}(U)) :=
\left \{ 
c : \mathfrak{X}(U)^q 
\to \mathbb{R} 
\, \mid \,
c \text{ multilinear and jointly continuous} \right \}.
\end{align}
\item[ii)] Let~$\{X_1,\dots,X_q \} \subset \mathfrak{X}(U)$ be a finite collection of vector fields and~$k \geq 1$ arbitrary. We say that this collection \emph{has the property~$\Delta_k$} if for every finite set~$\Gamma \subset U$ of~$k$ arbitrary points, there is an~$X_i$ that vanishes in a neighborhood of~$\Gamma$.
\item[iii)]
For~$q, k \geq 1$ integers, we define the \emph{$k$-diagonal distributions} as those~$c \in B^q(\mathfrak{X}(U))$ with
\begin{align}
\{X_1,\dots,X_q\} \text{ has property } \Delta_k \implies c(X_1,\dots,X_q) = 0.
\end{align}
Their collection is denoted~$\Delta_k B^q(\mathfrak{X}(U))$.
\\
If~$q = 0$, set~$\Delta_k B^0(\mathfrak{X}(U)) = B^0(\mathfrak{X}(U))$ for all~$k \geq 1$.
\item[iv)] Define the \emph{$k$-diagonal cochains}~$\Delta_k C^q(\mathfrak{X}(U)) \subset C^q(\mathfrak{X}(M))$ as the skew-symmetric cochains which are contained in~$\Delta_k B^q(\mathfrak{X}(U))$.
\end{itemize}
\end{definition}
\begin{proposition}
For all~$q \geq 0$ and all open~$U \subset M$, we have the ascending chain
\begin{equation}
\begin{aligned}
0 =: \Delta_0 C^q (\mathfrak{X}(U))
&\subset \Delta_1 C^q (\mathfrak{X}(U))
\subset \dots 
\\
&\subset \Delta_{q-1} C^q (\mathfrak{X}(U))
\subset \Delta_q C^q (\mathfrak{X}(U)) = 
C^k (\mathfrak{X}(U)).
\end{aligned}
\end{equation}
Further, the~$\Delta_k C^\bullet(\mathfrak{X}(U))$ constitute a multiplicative filtration of~$C^\bullet(\mathfrak{X}(U))$.
\end{proposition} 
\begin{proof}
The property~$\Delta_k$ for a set~$\{X_1,\dots,X_q\}$ implies the property~$\Delta_{k-1}$. This proves the ascending chain of inclusions.
\\
Further, a set~$\{X_1,\dots,X_k\}$ of~$k$ vector fields can only have the property~$\Delta_k$ if one of the~$X_i$ is zero everywhere. Hence~$\Delta_k C^k (\mathfrak{X}(M)) = C^k (\mathfrak{X}(M))$.
Further, notice that if~$\{X_1,\dots,X_{q+1}\}$ has the property~$\Delta_k$, so does~$ \{[X_1,X_2],X_3,\dots,X_{q}\}$. From this it follows that
\begin{align}
d(\Delta_k C^q(\mathfrak{X}(M)) 
\subseteq 
\Delta_k C^{q+1}\mathfrak{X}(M).
\end{align}
Lastly, if~$\{X_1,\dots,X_{q + r}\}$ has the property~$\Delta_{k + l}$, then the set~$\{X_1,\dots,X_q\}$ must have the property $\Delta_k$ or the set~$\{X_{q+1}, \dots, X_{q + r}\}$  must have the property~$\Delta_l$. Hence
\begin{align}
\Delta_k C^\bullet(\mathfrak{X}(M)) \wedge
\Delta_l C^\bullet(\mathfrak{X}(M))
\subseteq  \Delta_{k + l} C^\bullet(\mathfrak{X}(M)).
\end{align}
\end{proof}
\begin{example}
A set~$\{X_1,\dots,X_k\} \subset \mathfrak{X}(M)$ has the property~$\Delta_1$ if and only if
\begin{align}
\bigcap_{i=1}^k \supp X_i = \emptyset.
\end{align}
\end{example}
Note that~$\Delta_q C^q(\mathfrak{X}(U)) = C^q(\mathfrak{X}(U))$, 
hence
\begin{align}
\Delta_{k} H^q(\mathfrak{X}(U)) = H^q(\mathfrak{X}(U)) \quad \forall k \geq q + 1.
\end{align}
To put this in terms of more sheaflike data, let us view these cochains through a different lens.
\begin{definition}
Given~$q \geq 1$ and the canonical projections~$\pr_1, \dots, \pr_q : M^q \to M$, consider the vector bundle
\begin{align}
\boxtimes^q TM := \bigotimes^q_{i=1} \pr_i^* TM \to M^q.
\end{align}
Equipping the space of sections~$\mathfrak{X}(M)$ with its standard Fréchet topology, it is a well-known consequence of the Schwartz kernel theorem that there is a natural vector space isomorphism
\begin{align}
\label{eq:SchwartzKernel}
B^q(\mathfrak{X}(M))
\cong
\Gamma( \boxtimes^q TM)^*,
\end{align}
the star denoting the continuous dual with respect to the Fréchet topology (cf. \cite{hormander2003analysis} for the distributional statement). 
This isomorphism is dual to the map
\begin{equation}
\label{eq:SchwartzKernelDualMap}
\begin{gathered}
\mathfrak{X}(M) \otimes \dots \otimes \mathfrak{X}(M) \to \Gamma(\boxtimes^q TM), \quad
(X_1,\dots,X_q) \mapsto X_1 \boxtimes \dots \boxtimes X_q, \\
(X_1 \boxtimes \dots \boxtimes X_q)(x_1,\dots,x_q) := X_1(x_1) \otimes \dots \otimes X_q(x_q) \quad \forall x_1,\dots,x_q \in M.
\end{gathered}
\end{equation}
\end{definition}
\begin{definition}
\label{def:GeneralizedDiagonals}
Let~$X$ be a topological space and~$k,q \in \mathbb{N}$. The \emph{$k$-th diagonal} of~$X$ in~$X^q$ is the subspace
\begin{align}
X^q_k := \{(x_1,\dots,x_q) \in X^q :  |\{x_1,\dots,x_q\}| \leq k\}.
\end{align}
\end{definition}
Special examples are the \emph{thin diagonal}~$M^q_1$ and \emph{fat diagonal}~$M^q_{q-1}$, which take the following form:
\begin{gather*}
M_1^q:= \{(x,\dots,x) \in M^q \}, \\
M_{q-1}^q = \{(x_1,\dots,x_q) \in M^q \mid \exists i,j : i \neq j \text{ and } x_i = x_j \}.
\end{gather*}
Clearly~$M_1^q \subset M_2^q \subset \dots \subset M_q^q = M^q$.
Note that a set of vector fields~$\{X_1,\dots,X_q\}$ has the property~$\Delta_k$ if and only if the support of~$ X_1 \boxtimes \dots \boxtimes X_q$ is contained in~$M^q \setminus M^q_k$. This proves the following:
\begin{lemma}
\label{lem:DiagonalCochainsSchwartzKernel}
An element~$c \in B^q(\mathfrak{X}(M))$ is~$k$-diagonal if and only if the support of its image under the Schwartz kernel map in~$\Hom_{\mathbb{R}}(\Gamma(\boxtimes^q TM), \mathbb{R})$ is contained in~$M^q_k$.
\end{lemma}
With this perspective, we can deduce:
\begin{lemma}
\label{LemmaFineSheafOfDiagonalCochains}
For~$U \subset M^q$, the assignments
\begin{align}
M^q \supset U \mapsto 
\mathcal{B}^q(U) :=&
\Hom_{\mathbb{R}}(\Gamma( \boxtimes^q TM)\att_U, \mathbb{R}), \\
M^q_k \supset U \mapsto
\mathcal{B}^q_k(U) :=&
 \left \{ c \in 
\Hom_{\mathbb{R}}(\Gamma( \boxtimes^q TM), \mathbb{R})
:
\supp c \subset U
\right \},
\end{align}
constitute flabby cosheaves on~$M^q$ and~$M^q_k$, respectively, where the extension maps are induced by the restriction maps of the section spaces.
\end{lemma}
\begin{proof}
The sheaves of distributions
\begin{align}
M^q \supset U \mapsto 
\mathcal{D}^q(U) &:= 
\Hom_{\mathbb{R}}
\left(
\Gamma_c(\boxtimes^q TM) \att_U, \mathbb{R}
\right), \\
M^q_k \supset U \mapsto
\mathcal{D}^q_k(U) &:=
 \left \{ c \in 
\Hom_{\mathbb{R}}(\Gamma_c( \boxtimes^q TM), \mathbb{R})
:
\supp c \subset U
\right \}
\end{align}
are shown to be soft using standard partition of unity arguments. Alternatively, this follows since the first sheaf is a module over a soft sheaf of rings, namely the sheaf of smooth functions on~$M^q$, and the second one is a restriction of the first sheaf to a closed subspace, hence soft (see \cite[Chapter II, Thm 9.2, Thm 9.16]{bredon2012sheaf}).
But~$\mathcal{B}^q$ is exactly the precosheaf of compactly supported sections of the sheaf~$\mathcal{D}^q$, and analogously for~$\mathcal{B}^q_k$ and~$\mathcal{D}^q_k$. By Proposition~\ref{PropositionSoftSheavesGiveFlabbyCosheaves}, this implies that these precosheaves are flabby cosheaves.
\end{proof}

\subsection{Generalized good covers}
We have seen that the precosheaf~$U \mapsto C^k(\mathfrak{X}(U))$ for open sets~$U$ of a smooth manifold~$M$ does not define a cosheaf for~$k \geq 2$. However, Lemma~\ref{LemmaFineSheafOfDiagonalCochains} gives some hope that we can meaningfully study them over the Cartesian power~$M^k$. As such, we will need methods to compare different Cartesian powers~$M, M^2, M^3,\dots$ of~$M$. One such tool we can use is the notion of a~$k$-good cover in the sense of \cite[Def 2.9]{boavida2013manifold}:
\begin{definition}
\label{def:K-Good-Covers}
Let~$k \geq 1$. An open cover~$\mathcal{U}$ of~$M$ is \emph{$k$-good} if:
\begin{itemize}
\item[i)] Given~$k$ points~$x_1,\dots,x_k \in M$, there is a~$U \in \mathcal{U}$ with~$x_1,\dots,x_k \in U$.
\item[ii)] Intersections of elements of~$\mathcal{U}$ are diffeomorphic to a disjoint union of at most~$k$ copies of~$\mathbb{R}^n$.
\end{itemize}
\end{definition}
For~$k = 1$ this agrees with the usual notion of a good cover. 
\begin{remark}
The~$k$-good covers are, in a sense, finite approximations to so-called \emph{Weiss covers}, which have property i) of the previous definition with no restriction on the number~$k$, but without any replacement for property ii), so the sets in the cover may, a priori, be homologically wild. Weiss covers are heavily used in the theory of \emph{factorization algebras,} which have strong ties to our setting, see \cite{costello2016factorization,hennion2018gelfand,ginot2015notes}.
\end{remark}
\begin{definition}
\label{def:OpenCoversOfDiagonals}
Let~$X$ be a topological space,~$\mathcal{U}$ an open cover of~$X$, and~$k,q \geq 1$ integers. Define
\begin{gather*}
\mathcal{U}^q := \{V \subset M^q : V \text{ a connected component of } U^q \text{ for some } U \in \mathcal{U} \},\\
\mathcal{U}^q_{k} := \{V \cap M^q_k : V \in \mathcal{U}^q \}.
\end{gather*}
\end{definition}
Property i) of a~$k$-good cover~$\mathcal{U}$ implies that the sets~$\mathcal{U}^1, \mathcal{U}^2,\dots, \mathcal{U}^k$ are open covers of~$M,M^2,\dots,M^k$, making~$k$-good covers useful tools in comparing data between Cartesian powers of~$M$. 
\\
If~$\mathcal{U}$ is a~$k$-good cover, then the equation
\begin{align}
(U \cap V)^q = U^q \cap V^q
\end{align}
shows that all intersections of the open cover~$\mathcal{U}^q$ are diffeomorphic to finite disjoint unions of at most~$k$-copies of~$\mathbb{R}^{qn}$. To show a similar result for the open cover~$\mathcal{U}^q_k$ of~$M^q_k$, let us prepare an auxiliary statement:
\begin{lemma}
\label{lem:DiagonalsOfContractiblesAreContractible}
Let~$X$ be a topological space and~$q,k \geq 1$ integers. 
If~$X$ has finitely many connected components and all of them are contractible, then the same holds for~$X^q_k$.
\end{lemma}
\begin{proof}
Let~$X_1, \dots, X_s$ be the connected components of~$X$. Every connected component~$C$ of~$X^q$ is then of the form
\begin{align}
C = X_{i_1} \times \dots \times X_{i_q}	
\end{align} 
for some~$i_1,\dots,i_q \in \{1,\dots,s\}$, not necessarily all different.
By assumption, the~$X_1,\dots,X_s$ are contractible, hence, for all~$j=1,\dots,s$, there are deformation retracts~$F_j : X_j \times [0,1] \to X_j$ of~$X_j$ onto a point.
Then the map
\begin{align}
F : C \times [0,1] \to C,
\quad
(x_1,\dots,x_q ,t) \mapsto (F_{i_1}(x_1,t),\dots, F_{i_q}(x_q,t))
\end{align}
is a deformation retract of~$C$ onto a point.
\\
If~$C \cap X^q_k$ is nonempty, this map restricts to a deformation retract of~$C \cap X^q_k$ to a point. Hence, if~$C$ is a connected component of~$U^q$, and~$C \cap X^q_k \neq \emptyset$, then~$C \cap X^q_k$ is contractible.
But the connected components of~$X^q_k$ are exactly the nonempty sets~$C \cap X^q_k$ for connected components~$C$ of~$X^q$. Hence all connected components of~$X^q_k$ are contractible. Lastly, since~$X^q$ only has finitely many connected components, so does~$X^q_k$. This finishes the proof.
\end{proof}
\begin{lemma}
\label{LemmaPowersOfKCoverForPowersOfM}
Let~$X$ be a locally connected topological space,~$q,k \geq 1$ integers, and~$\mathcal{U}$ a k-good open cover of~$X$. Then~$\mathcal{U}^q_k$ is an open cover of~$X^q_k$.
Further, all nonempty, finite intersections of elements in~$\mathcal{U}^q_k$ have finitely many connected components, and all these connected components are contractible.
\end{lemma}
\begin{proof}
If~$(x_1,\dots,x_q) \in X^q_k$, then~$|\{x_1,\dots,x_q\}| \leq k$. But since~$\mathcal{U}$ is~$k$-good, there is some~$U \in \mathcal{U}$ containing all~$x_1,\dots,x_k$. Hence~$x \in U^q \cap X^q_k$. Since~$x$ was arbitrary, this shows that~$\mathcal{U}_k^q$ is an open cover of~$X^q_k$.
\\
Let us now show that if~$V, V' \in \mathcal{U}^q_k$ have nonempty intersection, then~$V \cap V'$ is a finite disjoint union of contractible sets. By definition there exist~$U, U' \in \mathcal{U}$ so that~$V$ and~$V'$ are connected components of~$U^q_k$ and~$U'^q_k$, respectively.
Since~$\mathcal{U}$ is a~$k$-good cover,~$U \cap U'$ has only finitely many connected components, and all these connected components are contractible. Thus by Lemma~\ref{lem:DiagonalsOfContractiblesAreContractible} the same holds for~$(U \cap U')^q_k = U^q_k \cap (U')^q_k$.
\\
Since~$X$ is locally connected, so are~$X^q_k$ and its open subsets~$U^q_k$ and~$(U')^q_k$. In a locally connected space, all connected components are closed and open, hence~$V$ and~$V'$ are closed and open in~$U^q_k$ and~$(U')^q_k$, respectively. Thus~$V \cap V'$ is closed and open in~$U^q_k \cap (U')^q_k$, and hence must be a union of connected components of~$U^q_k \cap (U')^q_k$. But we have seen that there are only finitely many connected components of~$U^q_k \cap (U')^q_k$, and that they are all contractible. This shows that the same holds for~$V \cap V'$.
\\
By induction this extends to arbitrary finite intersections of sets~$V_1,\dots,V_q \in \mathcal{U}^q_k$, and the lemma is shown.
\end{proof}
The first part of the following theorem is Proposition 2.10 in \cite{boavida2013manifold}:
\begin{theorem}
For every smooth manifold~$M$, a~$k$-good open cover exists.
Further, if~$M$ is compact, then~$M$ admits finite~$k$-good open covers.
\end{theorem}
\begin{proof}
The existence of~$k$-good open covers is shown in \cite{boavida2013manifold}. 
If~$M$ is compact, choose any~$k$-good cover~$\mathcal{U}$, then~$\mathcal{U}^k$ is a cover of~$M^k$, and since~$M^k$ is compact, there is a finite subcover~$\tilde{\mathcal{U}} \subset \mathcal{U}$ so that~$\tilde{\mathcal{U}}^k$ is a cover of~$M^k$. Hence the set~$\tilde{\mathcal{U}}$ fulfils property i) of being a~$k$-good cover, and as a subset of a~$k$-good cover, it also fulfils property ii).
\end{proof}
\subsection{The \v{C}ech-Bott-Segal double complex}
Finally, let us define a double complex which intertwines \v{C}ech complexes with the Che\-val\-ley\--Ei\-len\-berg complex structure. Analyzing this double complex will provide us with spectral sequences that calculate the Gelfand--Fuks cohomology of~$M$. We name this double complex after Bott and Segal, in reference to their spirtually similar local-to-global analysis in \cite{bott1977cohomology}, though we emphasize that our double complex differs from theirs. 
\begin{definition}
Let~$\mathcal{U} := \{U_i\}_{i \in I}$ be an open cover of a smooth manifold~$M$, and~$k \geq 1$. For~$U_{i_1},\dots, U_{i_q} \in \mathcal{U}$, set
\begin{align}
U_{i_1 \dots i_q}
:=
U_{i_1} \cap \dots \cap U_{i_q}.
\end{align}
We define the \emph{$k$-th \v{C}ech-Bott-Segal (CBS) double complex} for the cover~$\mathcal{U}$ as the following double complex:
\begin{center}
\tikzset{every picture/.style={line width=0.75pt}} 
\begin{tikzpicture}[x=0.75pt,y=0.75pt,yscale=-1,xscale=1]
\draw (156,195.97) node    {$\bigoplus _{i} \Delta _{k} C^{1}(\mathfrak{X}( U_{i}))$};
\draw (324.5,195.97) node    {$\bigoplus _{i,j} \Delta _{k} C^{1}( \Gamma (\mathfrak{X}( U_{ij}))$};
\draw (452,195.97) node    {$\ \dotsc \ $};
\draw (452,143.47) node    {$\ \dotsc \ $};
\draw (156,40) node    {$\dotsc $};
\draw (324.5,40) node    {$\dotsc $};
\draw (156,143.47) node    {$\bigoplus _{i} \Delta _{k} C^{2}(\mathfrak{X}( U_{i}))$};
\draw (324.5,143.47) node    {$\bigoplus _{i,j} \Delta _{k} C^{2}( \Gamma (\mathfrak{X}( U_{ij}))$};
\draw (156,91.47) node    {$\bigoplus _{i} \Delta _{k} C^{3}(\mathfrak{X}( U_{i}))$};
\draw (324.5,91.47) node    {$\bigoplus _{i,j} \Delta _{k} C^{3}(\mathfrak{X}( U_{ij}))$};
\draw (452,91.47) node    {$\ \dotsc \ $};
\draw    (249,195.97) -- (221,195.97) ;
\draw [shift={(219,195.97)}, rotate = 360] [color={rgb, 255:red, 0; green, 0; blue, 0 }  ][line width=0.75]    (10.93,-3.29) .. controls (6.95,-1.4) and (3.31,-0.3) .. (0,0) .. controls (3.31,0.3) and (6.95,1.4) .. (10.93,3.29)   ;
\draw    (249,143.47) -- (221,143.47) ;
\draw [shift={(219,143.47)}, rotate = 360] [color={rgb, 255:red, 0; green, 0; blue, 0 }  ][line width=0.75]    (10.93,-3.29) .. controls (6.95,-1.4) and (3.31,-0.3) .. (0,0) .. controls (3.31,0.3) and (6.95,1.4) .. (10.93,3.29)   ;
\draw    (256.5,91.47) -- (221,91.47) ;
\draw [shift={(219,91.47)}, rotate = 360] [color={rgb, 255:red, 0; green, 0; blue, 0 }  ][line width=0.75]    (10.93,-3.29) .. controls (6.95,-1.4) and (3.31,-0.3) .. (0,0) .. controls (3.31,0.3) and (6.95,1.4) .. (10.93,3.29)   ;
\draw    (431.5,91.47) -- (394.5,91.47) ;
\draw [shift={(392.5,91.47)}, rotate = 360] [color={rgb, 255:red, 0; green, 0; blue, 0 }  ][line width=0.75]    (10.93,-3.29) .. controls (6.95,-1.4) and (3.31,-0.3) .. (0,0) .. controls (3.31,0.3) and (6.95,1.4) .. (10.93,3.29)   ;
\draw    (431.5,143.47) -- (402,143.47) ;
\draw [shift={(400,143.47)}, rotate = 360] [color={rgb, 255:red, 0; green, 0; blue, 0 }  ][line width=0.75]    (10.93,-3.29) .. controls (6.95,-1.4) and (3.31,-0.3) .. (0,0) .. controls (3.31,0.3) and (6.95,1.4) .. (10.93,3.29)   ;
\draw    (431.5,195.97) -- (402,195.97) ;
\draw [shift={(400,195.97)}, rotate = 360] [color={rgb, 255:red, 0; green, 0; blue, 0 }  ][line width=0.75]    (10.93,-3.29) .. controls (6.95,-1.4) and (3.31,-0.3) .. (0,0) .. controls (3.31,0.3) and (6.95,1.4) .. (10.93,3.29)   ;
\draw    (156,181.47) -- (156,159.47) ;
\draw [shift={(156,157.47)}, rotate = 90] [color={rgb, 255:red, 0; green, 0; blue, 0 }  ][line width=0.75]    (10.93,-3.29) .. controls (6.95,-1.4) and (3.31,-0.3) .. (0,0) .. controls (3.31,0.3) and (6.95,1.4) .. (10.93,3.29)   ;
\draw    (156,129.47) -- (156,107.97) ;
\draw [shift={(156,105.97)}, rotate = 90] [color={rgb, 255:red, 0; green, 0; blue, 0 }  ][line width=0.75]    (10.93,-3.29) .. controls (6.95,-1.4) and (3.31,-0.3) .. (0,0) .. controls (3.31,0.3) and (6.95,1.4) .. (10.93,3.29)   ;
\draw    (324.5,128.97) -- (324.5,107.97) ;
\draw [shift={(324.5,105.97)}, rotate = 90] [color={rgb, 255:red, 0; green, 0; blue, 0 }  ][line width=0.75]    (10.93,-3.29) .. controls (6.95,-1.4) and (3.31,-0.3) .. (0,0) .. controls (3.31,0.3) and (6.95,1.4) .. (10.93,3.29)   ;
\draw    (324.5,181.47) -- (324.5,159.97) ;
\draw [shift={(324.5,157.97)}, rotate = 90] [color={rgb, 255:red, 0; green, 0; blue, 0 }  ][line width=0.75]    (10.93,-3.29) .. controls (6.95,-1.4) and (3.31,-0.3) .. (0,0) .. controls (3.31,0.3) and (6.95,1.4) .. (10.93,3.29)   ;
\draw    (156,76.97) -- (156,54.5) ;
\draw [shift={(156,52.5)}, rotate = 90] [color={rgb, 255:red, 0; green, 0; blue, 0 }  ][line width=0.75]    (10.93,-3.29) .. controls (6.95,-1.4) and (3.31,-0.3) .. (0,0) .. controls (3.31,0.3) and (6.95,1.4) .. (10.93,3.29)   ;
\draw    (324.5,76.97) -- (324.5,54.5) ;
\draw [shift={(324.5,52.5)}, rotate = 90] [color={rgb, 255:red, 0; green, 0; blue, 0 }  ][line width=0.75]    (10.93,-3.29) .. controls (6.95,-1.4) and (3.31,-0.3) .. (0,0) .. controls (3.31,0.3) and (6.95,1.4) .. (10.93,3.29)   ;

\end{tikzpicture}

\end{center}
The horizontal maps are given by the \v{C}ech differentials associated to the precosheaf structure, the vertical maps by the direct sum of Chevalley-Eilenberg differentials for the complexes $C^\bullet \left( \mathfrak{X} \left( U_{i_1 \dots i_p} \right) \right)$.
The grading is defined so that the term~$\bigoplus_{i_1,\dots,i_p} \Delta_k C^q(\mathfrak{X}(U_{i_1\dots i_p}))$ lies in degree $(p,q)$.
\\
The \emph{$k$-th skew-symmetrized CBS double complex} is the CBS double complex with all horizontal \v{C}ech complexes replaced by their skew-symmetrized versions, see Remark~\ref{rem:AntisymmetrizedCechComplex}.
\end{definition}
\begin{remark}
Many remarks on the form of this double complex are in order:
\begin{itemize}
\item[i)] Note that the zeroeth row of the CBS double complex is zero, and not, as one might expect, the \v{C}ech complex associated to $\Delta_k C^0(\mathfrak{X}(M))$. Since the zeroeth degree is connected to the rest of the complex by a zero differential, we do not lose any information by leaving it out.
\item[ii)] To deduce the ring structure on Gelfand--Fuks cohomology, it would be helpful if we could define a product structure on the CBS double complex in a way which extends the wedge product of cochains. At this point in time, it is unclear to the author how to accomplish this.
\item[iii)] The CBS double complex is not a first-quadrant double complex: It mixes a cohomological and a homological differential.  
A priori, this means there is an ambiguity in defining the associated total complex, given by the choice of taking either direct sums or direct products on the relevant diagonals, since there may now be infinitely many nonzero terms on each such diagonal. The usual convergence theorems for the spectral sequences arising from horizontal and vertical filtration will, in general, not apply. 
\end{itemize}
\end{remark}
Especially part iii) of the previous remark poses a significant problem. We borrow an argument from \cite{bott1977cohomology} to circument this:
\begin{lemma}
\label{lem:NormalizedCechIsBounded}
Let $k \in \mathbb{N}$ and $\mathcal{U}$ a finite open cover of a smooth manifold $M$. The $k$-th skew-symmetrized CBS double complex associated to $\mathcal{U}$ has only finitely many nonzero columns. In particular, it is bounded as a double complex.
\end{lemma}
\begin{proof}
By finiteness of $\mathcal{U}$, there is a largest $n$ such that there is a nonempty intersection $U_1 \cap \dots \cap U_n$ with $U_i \neq U_j$ for $i \neq j$. Hence all columns in degree $> n$ vanish in the skew-symmetrized double complex.
This concludes the proof.
\end{proof}
We begin with horizontal cohomology. 
\begin{proposition}
\label{PropositionCechIsomorphismSchwartzKernel}
Let $k,q \geq 1$ be integers and $\mathcal{U} = \{U_i\}_{i \in I}$ a $k$-good cover of a smooth manifold $M$. 
The \v{C}ech complex
\begin{align}
\bigoplus_i \Delta_k B^q(\mathfrak{X}(U_i)) \leftarrow \bigoplus_{i,j}  \Delta_k B^q(\mathfrak{X}(U_{ij})) \leftarrow \dots
\end{align}
is isomorphic to the \v{C}ech complex associated to the flabby cosheaf $\mathcal{B}^q_k$ on $M^q_k$ defined in Lemma~\ref{LemmaFineSheafOfDiagonalCochains}) with respect to the open cover $\{ U^q_k \subset M^q_k : U \in \mathcal{U} \}$ of $M^q_k$.
\\
The same statement holds for the skew-symmetrized \v{C}ech complex, and the isomorphism is equivariant under the natural permutation action of the symmetric group~$\Sigma_q$.
\end{proposition}
\begin{proof}
Note first that $\mathcal{U}^q_k$ is a cover of $M^q_k$ by Lemma~\ref{LemmaPowersOfKCoverForPowersOfM}.
The (restriction of the) Schwartz kernel maps (\ref{eq:SchwartzKernel}) give us a family of isomorphisms $\{\phi_U : U \subset M \text{ open}\}$ as in Lemma~\ref{lem:DiagonalCochainsSchwartzKernel}, so that for all open $U \subset V$ the following diagram commutes:
\begin{center}
\begin{tikzcd}
\Delta_k B^q(\mathfrak{X}(U)) 
\arrow{r} 
\arrow{d}{\phi_U}
& \Delta_k B^q(\mathfrak{X}(V))  
\arrow{d}{\phi_V}\\
\mathcal{B}^q_k(U^q_k) 
\arrow{r}
& 
\mathcal{B}^q_k(V^q_k)
\end{tikzcd}
\end{center}
Hence, we have isomorphisms on the precosheaf data; this lifts to an isomorphism of the two \v{C}ech complexes.
This argument is independent of the choice of the standard or the skew-symmetrized \v{C}ech complex. Since the sets $U^q_k$ are invariant under the natural $\Sigma_q$-action on $M^q$, both of the terms $\Delta_k B^q(\mathfrak{X}(U)) $ and $\mathcal{B}^q_k(U^q_k)$ admit a $\Sigma_k$-action by permutation of vector fields. The Schwartz kernel map is equivariant with respect to this permutation, as one finds from the explicit formula of its dual map~\ref{eq:SchwartzKernelDualMap}.
\end{proof}
\begin{theorem}
\label{TheoremFirstSpectralSequenceCollapses}
Let $q, k \geq 1$ be integers, and consider the $k$-th (skew-symmetrized) CBS double complex for a $k$-good cover $\mathcal{U}$ of $M$.
The cohomology of the $q$-th row is equal to $\Delta_k C^q (\mathfrak{X}(M))$ in degree zero, and trivial in all other degrees.
\end{theorem}
\begin{proof}
By Proposition~\ref{PropositionCechIsomorphismSchwartzKernel}, the \v{C}ech complex in this row, associated to the cover~$\mathcal{U}$ and the presheaf $U \mapsto \Delta_k B^q(\mathfrak{X}(U))$ over $M$, has the same homology as the \v{C}ech complex of the flabby cosheaf $U \mapsto \mathcal{B}^q_k(U)$ over $M^q_k$ with respect to the cover $\mathcal{U}^q_k$. Flabby cosheaves have trivial \v{C}ech homology independent of the chosen cover by Proposition~\ref{prop:FlabbyCosheavesTrivialHomology}, hence the homology is equal to $\Delta_k B^q(\mathfrak{X}(M))$ in zeroeth degree and zero in higher degree. 
\\
The isomorphism identifying the two complexes is equivariant with respect to the $\Sigma_q$-action on both spaces. The functor taking the complexes to its $\Sigma_q$-invariants is exact, as it arises from the action of a finite group in characteristic zero. 
Hence, the skew-symmetrized complex also has trivial cohomology in nonzero degree, and in degree zero $\left( \Delta_k B^q(\mathfrak{X}(M)) \right)^{\Sigma_q} = \Delta_k C^q(\mathfrak{X}(M))$. Since the skew-symmetrized complex is exactly the $q$-th row of the skew-symmetrized CBS complex, this concludes the proof.
\end{proof}
\begin{corollary}
\label{cor:ConvergenceOfSpectralSequence}
Let~$k \geq 1$ and assume there exists a finite,~$k$-good cover~$\mathcal{U}$ of~$M$, e.g. when~$M$ is compact. The spectral sequences~$\{E^{p,q}_r, d_r\}$ associated to the skew\--sym\-me\-trized~$k$-th CBS double complex for~$\mathcal{U}$ by filtering horizontally or vertically converges to~$\Delta_k \tilde{H}^\bullet(\mathfrak{X}(M))$. The tilde denotes reduced cohomology, cf. Definition~~\ref{def:ReducedComplex}.
\end{corollary}
\begin{proof}
By Theorem~\ref{TheoremFirstSpectralSequenceCollapses}, filtering by rows makes the spectral sequence collapse on the second page, with the indicated limit term~$\Delta_k \tilde{H}^\bullet(\mathfrak{X}(M))$. Due to finiteness of~$\mathcal{U}$ and Lemma~\ref{lem:NormalizedCechIsBounded}, the skew-symmetrized double complex has bounded rows, and for such double complexes both filtrations yield spectral sequences which converge to the same cohomology, see \cite[Chapter XV]{cartan2016homological}.
This shows the statement.
\end{proof}
\subsection{Spectral sequences for diagonal cohomology}
\begin{figure}[t]
\centering
\tikzset{
pattern size/.store in=\mcSize, 
pattern size = 5pt,
pattern thickness/.store in=\mcThickness, 
pattern thickness = 0.3pt,
pattern radius/.store in=\mcRadius, 
pattern radius = 1pt}
\makeatletter
\pgfutil@ifundefined{pgf@pattern@name@_7ax8lpbrf}{
\pgfdeclarepatternformonly[\mcThickness,\mcSize]{_7ax8lpbrf}
{\pgfqpoint{0pt}{0pt}}
{\pgfpoint{\mcSize}{\mcSize}}
{\pgfpoint{\mcSize}{\mcSize}}
{
\pgfsetcolor{\tikz@pattern@color}
\pgfsetlinewidth{\mcThickness}
\pgfpathmoveto{\pgfqpoint{0pt}{\mcSize}}
\pgfpathlineto{\pgfpoint{\mcSize+\mcThickness}{-\mcThickness}}
\pgfpathmoveto{\pgfqpoint{0pt}{0pt}}
\pgfpathlineto{\pgfpoint{\mcSize+\mcThickness}{\mcSize+\mcThickness}}
\pgfusepath{stroke}
}}
\makeatother
 
\tikzset{
pattern size/.store in=\mcSize, 
pattern size = 5pt,
pattern thickness/.store in=\mcThickness, 
pattern thickness = 0.3pt,
pattern radius/.store in=\mcRadius, 
pattern radius = 1pt}
\makeatletter
\pgfutil@ifundefined{pgf@pattern@name@_a1swgzc23}{
\pgfdeclarepatternformonly[\mcThickness,\mcSize]{_a1swgzc23}
{\pgfqpoint{0pt}{0pt}}
{\pgfpoint{\mcSize}{\mcSize}}
{\pgfpoint{\mcSize}{\mcSize}}
{
\pgfsetcolor{\tikz@pattern@color}
\pgfsetlinewidth{\mcThickness}
\pgfpathmoveto{\pgfqpoint{0pt}{\mcSize}}
\pgfpathlineto{\pgfpoint{\mcSize+\mcThickness}{-\mcThickness}}
\pgfpathmoveto{\pgfqpoint{0pt}{0pt}}
\pgfpathlineto{\pgfpoint{\mcSize+\mcThickness}{\mcSize+\mcThickness}}
\pgfusepath{stroke}
}}
\makeatother
 
\tikzset{
pattern size/.store in=\mcSize, 
pattern size = 5pt,
pattern thickness/.store in=\mcThickness, 
pattern thickness = 0.3pt,
pattern radius/.store in=\mcRadius, 
pattern radius = 1pt}
\makeatletter
\pgfutil@ifundefined{pgf@pattern@name@_ah0jjjfdl}{
\pgfdeclarepatternformonly[\mcThickness,\mcSize]{_ah0jjjfdl}
{\pgfqpoint{0pt}{0pt}}
{\pgfpoint{\mcSize}{\mcSize}}
{\pgfpoint{\mcSize}{\mcSize}}
{
\pgfsetcolor{\tikz@pattern@color}
\pgfsetlinewidth{\mcThickness}
\pgfpathmoveto{\pgfqpoint{0pt}{\mcSize}}
\pgfpathlineto{\pgfpoint{\mcSize+\mcThickness}{-\mcThickness}}
\pgfpathmoveto{\pgfqpoint{0pt}{0pt}}
\pgfpathlineto{\pgfpoint{\mcSize+\mcThickness}{\mcSize+\mcThickness}}
\pgfusepath{stroke}
}}
\makeatother
 
\tikzset{
pattern size/.store in=\mcSize, 
pattern size = 5pt,
pattern thickness/.store in=\mcThickness, 
pattern thickness = 0.3pt,
pattern radius/.store in=\mcRadius, 
pattern radius = 1pt}
\makeatletter
\pgfutil@ifundefined{pgf@pattern@name@_f6qfgglj4}{
\pgfdeclarepatternformonly[\mcThickness,\mcSize]{_f6qfgglj4}
{\pgfqpoint{0pt}{0pt}}
{\pgfpoint{\mcSize}{\mcSize}}
{\pgfpoint{\mcSize}{\mcSize}}
{
\pgfsetcolor{\tikz@pattern@color}
\pgfsetlinewidth{\mcThickness}
\pgfpathmoveto{\pgfqpoint{0pt}{\mcSize}}
\pgfpathlineto{\pgfpoint{\mcSize+\mcThickness}{-\mcThickness}}
\pgfpathmoveto{\pgfqpoint{0pt}{0pt}}
\pgfpathlineto{\pgfpoint{\mcSize+\mcThickness}{\mcSize+\mcThickness}}
\pgfusepath{stroke}
}}
\makeatother
 
\tikzset{
pattern size/.store in=\mcSize, 
pattern size = 5pt,
pattern thickness/.store in=\mcThickness, 
pattern thickness = 0.3pt,
pattern radius/.store in=\mcRadius, 
pattern radius = 1pt}
\makeatletter
\pgfutil@ifundefined{pgf@pattern@name@_h3gipah8g}{
\pgfdeclarepatternformonly[\mcThickness,\mcSize]{_h3gipah8g}
{\pgfqpoint{0pt}{0pt}}
{\pgfpoint{\mcSize}{\mcSize}}
{\pgfpoint{\mcSize}{\mcSize}}
{
\pgfsetcolor{\tikz@pattern@color}
\pgfsetlinewidth{\mcThickness}
\pgfpathmoveto{\pgfqpoint{0pt}{\mcSize}}
\pgfpathlineto{\pgfpoint{\mcSize+\mcThickness}{-\mcThickness}}
\pgfpathmoveto{\pgfqpoint{0pt}{0pt}}
\pgfpathlineto{\pgfpoint{\mcSize+\mcThickness}{\mcSize+\mcThickness}}
\pgfusepath{stroke}
}}
\makeatother
 
\tikzset{
pattern size/.store in=\mcSize, 
pattern size = 5pt,
pattern thickness/.store in=\mcThickness, 
pattern thickness = 0.3pt,
pattern radius/.store in=\mcRadius, 
pattern radius = 1pt}
\makeatletter
\pgfutil@ifundefined{pgf@pattern@name@_4ccbxufcd}{
\pgfdeclarepatternformonly[\mcThickness,\mcSize]{_4ccbxufcd}
{\pgfqpoint{0pt}{0pt}}
{\pgfpoint{\mcSize}{\mcSize}}
{\pgfpoint{\mcSize}{\mcSize}}
{
\pgfsetcolor{\tikz@pattern@color}
\pgfsetlinewidth{\mcThickness}
\pgfpathmoveto{\pgfqpoint{0pt}{\mcSize}}
\pgfpathlineto{\pgfpoint{\mcSize+\mcThickness}{-\mcThickness}}
\pgfpathmoveto{\pgfqpoint{0pt}{0pt}}
\pgfpathlineto{\pgfpoint{\mcSize+\mcThickness}{\mcSize+\mcThickness}}
\pgfusepath{stroke}
}}
\makeatother
 
\tikzset{
pattern size/.store in=\mcSize, 
pattern size = 5pt,
pattern thickness/.store in=\mcThickness, 
pattern thickness = 0.3pt,
pattern radius/.store in=\mcRadius, 
pattern radius = 1pt}
\makeatletter
\pgfutil@ifundefined{pgf@pattern@name@_f929sfir4}{
\pgfdeclarepatternformonly[\mcThickness,\mcSize]{_f929sfir4}
{\pgfqpoint{0pt}{0pt}}
{\pgfpoint{\mcSize}{\mcSize}}
{\pgfpoint{\mcSize}{\mcSize}}
{
\pgfsetcolor{\tikz@pattern@color}
\pgfsetlinewidth{\mcThickness}
\pgfpathmoveto{\pgfqpoint{0pt}{\mcSize}}
\pgfpathlineto{\pgfpoint{\mcSize+\mcThickness}{-\mcThickness}}
\pgfpathmoveto{\pgfqpoint{0pt}{0pt}}
\pgfpathlineto{\pgfpoint{\mcSize+\mcThickness}{\mcSize+\mcThickness}}
\pgfusepath{stroke}
}}
\makeatother
 
\tikzset{
pattern size/.store in=\mcSize, 
pattern size = 5pt,
pattern thickness/.store in=\mcThickness, 
pattern thickness = 0.3pt,
pattern radius/.store in=\mcRadius, 
pattern radius = 1pt}
\makeatletter
\pgfutil@ifundefined{pgf@pattern@name@_nwf2vbnco}{
\pgfdeclarepatternformonly[\mcThickness,\mcSize]{_nwf2vbnco}
{\pgfqpoint{0pt}{0pt}}
{\pgfpoint{\mcSize}{\mcSize}}
{\pgfpoint{\mcSize}{\mcSize}}
{
\pgfsetcolor{\tikz@pattern@color}
\pgfsetlinewidth{\mcThickness}
\pgfpathmoveto{\pgfqpoint{0pt}{\mcSize}}
\pgfpathlineto{\pgfpoint{\mcSize+\mcThickness}{-\mcThickness}}
\pgfpathmoveto{\pgfqpoint{0pt}{0pt}}
\pgfpathlineto{\pgfpoint{\mcSize+\mcThickness}{\mcSize+\mcThickness}}
\pgfusepath{stroke}
}}
\makeatother
 
\tikzset{
pattern size/.store in=\mcSize, 
pattern size = 5pt,
pattern thickness/.store in=\mcThickness, 
pattern thickness = 0.3pt,
pattern radius/.store in=\mcRadius, 
pattern radius = 1pt}
\makeatletter
\pgfutil@ifundefined{pgf@pattern@name@_wv3vm78ax}{
\pgfdeclarepatternformonly[\mcThickness,\mcSize]{_wv3vm78ax}
{\pgfqpoint{0pt}{0pt}}
{\pgfpoint{\mcSize}{\mcSize}}
{\pgfpoint{\mcSize}{\mcSize}}
{
\pgfsetcolor{\tikz@pattern@color}
\pgfsetlinewidth{\mcThickness}
\pgfpathmoveto{\pgfqpoint{0pt}{\mcSize}}
\pgfpathlineto{\pgfpoint{\mcSize+\mcThickness}{-\mcThickness}}
\pgfpathmoveto{\pgfqpoint{0pt}{0pt}}
\pgfpathlineto{\pgfpoint{\mcSize+\mcThickness}{\mcSize+\mcThickness}}
\pgfusepath{stroke}
}}
\makeatother
\tikzset{every picture/.style={line width=0.75pt}} 
\begin{tikzpicture}[x=0.75pt,y=0.75pt,yscale=-1,xscale=1]
\draw  (363.11,116.56) -- (544.27,116.56)(450.27,22) -- (450.27,207.95) (537.27,111.56) -- (544.27,116.56) -- (537.27,121.56) (445.27,29) -- (450.27,22) -- (455.27,29)  ;
\draw    (166.25,116.41) -- (347.8,116.41) ;
\draw [shift={(349.8,116.41)}, rotate = 180] [color={rgb, 255:red, 0; green, 0; blue, 0 }  ][line width=0.75]    (10.93,-3.29) .. controls (6.95,-1.4) and (3.31,-0.3) .. (0,0) .. controls (3.31,0.3) and (6.95,1.4) .. (10.93,3.29)   ;
\draw  [fill={rgb, 255:red, 0; green, 0; blue, 0 }  ,fill opacity=1 ] (265.22,114.5) -- (332.31,114.5) -- (332.31,118.33) -- (265.22,118.33) -- cycle ;
\draw  [fill={rgb, 255:red, 0; green, 0; blue, 0 }  ,fill opacity=1 ] (170.08,114.02) -- (199.32,114.02) -- (199.32,118.33) -- (170.08,118.33) -- cycle ;
\draw    (250.6,113.22) -- (250.6,119.93) ;
\draw  [fill={rgb, 255:red, 0; green, 0; blue, 0 }  ,fill opacity=1 ] (208.42,114.02) -- (237.66,114.02) -- (237.66,118.33) -- (208.42,118.33) -- cycle ;
\draw  [fill={rgb, 255:red, 0; green, 0; blue, 0 }  ,fill opacity=1 ] (464.47,114.02) -- (531.57,114.02) -- (531.57,117.85) -- (464.47,117.85) -- cycle ;
\draw  [fill={rgb, 255:red, 0; green, 0; blue, 0 }  ,fill opacity=1 ] (369.34,114.5) -- (398.58,114.5) -- (398.58,118.81) -- (369.34,118.81) -- cycle ;
\draw  [fill={rgb, 255:red, 0; green, 0; blue, 0 }  ,fill opacity=1 ] (406.72,114.5) -- (435.96,114.5) -- (435.96,118.81) -- (406.72,118.81) -- cycle ;
\draw  [fill={rgb, 255:red, 0; green, 0; blue, 0 }  ,fill opacity=1 ] (448.06,101.44) -- (448.06,34.34) -- (451.89,34.34) -- (451.89,101.44) -- cycle ;
\draw  [fill={rgb, 255:red, 0; green, 0; blue, 0 }  ,fill opacity=1 ] (448.54,197.53) -- (448.54,168.29) -- (452.85,168.29) -- (452.85,197.53) -- cycle ;
\draw  [fill={rgb, 255:red, 0; green, 0; blue, 0 }  ,fill opacity=1 ] (448.54,160.14) -- (448.54,130.91) -- (452.85,130.91) -- (452.85,160.14) -- cycle ;
\draw  [pattern=_7ax8lpbrf,pattern size=6pt,pattern thickness=0.75pt,pattern radius=0pt, pattern color={rgb, 255:red, 218; green, 218; blue, 218}] (464.71,101.44) -- (464.71,34.34) -- (531.81,34.34) -- (531.81,101.44) -- cycle ;
\draw  [pattern=_a1swgzc23,pattern size=6pt,pattern thickness=0.75pt,pattern radius=0pt, pattern color={rgb, 255:red, 218; green, 218; blue, 218}] (464.62,197.53) -- (464.62,168.29) -- (532.77,168.29) -- (532.77,197.53) -- cycle ;
\draw  [pattern=_ah0jjjfdl,pattern size=6pt,pattern thickness=0.75pt,pattern radius=0pt, pattern color={rgb, 255:red, 218; green, 218; blue, 218}] (464.62,160.14) -- (464.62,130.91) -- (532.77,130.91) -- (532.77,160.14) -- cycle ;
\draw  [pattern=_f6qfgglj4,pattern size=6pt,pattern thickness=0.75pt,pattern radius=0pt, pattern color={rgb, 255:red, 218; green, 218; blue, 218}] (407.2,101.44) -- (407.2,34.34) -- (435.96,34.34) -- (435.96,101.44) -- cycle ;
\draw  [pattern=_h3gipah8g,pattern size=6pt,pattern thickness=0.75pt,pattern radius=0pt, pattern color={rgb, 255:red, 218; green, 218; blue, 218}] (406.69,197.53) -- (406.69,168.29) -- (436.56,168.29) -- (436.56,197.53) -- cycle ;
\draw  [pattern=_4ccbxufcd,pattern size=6pt,pattern thickness=0.75pt,pattern radius=0pt, pattern color={rgb, 255:red, 218; green, 218; blue, 218}] (406.69,160.14) -- (406.69,130.91) -- (436.56,130.91) -- (436.56,160.14) -- cycle ;
\draw  [pattern=_f929sfir4,pattern size=6pt,pattern thickness=0.75pt,pattern radius=0pt, pattern color={rgb, 255:red, 218; green, 218; blue, 218}] (369.82,101.44) -- (369.82,34.34) -- (398.58,34.34) -- (398.58,101.44) -- cycle ;
\draw  [pattern=_nwf2vbnco,pattern size=6pt,pattern thickness=0.75pt,pattern radius=0pt, pattern color={rgb, 255:red, 218; green, 218; blue, 218}] (369.31,197.53) -- (369.31,168.29) -- (399.18,168.29) -- (399.18,197.53) -- cycle ;
\draw  [pattern=_wv3vm78ax,pattern size=6pt,pattern thickness=0.75pt,pattern radius=0pt, pattern color={rgb, 255:red, 218; green, 218; blue, 218}] (369.31,160.14) -- (369.31,130.91) -- (399.18,130.91) -- (399.18,160.14) -- cycle ;
\draw  [dash pattern={on 0.84pt off 2.51pt}]  (362.33,204.12) -- (536.6,29.84) ;
\draw (246.54,119.89) node [anchor=north west][inner sep=0.75pt]   [align=left] {{\footnotesize 0}};
\draw (282.36,87.73) node [anchor=north west][inner sep=0.75pt]    {$A\subset \mathbb{R}$};
\draw (472.71,8.14) node [anchor=north west][inner sep=0.75pt]    {$A^{2} \subset \mathbb{R}^{2}$};
\end{tikzpicture}
\caption{An illustration of Lemma~\ref{LemmaConnectedComponentIntersectionsWithDiagonals}: the set $A$ has three connected components in $\mathbb{R}$, so its square $A^2 \subset \mathbb{R}^2$ has $9 = 3^2$, all arising by taking products of connected components of $A$. The products of a connected component of $A$ with itself are exactly the connected components of $A^2$ which intersect the diagonal in $\mathbb{R}^2$.}
\label{FigureProductOfComponents}
\end{figure}
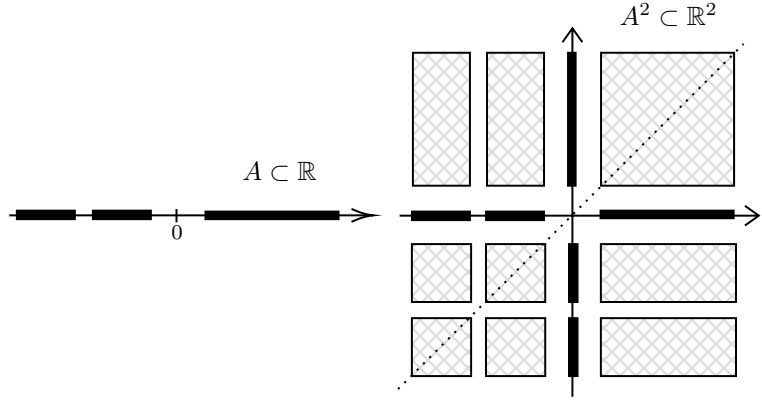
To arrive at Corollary~\ref{cor:ConvergenceOfSpectralSequence} we filtered the CBS double complex by rows, so now, let us study its filtration by rows. The cohomology among the arising vertical complexes amounts to calculating~$k$-diagonal Lie algebra cohomology of~$\mathfrak{X}(U)$, where the sets~$U$ are finite disjoint unions of~$\mathbb{R}^n$. 
We begin by showing that our methods of Section~\ref{SectionLocalCohomology} allow us to calculate~$k$-diagonal cohomology for such~$U$.
\begin{proposition}
\label{PropositionDiagonalLocalCohomologyIsJustLocal}
Let~$1 \leq r \leq k$ and~$U = \bigsqcup_{i=1}^r \mathbb{R}^n$. The inclusion
\begin{align}
\Delta_k C^\bullet(\mathfrak{X}(U))  \subset C^\bullet(\mathfrak{X}(U))
\end{align}
induces an isomorphism
\begin{align}
\Delta_k H^\bullet(\mathfrak{X}(U)) \cong H^\bullet(\mathfrak{X}(U)).
\end{align}
\end{proposition}
\begin{proof}
The construction in the proof of Proposition~\ref{PropositionCohomologyOfDisjointUnionsOfRns} restricts without change to the diagonally filtered complex. Specifically, the filtration~$F^{q} C^\bullet(\mathfrak{X}(\mathbb{R}^n))$ restricts to a filtration~$F^{q} (\Delta_k C^\bullet)(\mathfrak{X}(\mathbb{R}^n))$, it is straightforward to check that the exact sequence
\begin{equation}
 \begin{aligned}
0 \to 
F^{q+1} C^\bullet(\mathfrak{X}(\mathbb{R}^n))
&\to
F^{q} C^\bullet(\mathfrak{X}(\mathbb{R}^n))
\\
&\to
\bigoplus_{k_1 + \dots + k_r = q}
C^\bullet_{(k_1)}(W_n)
\otimes
\dots
\otimes
C^\bullet_{(k_r)}(W_n)
\to 
0
\end{aligned}
\end{equation}
restricts to an exact sequence
\begin{equation}
 \begin{aligned}
0 
\to 
F^{q+1} (\Delta_k C^\bullet)(\mathfrak{X}(\mathbb{R}^n))
&\to
F^{q} (\Delta_k C^\bullet)(\mathfrak{X}(\mathbb{R}^n))
\\
&\to
\bigoplus_{k_1 + \dots + k_r = q}
C^\bullet_{(k_1)}(W_n)
\otimes
\dots
\otimes
C^\bullet_{(k_r)}(W_n)
\to 
0,
 \end{aligned}
\end{equation}
and the image of the splitting~$\beta_0^{(r)} :
C^\bullet_{(0)}(W_n)^{\otimes^r}
\to
C^\bullet(\mathfrak{X}(M))$ from the proof of Proposition~\ref{PropositionCohomologyOfDisjointUnionsOfRns} is contained in 
$\Delta_k C^\bullet(\mathfrak{X}(M))$. Hence, 
\begin{align}
\Delta_k H^\bullet(\mathfrak{X}(U)) 
\cong 
H^\bullet(W_n)^{\otimes^k}
\cong
H^\bullet(\mathfrak{X}(U)),
\end{align}
so that all nontrivial cohomology classes in~$H^\bullet(\mathfrak{X}(U))$ have representatives contained in~$\Delta_k H^\bullet(\mathfrak{X}(U))$. This shows that the inclusion of complexes induces an isomorphism and the proposition is shown.
\end{proof}
Consider now the CBS double complex for some~$k$-good cover~$\mathcal{U}$, and the spectral sequence with respect to the filtration by columns. Every intersection in the cover~$\mathcal{U}$ is diffeomorphic to a disjoint union of at most~$k$ copies of~$\mathbb{R}^n$. Hence, Proposition~\ref{PropositionDiagonalLocalCohomologyIsJustLocal} applies in every column and we can replace diagonal cohomology with standard Lie algebra cohomology. Hence, the first page of the spectral sequence assumes the form in Figure \ref{fig:FirstPageOfCBSDoubleComplex}.
\begin{figure}
\centering
\tikzset{every picture/.style={line width=0.75pt}} 
\begin{tikzpicture}[x=0.75pt,y=0.75pt,yscale=-1,xscale=1]
\draw (126,195.97) node    {$\bigoplus _{i} H^{1}(\mathfrak{X}( U_{i}))$};
\draw (319.5,195.47) node    {$\bigoplus _{i,j} H^{1}( \Gamma (\mathfrak{X}( U_{ij}))$};
\draw (459,195.97) node    {$\ \dotsc \ $};
\draw (459,141.97) node    {$\ \dotsc \ $};
\draw (126,52) node    {$\dotsc $};
\draw (319.5,52) node    {$\dotsc $};
\draw (126,143.47) node    {$\bigoplus _{i} H^{2}(\mathfrak{X}( U_{i}))$};
\draw (319.5,143.47) node    {$\bigoplus _{i,j} H^{2}( \Gamma (\mathfrak{X}( U_{ij}))$};
\draw (126,90.97) node    {$\bigoplus _{i} H^{3}(\mathfrak{X}( U_{i}))$};
\draw (319.5,91.47) node    {$\bigoplus _{i,j} H^{3}(\mathfrak{X}( U_{ij}))$};
\draw (459,91.47) node    {$\ \dotsc \ $};
\draw    (233.5,195.69) -- (182,195.82) ;
\draw [shift={(180,195.83)}, rotate = 359.85] [color={rgb, 255:red, 0; green, 0; blue, 0 }  ][line width=0.75]    (10.93,-3.29) .. controls (6.95,-1.4) and (3.31,-0.3) .. (0,0) .. controls (3.31,0.3) and (6.95,1.4) .. (10.93,3.29)   ;
\draw    (233.5,143.47) -- (182,143.47) ;
\draw [shift={(180,143.47)}, rotate = 360] [color={rgb, 255:red, 0; green, 0; blue, 0 }  ][line width=0.75]    (10.93,-3.29) .. controls (6.95,-1.4) and (3.31,-0.3) .. (0,0) .. controls (3.31,0.3) and (6.95,1.4) .. (10.93,3.29)   ;
\draw    (241.5,91.27) -- (182,91.11) ;
\draw [shift={(180,91.11)}, rotate = 360.15] [color={rgb, 255:red, 0; green, 0; blue, 0 }  ][line width=0.75]    (10.93,-3.29) .. controls (6.95,-1.4) and (3.31,-0.3) .. (0,0) .. controls (3.31,0.3) and (6.95,1.4) .. (10.93,3.29)   ;
\draw    (438.5,91.47) -- (399.5,91.47) ;
\draw [shift={(397.5,91.47)}, rotate = 360] [color={rgb, 255:red, 0; green, 0; blue, 0 }  ][line width=0.75]    (10.93,-3.29) .. controls (6.95,-1.4) and (3.31,-0.3) .. (0,0) .. controls (3.31,0.3) and (6.95,1.4) .. (10.93,3.29)   ;
\draw    (438.5,142.19) -- (407.5,142.52) ;
\draw [shift={(405.5,142.54)}, rotate = 359.38] [color={rgb, 255:red, 0; green, 0; blue, 0 }  ][line width=0.75]    (10.93,-3.29) .. controls (6.95,-1.4) and (3.31,-0.3) .. (0,0) .. controls (3.31,0.3) and (6.95,1.4) .. (10.93,3.29)   ;
\draw    (438.5,195.89) -- (407.5,195.78) ;
\draw [shift={(405.5,195.77)}, rotate = 360.21000000000004] [color={rgb, 255:red, 0; green, 0; blue, 0 }  ][line width=0.75]    (10.93,-3.29) .. controls (6.95,-1.4) and (3.31,-0.3) .. (0,0) .. controls (3.31,0.3) and (6.95,1.4) .. (10.93,3.29)   ;
\end{tikzpicture}
\caption{The first page of the spectral sequence associated to the CBS double complex by beginning with taking the cohomology along the vertical, Chevalley-Eilenberg differential.}
\label{fig:FirstPageOfCBSDoubleComplex}
\end{figure}
We state the following simple lemma without proof (cf. Figure~\ref{FigureProductOfComponents}):
\begin{lemma}
\label{LemmaConnectedComponentIntersectionsWithDiagonals}
Let $X$ be a topological space.
The connected components of $X^q$ that do not intersect $X^q_{q-1}$ are exactly the Cartesian products of $q$ pairwise different connected components of $X$.
\end{lemma}
The following proposition makes use of \emph{relative} \v{C}ech homology of a pair of topological spaces $(X,A)$ with respect to a cover $\mathcal{U}$ of $X$, cf. \cite[Chapter IX]{eilenberg2015foundations}. In this situation, the space
\begin{align}
\mathcal{U}_A := \{U \cap A : U \in \mathcal{U}\}
\end{align}
is a cover of $A$, and the relative complex is the quotient by the inclusion of \v{C}ech complexes $\check{C}(\mathcal{U}_A) \to \check{C}(\mathcal{U}_X)$.
We denote the relative complex by $\check{C}_\bullet ( \mathcal{U}_X , \mathcal{U}_A )$ and its homology by $\check{H}_\bullet( \mathcal{U}_X, \mathcal{U}_A)$.
If $\mathcal{U}$ and $\mathcal{U}_A$ are good covers of $X$ and $A$, respectively, it is well known that the \v{C}ech homologies are isomorphic to singular homology\footnote{see \cite[Chapter VI.D, Theorem 4]{gunning2009analytic} for a proof in \v{C}ech cohomology. This easily dualizes to our setting.}
\begin{align}
\check{H}^\bullet(\mathcal{U}_X) \cong H^\bullet(X), \quad
\check{H}^\bullet(\mathcal{U}_A) \cong H^\bullet(A).
\end{align} 
Hence, by an argument on long exact sequences in homology, one has an isomorphism to relative singular homology.
\begin{align}
\check{H}_\bullet (\mathcal{U}_M, \mathcal{U}_A) 
\cong
\check{H}_\bullet( M, A)
\end{align}
\begin{proposition}
\label{PropositionSecondSpectralSequenceIsRelativeCechComplexes}
Let $q, k \geq 1$ be integers, $M$ be a smooth, orientable manifold, and $\mathcal{U}$  a $k$-good cover of $M$. Denote by $\{E^{p,q}_r, d_r\}$ the spectral sequence associated to the $k$-th CBS double complex with respect to $\mathcal{U}$, arising from the horizontal filtration by columns.
\\
The $q$-th row of the first page $E_1^{\bullet,\bullet}$ is naturally isomorphic to a direct sum of relative \v{C}ech complexes with respect to covers $(\mathcal{U}^r, \mathcal{U}^r_{r-1}$)  of $(M^r, M^r_{r-1})$ for $r = 1,\dots k$ (cf. Definition~\ref{def:OpenCoversOfDiagonals}):
\begin{equation}
\begin{aligned}
\label{eq:ReorganizedSecondPage}
E^{\bullet, q}_1 
\cong 
\, \, \, \, 
 &\check{C}_\bullet(\mathcal{U}_M) \otimes H^q(W_n) 
\\  
\,
\oplus
\, \, \, \, 
&
\left(
\bigoplus_{\substack{q_1 + q_2 = q \\ q_1,q_2 > 0}}
\check{C}_\bullet
(
\mathcal{U}^2_{M^2}, \mathcal{U}^2_{M^2_{1}}
)
\otimes
H^{q_1}(W_n) \otimes H^{q_2}(W_n) 
\right)^{\Sigma_2}
\\
\oplus \, \, \, \, &
\dots 
\\
\oplus
\, \, \, \, 
&
\left(
\bigoplus_{\substack{q_1 + \dots + q_k = q \\ q_1,\dots,q_k > 0}}
\check{C}_\bullet
(
\mathcal{U}^k_{M^k}, \mathcal{U}^k_{M^k_{k-1}}
)
\otimes
H^{q_1}(W_n) \otimes \dots \otimes H^{q_k}(W_n)
\right)^{\Sigma_k}. 
\end{aligned}
\end{equation}
Here, the symmetric groups $\Sigma_2,\dots,\Sigma_k$ act by simultaneous, skew-symmetric permutation of the Cartesian factors $U_1 \times \dots \times U_k$ of any set in the covers $(\mathcal{U}^k, \mathcal{U}^k_{k-1})$ and the tensor factors of~$H^{q_1}(W_n) \otimes \dots \otimes H^{q_k}(W_n)$.
\\
The same statement holds for the skew-symmetrized CBS complex, when the \v{C}ech complexes in \eqref{eq:ReorganizedSecondPage} are replaced by their skew-symmetrized versions.
\end{proposition} 
\begin{proof}
Set $n := \dim M$. Since $\mathcal{U}$ is $k$-good, all intersections of elements in $\mathcal{U}$ are diffeomorphic to a finite disjoint union of open balls in $\mathbb{R}^n$. Hence we can apply Corollary~\ref{cor:CosheafFunctorsEquivalentOnBalls} to find that the \v{C}ech complex in $q$-th row of $E_1^{\bullet,\bullet}$ is naturally isomorphic to the \v{C}ech complex of the $q$-th degree component of the constant precosheaf of algebras $U \mapsto H^\bullet(\mathfrak{X}(\mathbb{R}^n)) ^{\pi_0(U)}$, i.e.
\begin{align}
\label{eq:ProductConstantCosheaf}
U 
\mapsto 
\bigoplus_{q_1 + \dots + q_{\pi_0(U)} = q} 
H^{q_1}(\mathfrak{X}(\mathbb{R}^n)) 
\otimes \dots \otimes 
H^{q_{\pi_0(U)}}(\mathfrak{X}(\mathbb{R}^n)).
\end{align}
Denote this precosheaf by $H^q_M$. The multiplication of $H^\bullet(\mathfrak{X}(\mathbb{R}^n))$ is trivial on two elements of positive degree. Hence, if $H^q_M(U) \to H^q_M(V)$ is an extension map associated to an inclusion $U \subset V$, then the associated map
\begin{align}
H^{q_1}(\mathfrak{X}(\mathbb{R}^n)) 
\otimes \dots \otimes 
H^{q_{\pi_0(U)}}(\mathfrak{X}(\mathbb{R}^n))
\to
H^{q_1'}(\mathfrak{X}(\mathbb{R}^n)) 
\otimes \dots \otimes 
H^{q_{\pi_0(V)}'}(\mathfrak{X}(\mathbb{R}^n))
\end{align}
is only nonzero if
\begin{align}
| \{r :  q_r > 0, r = 1,\dots, \pi_0(U)\}| = | \{r : q_r' > 0, r = 1,\dots, \pi_0(V)\}|.
\end{align}
Hence the \v{C}ech complex associated to \eqref{eq:ProductConstantCosheaf} decomposes into a direct sum of complexes~$C_1, C_2,\dots$, where $C_r$ is defined as the subcomplex on which the number of tensor factors of nonzero degree in every term equals $r$. Since $\mathcal{U}$ is a $k$-good cover,~$C_r = 0$ if $r > k$, so
\begin{align}
\check{C}(\mathcal{U}_M, H^q_M) =
C_1 \oplus \dots \oplus C_k.
\end{align}
We are done if we can show that we have the following isomorphism of chain complexes:
\begin{align}
\label{eq:IdentificationSubcomplexWithBigCech}
C_r 
\cong
\left(
\bigoplus_{\substack{q_1 + \dots + q_r = q \\ q_1,\dots,q_r > 0}}
\check{C}_\bullet
(
\mathcal{U}^r_{M^r}, \mathcal{U}^r_{M^r_{r-1}}
)
\otimes
H^{q_1}(W_n) \otimes \dots \otimes H^{q_r}(W_n)
\right)^{\Sigma_r}.
\end{align}
Let now $U$ be an intersection of elements of $\mathcal{U}$, and write $U_1, \dots, U_s \subset U$ for the connected components of $U$, respectively.
If $s \geq r$, then there is a nontrivial direct summand in the complex $C_r$ associated to $U$, specifically
\begin{align}
\bigoplus_{
\substack{q_1 + \dots + q_{s} = q 
\\ 
| \{i :  q_i > 0, i = 1,\dots, s\}| = r}
}
&H^{q_1}(\mathfrak{X}(\mathbb{R}^n)) 
\otimes \dots \otimes 
H^{q_{s}}(\mathfrak{X}(\mathbb{R}^n))
\\
&
=
\bigoplus_{
\substack{q_1 + \dots + q_{r} = q 
\\ 
q_1,\dots,q_r > 0}
}
H^{q_1}(\mathfrak{X}(\mathbb{R}^n)) 
\otimes \dots \otimes 
H^{q_{r}}(\mathfrak{X}(\mathbb{R}^n)).
\end{align}
Every term $H^{q_1}(\mathfrak{X}(\mathbb{R}^n)) 
\otimes \dots \otimes 
H^{q_{r}}(\mathfrak{X}(\mathbb{R}^n))$ in the latter direct sum is associated to a a subset of $r$ pairwise different connected components $U_{i_1},\dots,U_{i_r}$ of $U$. 
The product $U_{i_{\sigma(1)}} \times \dots \times U_{i_{\sigma(r)}}$ is a connected component of the set $U^r$, and $U$ is an intersection of elements in $\mathcal{U}$. Hence $U_{i_{\sigma(1)}} \times \dots \times U_{i_{\sigma(r)}}$ is an intersection of elements of the cover $\mathcal{U}^r$, and we can write $c_{i_1 \dots i_r} \in \check{C}_\bullet (\mathcal{U}^r_{M^r})$ for the \v{C}ech simplex associated to it.\\
Now we can identify the term $H^{q_1}(\mathfrak{X}(\mathbb{R}^n)) 
\otimes \dots \otimes 
H^{q_{r}}(\mathfrak{X}(\mathbb{R}^n))$ with a direct summand of the right-hand side of \eqref{eq:IdentificationSubcomplexWithBigCech}: 
\begin{gather*}
H^{q_1}(\mathfrak{X}(\mathbb{R}^n)) 
\otimes \dots \otimes 
H^{q_{r}}(\mathfrak{X}(\mathbb{R}^n)
\\
\cong
\mathbb{R} c_{i_{1} \dots i_{r}}
\otimes
H^{q_1}(\mathfrak{X}(\mathbb{R}^n)) 
\otimes \dots \otimes 
H^{q_{r}}(\mathfrak{X}(\mathbb{R}^n)
\\
\cong
\left(
\bigoplus_{\sigma \in \Sigma_r} 
\mathbb{R} c_{i_{\sigma(1)} \dots i_{\sigma(r)}}
\otimes
H^{q_{\sigma(1)}}(\mathfrak{X}(\mathbb{R}^n)) 
\otimes \dots \otimes 
H^{q_{\sigma(r)}}(\mathfrak{X}(\mathbb{R}^n))
\right)^{\Sigma_r}.
\end{gather*}
Lemma~\ref{LemmaConnectedComponentIntersectionsWithDiagonals} shows that this isomorphism induces the isomorphism \eqref{eq:IdentificationSubcomplexWithBigCech}:\\
Firstly, the lemma implies that the Cartesian product of the pairwise different connected components $U_{i_1}, \dots, U_{i_r}$ does not intersect the diagonal $M^r_{r-1}$. Hence the \v{C}ech simplices $c_{i_1 \dots i_r}$ do not vanish in the relative \v{C}ech complex.\\
Secondly, the lemma implies that every direct summand in the relative \v{C}ech is associated to a product of $r$ pairwise different connected components. Thus, this construction exhausts the right-hand side of \eqref{eq:IdentificationSubcomplexWithBigCech}. \\
Hence the isomorphism \eqref{eq:IdentificationSubcomplexWithBigCech} holds on a level of vector spaces, and it is straightforward to see that this identification respects extension maps. Hence it is even an isomorphism of chain complexes, and the statement is shown.
\end{proof}
By the argument before Proposition~\ref{PropositionSecondSpectralSequenceIsRelativeCechComplexes}, the cohomology of the relative \v{C}ech complexes in the previous proposition is isomorphic to relative singular homology of the associated spaces. Together with Corollary~\ref{cor:ConvergenceOfSpectralSequence} and a degree reflection~$p \mapsto -p$ to bring the spectral sequence into a cohomological form, we end up with the following corollary:
\begin{corollary}
\label{cor:SecondPageOfSpectralSequence}
Let~$M$ be an orientable manifold which admits a finite,~$k$-good open cover (e.g. if~$M$ is compact).
There exists a cohomological spectral sequence~$\{E^{\bullet,\bullet}_r, d_r\}$ which converges to reduced~$k$-diagonal cohomology~$\Delta_k \tilde{H}^\bullet(\mathfrak{X}(M))$, and the entries~$E_2^{p,q}$ of its second page are, for~$q \geq 1$, of the following form:
\begin{equation}
\begin{aligned}
\label{eq:SecondPageOfSpectralSequence}
E^{p,q}_2 &\cong  H_{-p}(M) \otimes H^q(W_n) 
\\
&\oplus \;\, \,
\bigoplus_{\substack{q_1 + q_2 = q \\ q_i > 0}} \;\;
\left(
H_{-p}(M^2, M^2_1) \otimes H^{q_1}(W_n) \otimes H^{q_2}(W_n) 
\right)^{\Sigma_2}   
\\
&\oplus \quad \; \, \,  
\dots\\
&\oplus
\bigoplus_{\substack{q_1 + \dots +  q_k = q \\ q_i > 0}}
\left(
H_{-p}(M^k, M^k_{k-1}) \otimes H^{q_1}(W_n) \otimes \dots \otimes H^{q_k}(W_n) 
\right)^{\Sigma_k} . 
\end{aligned}
\end{equation}
Here, a permutation~$\sigma \in \Sigma_r$ acts by simultaneous permutation of the Cartesian factors of~$M^k$ and the tensor factors~$H^{q}(W_n)$.
\end{corollary} 

\begin{remark}
These spectral sequences differ from the ones stated in \cite{fuks1984cohomology}, but only insofar as \cite{fuks1984cohomology} considers the quotient complexes~$\Delta_k C^\bullet(\mathfrak{X}(\mathbb{R}^n)/\Delta_{k-1} C^\bullet(\mathfrak{X}(\mathbb{R}^n)$ rather than the diagonal complexes themselves. This essentially gives one spectral sequence for every row in (\ref{eq:SecondPageOfSpectralSequence}).
\end{remark} 
 
For~$k \geq q +1$, we have~$\Delta_k H^q(\mathfrak{X}(M)) = H^q(\mathfrak{X}(M))$ so in principle, these spectral sequences can be used to calculate the full Lie algebra cohomology of~$\mathfrak{X}(M)$, degree by degree.
In particular, since we know that the nontrivial cohomology of~$W_n$ is contained within the degrees~$q = 2n+1,\dots, 2n + n^2$ and the relative cohomology of~$(M^k, M^k_{k-1})$ in degrees~$\leq n k$, we have the following:
\begin{corollary}
For all smooth manifolds~$M$ that admits~$k$-good open covers for all~$k \in \mathbb{N}$, and all~$n \geq 0$, the Gelfand--Fuks cohomology~$H^n(\mathfrak{X}(M))$ is finite-dimensional. Further, if~$1 \leq k \leq \dim M$ then~$H^k(\mathfrak{X}(M)) = 0$.
\end{corollary}
\begin{example}
\begin{figure}
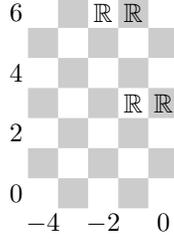

\centering
\begin{sseq}[grid=chess,labelstep=2]{-4...0}{0...6}
\ssmove 0 3
\ssdrop{\mathbb{R}}
\ssmove{-1}{0}
\ssdrop{\mathbb{R}}
\ssmove 0 3
\ssdrop{\mathbb{R}}
\ssmove{-1}{0}
\ssdrop{\mathbb{R}}
\end{sseq}
\caption{The spectral sequence for 2-diagonal Lie algebra cohomology for~$\mathfrak{X}(S^1)$. For the~$k$-diagonal spectral sequences for~$k \geq 2$, this pattern continues into the upper-left direction.}
\end{figure}
We sketch here how one now can calculate the well-known Gelfand--Fuks cohomology of $M = S^1$. By using excision and Poincaré duality\footnote{We thank Moishe Kohan for communicating to us a proof idea for this.}, we can find
\begin{align}
H_r \left( (S^1)^k, (S^1)^k_{k-1} \right) 
=  
\tilde{H}_r \left( (S^1)^k \setminus (S^1)^k_{k-1} \right) 
=
\begin{cases}
\mathbb{R}^{(k-1)!} &\text{ if } r = k, k-1,\\
0 &\text{ else,}
\end{cases}
\end{align}
where the copies of~$\mathbb{R}$ in~$\mathbb{R}^{(k-1)!}$ are enumerated by permutations of the~$(k-1)$-th symmetric group, and the invariant space under the action of the~$k$-th symmetric group~$\Sigma_k$ is one-dimensional.\footnote{Note that in \cite{gel1969cohomologyI} it is incorrectly claimed that the nontrivial degrees of $H_\bullet \left( (S^1)^k, (S^1)^k_{k-1} \right)$ are $k(k-1)/2$-dimensional. Already $(S^1)^3 \setminus (S^1)^3_2$ has only 2, not 3 connected components.}
Using this, we find that in the spectral sequence for~$k$-diagonal cohomology, there is only ever at most a single nontrivial term on every diagonal~$p + q = \text{const}$, and those only exist on the diagonals~$p + q = 0, 2,3, 5,6, 8,9,\dots$. From lacunary arguments, one concludes that all differentials beyond the second page must be trivial. 
\\
By retracing the construction of the spectral sequence, one can verify that the following cocycles~$c_2, c_3 \in C^\bullet(\mathfrak{X}(S^1))$ indeed are representatives of the nontrivial cohomology classes in degree 2 and 3: For~$x_0 \in S^1$ an arbitrary point,~$\partial_\phi$ the standard coordinate vector field on~$S^1$, and~$f,g \in C^\infty(S^1)$, we set
\begin{align}
c_2(f \partial_\phi ,g \partial_\phi) 
= 
\int_{S^1} \left( f(\phi) g'(\phi) - f'(\phi) g(\phi) \right) \dif \phi, \\
c_3(f \partial_\phi ,g \partial_\phi, h \partial_\phi)
= \det 
\begin{pmatrix}
f(x_0) & g(x_0) & h(x_0) \\
f'(x_0) & g'(x_0) & h'(x_0) \\
f''(x_0) & g''(x_0) & h''(x_0) \\
\end{pmatrix}.
\end{align}
We note that higher cohomology classes arise as the wedge products of these diagonal generators, but due to the lack of an obvious product on our spectral sequence, this is difficult to see directly. In any case, as a ring, we have
\begin{align}
H^\bullet(\mathfrak{X}(S^1)) \cong S^\bullet( \mathbb{R} c_2) \otimes \Lambda^\bullet(\mathbb{R} c_3).
\end{align}
\end{example}

\appendix
\section{Cosheaves and \v{C}ech homology}
\label{AppendixCosheaves}
In this appendix, we will recall the definition of (pre-)cosheaves, \v{C}ech homology of precosheaves and simple properties thereof from \cite{bredon2012sheaf}. 
For the remainder of this section, fix a topological space~$X$.
\begin{definition} \cite[Chapter V.1]{bredon2012sheaf}
\begin{itemize}
\item[i)]  A \emph{precosheaf} (of~$\mathbb{R}$-vector spaces)~$\mathcal{P}$ on~$X$ is a covariant functor from the category of open sets of~$X$, morphisms given by inclusions, into the category of~$\mathbb{R}$-vector spaces. Given an inclusion~$U \subset V$ of open sets, we denote the associated mapping~$\mathcal{P}(U) \to \mathcal{P}(V)$ by~$\iota_U^V$, called the \emph{extension map} from~$U$ to~$V$ of the precosheaf~$\mathcal{P}$.
\item[ii)]
A \emph{cosheaf} is a precosheaf~$\mathcal{P}$ with the property that for every open cover~$\mathcal{U} = \{U_i\}_{i \in I}$ of an open set~$U \subset X$, the sequence
\begin{align}
\bigoplus_{i,j} \mathcal{P}(U_i \cap U_j) \to 
\bigoplus_i \mathcal{P}(U_i) \to
\mathcal{P}(U)
\to 0
\end{align}
is exact, where the maps are given by
\begin{align}
(a_{ij})_{i,j} \mapsto 
\left( \sum_j
\iota_{U_i \cap U_j}^{U_i} (a_{ij} - a_{ji})
\right)_i, \quad
(b_i)_i \mapsto \sum_i \iota_{U_i}^U b_i.
\end{align}
\item[iii)] A \emph{morphism of (pre-)cosheaves} is a natural transformation between the functors defining the (pre-)cosheaves.
\end{itemize}
\end{definition}
Unless mentioned otherwise, all our (pre-)cosheaves take values in the category of~$\mathbb{R}$-vector spaces. 
\begin{definition}
Let~$S$ be a precosheaf over a topological space~$X$, and~$\mathcal{U} = \{U_\alpha\}$ an open cover of~$X$. 
Write~$\alpha := (\alpha_1,\dots,\alpha_{p+1})$ for the~$p$-simplex defined by a collection of indices~$\alpha_i$, and write
\begin{align}
U_\alpha := U_{\alpha_1} \cap \dots \cap U_{\alpha_{p+1}}.
\end{align}
Define for a~$p$-simplex~$\alpha$ and a number~$i \in \{1,\dots,p+1\}$ the~$(p-1)$-simplex~$\alpha^{(i)}$ as the simplex which arises from removing the~$i$-th index from~$\alpha$.
For all~$p \geq 0$, we define the space of \emph{\v{C}ech~$p$-chains for~$S$ associated to the cover~$\mathcal{U}$} as
\begin{align}
\check{C}_p(\mathcal{U}_X; S) 
:= 
\bigoplus_{\alpha = (\alpha_1,\dots,\alpha_{p+1})} S(U_\alpha).
\end{align}
We may then express elements~$c \in \check{C}_p(\mathcal{U}_X; S)$ as formal linear combinations
\begin{align}
\label{eq:CechCochainsAsFormalLinearCombinations}
c = 
\sum_{\alpha} c_\alpha, 
\quad 
c_\alpha \in S(U_{\alpha_1} \cap \dots \cap U_{\alpha_{p+1}}),
\end{align}
so that only finitely many~$c_\alpha$ are nonzero.\\
The \emph{\v{C}ech differential}~$\check{\partial} : \check{C}_p(\mathcal{U}_X; S) \to \check{C}_{p-1}(\mathcal{U}_X; S)$ via
\begin{align}
\check{\partial} 
(c_\alpha \alpha) 
:= 
\sum_{i=1}^{p+1} 
(-1)^{i - 1} 
\left( \iota_{U_{\alpha}}^{U_{\alpha^{(i)}}} c_\alpha \right)  \alpha^{(i)}.
\end{align}
The \emph{\v{C}ech homology associated to the cover~$\mathcal{U}$ and the precosheaf~$S$}, which is denoted by~$\check{H}_\bullet(\mathcal{U}_X; S)$, is defined as the homology of the chain complex~$\check{C}_\bullet(\mathcal{U}_X;P) := \bigoplus_{p \geq 0} \check{C}_p(\mathcal{U}_X; S)$.
\\
The \emph{\v{C}ech homology} of~$S$ is defined as
\begin{align}
\check{H}_\bullet(X; S) := \varprojlim \check{H}_\bullet(\mathcal{U}_X; S)
\end{align}
where the inverse limit is taken with respect to refinement of covers.
\end{definition}
\begin{remark}
\label{rem:AntisymmetrizedCechComplex}
The symmetric group~$\Sigma_p$ acts on multiindices~$\alpha$ of length~$p$ by permutation of the entries, and we denote this permutation by~$\sigma \cdot \alpha$.
Recall now the notation from (\ref{eq:CechCochainsAsFormalLinearCombinations}).  If~$c = \sum_\alpha c_\alpha \cdot \alpha$ and~$\alpha = (\alpha_1,\dots,\alpha_p)$ is one of the multiindices,  we call~$c$ \emph{skew-symmetric} if
\begin{align}
c_{\sigma \cdot \alpha} = \sign(\sigma) c_\alpha \quad \forall \alpha.
\end{align}
The \emph{skew-symmetrized} \v{C}ech complex is defined as the subcomplex 
\begin{align}
\check{C}_\bullet^a(\mathcal{U}_X; S) \subset \check{C}_\bullet(\mathcal{U}_X; S)
\end{align}
of skew-symmetric cochains.  
Dualizing the corresponding results for \v{C}ech cohomology of a sheaf \cite[Section 3.8]{godement1958topologie}, one finds that the inclusion~$\check{C}_\bullet^a(\mathcal{U}_X; S) \hookrightarrow \check{C}_\bullet(\mathcal{U}_X; S)
$ is a quasi-isomorphism.
\end{remark}
We want to remark on a special class of cosheaves on which \v{C}ech homology is trivial.
\begin{definition}
A cosheaf is called \emph{flabby} if all its extension maps are injective. 
\end{definition}
\begin{proposition}[\cite{bredon2012sheaf}, Chapter V, Prop. 1.6]
\label{PropositionSoftSheavesGiveFlabbyCosheaves}
Let~$P$ be a soft sheaf over a topological space~$X$. Then compactly supported sections of~$P$ admit the structure of a flabby cosheaf over~$X$, where the extension maps extend sections by zero.
\end{proposition}
\begin{proposition}[\cite{bredon2012sheaf}, Chapter VI, Cor. 4.5]
\label{prop:FlabbyCosheavesTrivialHomology}
Let~$S$ be a flabby cosheaf over~$X$ and~$\mathcal{U}$ an open cover of~$X$. Then
\begin{align}
\check{H}_p(\mathcal{U}_X; S) = 
\begin{cases}
S(X) & \text{ if } p = 0, \\
0 & \text{ else.}
\end{cases}
\end{align}
\end{proposition}

\section{The Hochschild-Serre Spectral Sequence for locally convex Lie algebras}
\label{AppendixHochschildSerre}
In the finite-dimensional setting, the Hochschild-Serre spectral sequence is standard and a proof is laid out in \cite[Chapter 1.5.1]{fuks1984cohomology} and \cite{hochschild1953cohomology}. For general locally convex Lie algebras and continuous cohomology, one generally needs a number of topological assumptions. For example, restriction maps of continuous cochains like~$C^q(\mathfrak{g}) \to C^r(\mathfrak{h},\Lambda^{q-r}(\mathfrak{g}/\mathfrak{h})^*)$ are not necessarily surjective if the subspace~$\frakh$ is not complemented. We formulate some assumptions which suffice for the setting in this paper:
\begin{theorem}
Let~$\mathfrak{g}$ be a complete, barreled, locally convex, nuclear Lie algebra whose strong dual space~$\frakg^*$ is complete,~$\mathfrak{h} \subset \mathfrak{g}$ a finite-dimensional Lie subalgebra, and~$A$ a complete, locally convex space on which~$\mathfrak{g}$ acts continuously. 
There is a cohomological spectral sequence~$\{E^{p,q}_r , d_r\}$ converging to continuous cohomology~$H^\bullet(\mathfrak{g})$ with
\begin{align}
E^{p,q}_1 = H^q \left(\mathfrak{h}, C^p\left( \mathfrak{g}/\mathfrak{h}, A \right) \right),
\end{align}
where~$C^p( X, Y)$ denotes skew-symmetric, jointly continuous, multilinear maps 
\begin{align}
\underbrace{X \times \dots \times X}_{p \text{ times }} \to Y,
\end{align}
and cohomology is taken with respect to continuous cochains.
This spectral sequence is contravariantly functorial, in the sense that a diagram of continuous Lie algebra morphisms
\begin{center}
\begin{tikzcd}
\mathfrak{h}
\arrow[hook]{r}
\arrow{d}
& 
\mathfrak{g} 
\arrow{d}\\
\tilde{\mathfrak{h}}
\arrow[hook]{r}
& 
\tilde{\mathfrak{g}}
\end{tikzcd}
\end{center}
induces linear maps 
\begin{align}
E_r^{p,q}(\tilde{\mathfrak{g}},\tilde{\mathfrak{h}}) 
\to 
E_r^{p,q}(\mathfrak{g},\mathfrak{h})
\end{align} 
compatible with the differentials for all~$p,q,r \geq 0$.
\end{theorem}
\begin{proof}
We define on the continuous cochains~$C^\bullet(\mathfrak{g})$ the filtration
\begin{align}
F^p C^{p+q}(\mathfrak{g} ; A) :=
\{c \in C^{p+q}(\mathfrak{g}, A) : c(X_1,\dots,X_{p+q}) = 0 \text{ when } X_1,\dots,X_{q+1} \in \mathfrak{h} \}.
\end{align}
This is an ascending filtration with
\begin{align}
C^r(\mathfrak{g},A) =
F^0 C^r(\mathfrak{g},A)
\supset
\dots
\supset 
F^r C^r(\mathfrak{g},A)
\supset 
F^{r+1} C^r(\mathfrak{g},A) = 0,
\end{align}
and
\begin{align}
d F^p C^{p + q}(\mathfrak{g} ; A) \subset F^p C^{p + q + 1}(\mathfrak{g} ; A).
\end{align}
Denote by~$\hat{\Lambda}^q$ the functor assigning to a locally convex vector space~$X$ the closure of the skew-symmetric tensors in its iterated projective tensor product~$X^{\hotimes^q}$, see for example \cite[Chapter III.7, IV.9]{schaefer1971locally}. 
We have a well-defined map
\begin{equation}
\begin{gathered}
F^p C^{p + q}(\mathfrak{g},A) 
\to  
L(
 \hat{\Lambda}^q\mathfrak{h} 
 \hotimes
 \hat{\Lambda}^p \mathfrak{g}/\mathfrak{h} , 
 A 
), 
\quad 
c \mapsto \tilde{c}, 
\\
\tilde{c}(
 (h_1 \wedge \dots \wedge h_q) 
 \otimes 
 [g_1] \wedge \dots \wedge [g_p])
 )
 :=
 c(h_1,\dots,h_q,g_1,\dots,g_p).
\end{gathered}
\end{equation}
This map is independent of the choices of representatives~$g_i$ by definition of the filtration and it is surjective because finite-dimensional subspaces are always complemented, so~$\mathfrak{g} \cong \mathfrak{h} \oplus \mathfrak{g}/\mathfrak{h}$ as a direct sum of locally convex vector spaces. The kernel of this map equals~$F^{p+1}C^{p+q}(\mathfrak{g},A)$. The image of this map is also indeed contained in the continuous linear maps by continuity of elements in the domain.
Since~$\mathfrak{h}$ is finite-dimensional, we trivially have
\begin{align}
(\mathfrak{h} \hotimes \mathfrak{g}/\mathfrak{h})^*
\cong 
\mathfrak{h}^*  \hotimes (\mathfrak{g}/\mathfrak{h})^*.
\end{align}
By the assumptions on~$\frakg$ and~$A$, we may apply \cite[Proposition 50.5]{treves1967topological} twice to find
\begin{align}
L(
 \hat{\Lambda}^q
 \mathfrak{h} 
  \hotimes
 \hat{\Lambda}^p 
 \mathfrak{g}/\mathfrak{h}, 
 A 
)
\cong
L(
 \hat{\Lambda}^q
 \mathfrak{h}, 
 L(
  \hat{\Lambda}^p 
  \mathfrak{g}/\mathfrak{h}, A
 )
)
\cong
C^q(\mathfrak{h}, 
L(
  \hat{\Lambda}^p \mathfrak{g}/\mathfrak{h}, A
 )
).
\end{align}
Hence we get an isomorphism of vector spaces
\begin{align}
F^p C^{p + q}(\mathfrak{g},A)
/
F^{p+1} C^{p + q}(\mathfrak{g},A) 
\cong 
C^q
\left(
\mathfrak{h},
L(\hat{\Lambda}^p \mathfrak{g}/\mathfrak{h}, 
A )
\right)
.
\end{align}
The differential of~$C^\bullet(\mathfrak{g},A)$ descends to the differential of this complex like in the purely algebraic case, so the spectral sequence associated to this filtration indeed has first page:
\begin{align}
E^{p,q}_1 = H^q \left(\mathfrak{h},L(\hat{\Lambda}^p \mathfrak{g}/\mathfrak{h}, A) \right).
\end{align}
The functoriality with respect to Lie algebra pairs~$(\mathfrak{g},\mathfrak{h})$ is analogous to the purely algebraic setting.
\end{proof}
\begin{remark}
This spectral sequence in the algebraic setting is generally also phrased with information about the second page if~$\mathfrak{h}$ is an ideal. Adapting this to the continuous setting would require stronger assumptions, since this in particular requires commuting the projective tensor product with the cohomology.
\end{remark}

\bibliographystyle{unsrt}
\bibliography{lit}
\end{document}